\documentclass[8pt]{article}
\usepackage[utf8]{inputenc}
\usepackage{amsmath}
\usepackage{amssymb}
\usepackage{mathrsfs}
\usepackage{amsthm}
\usepackage[breaklinks,colorlinks,dvipdfmx]{hyperref}
\usepackage{mathabx}
\usepackage{tikz-cd}
\usepackage{mathtools}
\usepackage{dsfont}
\usepackage[all]{xy}
\usepackage{stmaryrd}
\usepackage{url}

\newcommand{\id}{\mathrm{id}}
\newcommand\restr[2]{{
  \left.\kern-\nulldelimiterspace 
  #1 
  \littletaller 
  \right|_{#2} 
  }}

\newtheoremstyle{case}{}{}{}{}{}{:}{ }{}
\theoremstyle{case}
\newtheorem{case}{Case}

\newcommand{\littletaller}{\mathchoice{\vphantom{\big|}}{}{}{}}
\theoremstyle{plain}
\newtheorem{thm}{Theorem}[section]
\newtheorem{lem}[thm]{Lemma}
\newtheorem{prop}[thm]{Proposition}
\newtheorem{cor}[thm]{Corollary}

\theoremstyle{definition}

\newtheorem{defn}[thm]{Definition}

\theoremstyle{remark}
\newtheorem{Rem}[thm]{Remark}

\theoremstyle{plain}
\newtheorem*{thm*}{Theorem}
\newtheorem*{lem*}{Lemma}
\newtheorem*{prop*}{Proposition}
\newtheorem*{cor*}{Corollary}
\newtheorem*{conj*}{Conjecture}

\theoremstyle{definition}
\newtheorem*{ass*}{Assumption}
\newtheorem*{defn*}{Definition}
\newtheorem*{exmp}{Example}

\theoremstyle{remark}
\newtheorem*{Rem*}{Remark}
\newcommand{\compconj}[1]{%
  \overline{#1}%
}
\makeatletter
\newcommand*{\rom}[1]{\expandafter\@slowromancap\romannumeral #1@}
\makeatother

\DeclareMathOperator{\spn}{span}

\DeclareMathOperator{\op}{op}

\DeclareMathOperator{\ev}{ev}
\DeclareMathOperator{\Tr}{Tr}

\title{Examples of solvable and nilpotent finite quantum groups
}
\author{Gerard Glowacki\footnote{%
glowacki.gerard.rainier.u0@s.mail.nagoya-u.ac.jp}, Masamune Hattori\footnote{%
m21039e@math.nagoya-u.ac.jp} and Masato Tanaka\footnote{%
masato.tanaka.c7@math.nagoya-u.ac.jp%
}}
\date{November 2023}

\begin{document}

\maketitle
\begin{abstract}
We prove the solvability and nilpotency of Kac--Paljutkin's finite quantum group and Sekine quantum groups and we classify the solvable series of Kac--Paljutkin's finite quantum group via Cohen--Westreich's Burnside theorem. Some semisimple quasitriangular Hopf algebras of dimensions $2pq$ are also studied. In Appendix A, we give a direct computation of the universal $R$-matrices for Kac--Paljutkin's $8$-dimensional finite quantum group.
\end{abstract}
\section{Introduction}
A finite quantum group is a compact quantum group (in the sense of Woronowicz \cite{Wor}) whose `algebra of continuous functions' is finite dimensional. In other words, a finite quantum group is a finite dimensional Hopf $C^*$-algebra. The class of such objects is considered as a quantum analogue of finite groups. In fact, the algebra $C(G)$ of continuous functions on a finite group $G$ and the group algebra $\mathbb{C}[G]$ of a finite group $G$ (with their usual Hopf $\ast$-algebra structures) are finite quantum group and any commutative (resp.~cocommutative) finite quantum group is of the form $C(G)$ (resp.~$\mathbb{C}[G]$) for some finite group $G$.\par
To show that the class of finite groups is a proper subclass of finite quantum groups, a nontrivial (i.e.~noncommutative and noncocommutative) finite quantum group must be found. In $1966$, Kac--Paljutkin found a nontrivial $8$-dimensional finite quantum group (Kac--Paljutkin's finite quantum group). It is known that $8$ is the minimum of dimensions of nontrivial finite quantum groups. In $1996$, Sekine found a familyof nontrivial finite quantum groups indexed by $k\in\mathbb{Z}_{\geq3}$ (Sekine quantum groups). These examples of finite quantum groups are interesting objects and studied by many mathematicians. For example, category theoretical and representation theoretical aspects of Kac--Paljutkin's finite quantum group were studied by Tambara--Yamagami (\cite{TY}), the quasitriangular structure of Kac--Paljutkin's finite quantum group was detrmined by Suzuki and Wakui (\cite{Suz, W}), and probability theoretical and representation theoretical aspects of Kac--Paljutkin's finite quantum group and Sekine quantum groups were studied by McCarthy and Zhang (\cite{McC, Z}). Furthermore, by changing the ground field $\mathbb{C}$ of Kac--Paljutkin's finite quantum group and Sekine quantum groups, we have a non-Archimedean quantum group in the sense of Kochubei (\cite{Koc}).\par
In $2009$, Etingof--Nikshych--Ostrik formulated some group theoretical concepts of fusion categories including solvability and nilpotency (\cite{ENO}). As was pointed out in \cite[Proposition 4.5, Remark 4.6]{ENO}, nilpotency does not always imply solvability in this setting although any nilpotent group is always solvable. In $2016$, Cohen--Westreich defined solvability and nilpotency of semisimple Hopf algebras via integrals and developed the general theory(\cite{CW}). In this formlation, nilpotency implies solvability and the analogue of the celebrated $p^aq^b$ theorem by Burnside holds.\par 
In the present paper, we study examples of solvable and nilpotent finite quantum groups based on \cite{CW}. Our main results are as follows:
\begin{itemize}
\item Kac--Paljutkin's finite quantum group has only five solvable series of maximal length and at least one of them proves the nilpotency.
\item Sekine quantum groups $\mathcal{A}_{k}$ are nilpotent. In particular, Sekine quantum groups $\mathcal{A}_{15}$ is a counter example of the converse of Cohen--Westreich's Burnside theorem.
\item For distinct odd primes $p$ and $q$ such that $p-2$ is divided by $5$ and $q\neq5$, any quasi-triangular semisimple Hopf algebra of dimension $2pq$ has a nontrivial normal unital left coideal subalgebra.
\end{itemize}
By these results we have nontrivial examples of the general theory developed by Cohen--Westreich. In Appendix A, we compute the universal $R$-matrices of Kac--Paljutkin's $8$-dimensional finite quantum group. Of course this is not a new result. The quasitriangular structure of Kac--Paljutkin's finite quantum group was already studied and determined by Suzuki and Wakui (\cite{Suz, W}). In Appendix A, however, we give a direct computation without any difficult concepts and techniques. 
\\
\\
Acknowledgements: 
The second author was financially supported by JST SPRING, Grant Number
JPMJSP2125. The author M.~Hattori would like to take this opportunity to thank the
“Interdisciplinary Frontier Next-Generation Researcher Program of the Tokai
Higher Education and Research System.”
The third author would like to take this opportunity to thank the “Nagoya University 
Interdisciplinary Frontier Fellowship” supported by Nagoya University and JST, the 
establishment of university fellowships towards the creation of science technology innovation, 
Grant Number JPMJFS2120.
\section{Preliminaries}
The best general references here are \cite{CW}, \cite{FG} and \cite{K}.
\subsection{Notations and conventions}
\begin{itemize}
\item The symbol $\delta_{ij}$ denotes the Kronecker's delta, i.e.\ $\delta_{ij}=1$ if $i=j$ and $\delta_{ij}=0$ if $i\neq j$.
\item The vector spaces and algebras which we consider are the complex vector spaces unless otherwise stated.
\end{itemize}
\subsection{Finite quantum groups}
\begin{defn}
A ($\mathbb{C}$-)vector space $A$ is called a \emph{unital algebra} if $A$ has a bilinear and associative multiplication $A\times A\to A, (a,b)\mapsto ab$ and an element $1$ such that $a=1a=a1$ for all $a\in A$. A linear map $f\colon A\to B$ between two unital algebras $A$ and $B$ is called a \emph{homomorphism} if $f(ab)=f(a)f(b)$ for all $a,b\in A$ and $f(1)=1$. If a unital algebra $A$ has an operation $A\ni a\mapsto a^{\ast}\in A$ such that $a^{\ast\ast}=a, (\lambda a+b)^{\ast}=\compconj{\lambda}a^{\ast}+b^{\ast}, (ab)^{\ast}=b^{\ast}a^{\ast}$ for all $a,b\in A,\lambda\in\mathbb{C}$, we say $A$ is a \emph{unital $\ast$-algebra}.
\end{defn}
\begin{defn}
Let $A$ be a vector space. A function $\|\bullet\|\colon A\to [0,\infty)$ is called a \emph{norm} if the following three conditions are satisfied.
\begin{itemize}
\item $\|\alpha x\|=|\alpha|\|x\|$ for all $\alpha\in\mathbb{C}$ and $x\in\ A$.
\item $\|x+y\|\leq\|x\|+\|y\|$ for all $x, y\in A$
\item $\|x\|=0$ implies $x=0$ for all $x\in A$.
\end{itemize}
\end{defn}
\begin{exmp}
Let $\mathbb{M}_n$ denote the algebra of $n\times n$ complex matrices. We define $\|x\|\coloneqq\sup\{\|x\xi\|_{\mathbb{C}^n}\mid \|\xi\|_{\mathbb{C}^n}\leq1\}$, where we regard each $x\in\mathbb{M}_n$ as an operator on the Hilbert space $\mathbb{C}^n$ and $\|\xi\|_{\mathbb{C}^n}\coloneqq\sqrt{|\xi_1|^2+\cdots|\xi_n|^2},$ $(\xi=(\xi_1,\ldots,\xi_n))$. This defines a norm on $\mathbb{M}_n$ and is called the \emph{operator norm} on $\mathbb{M}_n$. Let $G$ be a finite set and let $C(G)$ denote the algebra of complex valued continuous functions on $G$. Put $\|f\|\coloneqq\sup\{|f(x)|\mid x\in G\}$ for $f\in C(G)$. This defines a norm on $C(G)$ and is called the \emph{sup norm}.
\end{exmp}
\begin{defn}
Let $A$ be a unital $\ast$-algebra. We call $A$ a \emph{unital $C^*$-algebra} if $A$ is equipped with a norm $\|\bullet\|$ satisfying $\|a^*a\|=\|a\|^2$ for all $a\in A$, $\|ab\|\leq\|a\|\|b\|$ for all $a,b\in A$ and $A$ is complete with respect to this norm.
\end{defn}
\begin{exmp}
We equip $\mathbb{M}_n$ with the operator norm, then $\mathbb{M}_n$ is a unital $C^*$-algebra. The algebra $C(G)$ of complex valued continuous functions on a finite set $G$ with the sup norm is also a unital $C^*$-algebra. 
\end{exmp}
\begin{defn}
Let $A$ be a unital $C^*$-algebra. A linear functional $f\colon A\to\mathbb{C}$ is called a \emph{state} if $f(a^*a)\geq0$ for all $a\in A$ and $f(1)=1$.
\end{defn}
\begin{exmp}
The normalized trace $(1/n)\Tr(\bullet)$ is a state on $\mathbb{M}_n$ and $\ev_g\colon f\mapsto f(g)$ is a state on $C(G)$, where $G$ is a finite set and $\ev_g$ denotes the evaluation at $g\in G$. 
\end{exmp}
\begin{Rem}
Of couse we can define a non-unital ($\ast$-, $C^{\ast}$-)algebra and a state for a non-unital $C^*$-algebra. In this thesis, however, we only consider the unital ones. In addition, although there are a lot of infinite dimensional examples we only consider the finite dimensional ones.
\end{Rem}
The following proposition is well-known. For details, see \cite{M} for example.
\begin{prop}
Any finite dimensioanl $C^*$-algebra is of the form $\bigoplus_{k=1}^N\mathbb{M}_{n_k}$ for some positive integers $N$ and $n_1,\ldots,n_N$.
\end{prop}
\begin{Rem}
The norm on $\bigoplus_{k=1}^N\mathbb{M}_{n_k}$ which makes $\bigoplus_{k=1}^N\mathbb{M}_{n_k}$ a $C^*$-algebra is given by $\|(x_1,\ldots,x_N)\|=\displaystyle\max_{1\leq k\leq N}\|x_k\|$.
\end{Rem}
\begin{defn}
A ($\ast$-)algebra $A$ is called a \emph{Hopf ($\ast$-)algebra} if $A$ is equipped with ($\ast$-preserving) homomorphisms $\Delta\colon A\to A\otimes A$ and $\epsilon\colon A\to\mathbb{C}$ and antimultiplicative linear map $S\colon A\to A$ such that $(\Delta\otimes\id)\Delta=(\id\otimes\Delta)\Delta$, $(\epsilon\otimes\id)\Delta=\id=(\id\otimes\epsilon)\Delta$ and $m(S\otimes\id)\Delta=\epsilon(\bullet)1=m(\id\otimes S)\Delta$.  If a finite dimensional Hopf $\ast$-algebra $A$ is also a $C^{\ast}$-algebra we say $A$ is a \emph{finite quantum group}.
\end{defn}
In particular any finite quantum group is a semisimple Hopf algebra.
\begin{exmp}
Let $G$ be a finite group. Let $\Delta\colon C(G)\to C(G\times G)\simeq C(G)\otimes C(G)$ be the linear map defined by $\Delta(f)(x,y)=f(xy)$ where $f\in C(G)$ and $x,y\in G$. Let $\epsilon=\ev_{e}$ where $e$ is the unit element of $G$. Let $S\colon C(G)\to C(G)$ be the linear map defined by $S(f)(x)=f(x^{-1})$ where $f\in C(G)$ and $x\in G$. Equipped with these maps, the algebra $C(G)$ is a finite quantum group which is commutative. Conversely any commutative finite quantum group is of this form. 
\end{exmp}
\begin{exmp}
Let $G$ be a finite group. Let $\mathbb{C}[G]$ denote the group algebra of $G$. We denote its elements by $\sum_{g\in G} a_g\delta_g$ where $a_g\in\mathbb{C}$. Let $\{e_g\mid g\in G\}$ be the canonical basis of $\mathbb{C}^{|G|}$ and let $\sum_{g\in G}a_g\delta_ge_h\coloneqq \sum_{g\in G}a_ge_{gh}$. In this way we can regard elements in $\mathbb{C}[G]$ as operators on the sapce $\mathbb{C}^{|G|}$ and we can embed $\mathbb{C}[G]$ in $\mathbb{M}_{|G|}(\mathbb{C})$. We equip $\mathbb{C}[G]$ with the operator norm on $\mathbb{M}_{|G|}(\mathbb{C})$. We define the linear map $\Delta\colon\mathbb{C}[G]\to\mathbb{C}[G]\otimes\mathbb{C}[G]$ as $\Delta(\delta_g)=\delta_g\otimes\delta_g$ where $g\in G$. Let $\epsilon\colon\mathbb{C}[G]\to\mathbb{C}$ be the linear map $\epsilon(\delta_g)=1$ and let $S\colon \mathbb{C}[G]\to\mathbb{C}[G]$ be the linear map $S(\delta_g)=\delta_{g^{-1}}$ where $g\in G$. Then $\mathbb{C}[G]$ with these maps is a finite quantum group, which is cocommutative meaning that $\Delta=\tau\circ\Delta$ where $\tau$ is the flip $x\otimes y\mapsto y\otimes x$. Conversely any cocommutative finite quantum group is of this form.
\end{exmp}
\begin{exmp}Kac--Paljutkin and Sekine found examples of non commutative and non cocommutative finite quantum groups. We review the definitions. For details see \cite{KP} and \cite{Sek}.\par
The \emph{Kac--Paljutkin's finite quantum group} is $\mathcal{A}=\mathbb{C}\oplus\mathbb{C}\oplus\mathbb{C}\oplus\mathbb{C}\oplus\mathbb{M}_2$ as a $C^{\ast}$-algebra. Its Hopf $\ast$-algebra structure is given by
\begin{align*}
\Delta(e_1) = e_1 &\otimes e_1 + e_2 \otimes e_2 + e_3 \otimes e_3 + e_4 \otimes e_4\\
+ \frac{1}{2} a_{11} \otimes a_{11} &+ \frac{1}{2} a_{12} \otimes a_{12} + \frac{1}{2} a_{21} \otimes a_{21} + \frac{1}{2} a_{22} \otimes a_{22}\\
\Delta(e_2) = e_1 &\otimes e_2 + e_2 \otimes e_1 + e_3 \otimes e_4 + e_4 \otimes e_3\\
 + \frac{1}{2} a_{11} \otimes a_{22} &+ \frac{1}{2} a_{22} \otimes a_{11} - \frac{\sqrt{-1}}{2} a_{12} \otimes a_{21} + \frac{\sqrt{-1}}{2} a_{21} \otimes a_{12}\\
\Delta(e_3) = e_1 &\otimes e_3 + e_3 \otimes e_1 + e_2 \otimes e_4 + e_4 \otimes e_2\\
+ \frac{1}{2} a_{11} \otimes a_{22} &+ \frac{1}{2} a_{22} \otimes a_{11} + \frac{\sqrt{-1}}{2} a_{12} \otimes a_{21} - \frac{\sqrt{-1}}{2} a_{21} \otimes a_{12}\\
\Delta(e_4) = e_1 &\otimes e_4 + e_4 \otimes e_1 + e_2 \otimes e_3 + e_3 \otimes e_2\\
 + \frac{1}{2} a_{11} \otimes a_{11}& + \frac{1}{2} a_{22} \otimes a_{22} - \frac{1}{2} a_{12} \otimes a_{12} - \frac{1}{2} a_{21} \otimes a_{21}\\
\Delta(a_{11}) = e_1 &\otimes a_{11} + a_{11} \otimes e_1 + e_2 \otimes a_{22} + a_{22} \otimes e_2\\
+e_3 \otimes a_{22} + &a_{22} \otimes e_3 + e_4 \otimes a_{11} + a_{11} \otimes e_4\\
\Delta(a_{12}) = e_1 &\otimes a_{12} + a_{12} \otimes e_1 + \sqrt{-1} e_2 \otimes a_{21} - \sqrt{-1} a_{21} \otimes e_2\\
 - \sqrt{-1} e_3 \otimes& a_{21} + \sqrt{-1} a_{21} \otimes e_3 - e_4 \otimes a_{12} - a_{12} \otimes e_4\\
\Delta(a_{21}) = e_1 &\otimes a_{21} + a_{21} \otimes e_1 - \sqrt{-1} e_2 \otimes a_{12} + \sqrt{-1} a_{12} \otimes e_2\\
 + \sqrt{-1} e_3 \otimes &a_{12} - \sqrt{-1} a_{12} \otimes e_3 - e_4 \otimes a_{21} - a_{21} \otimes e_4\\
\Delta(a_{22}) = e_1 &\otimes a_{22} + a_{22} \otimes e_1 + e_2 \otimes a_{11} + a_{11} \otimes e_2\\
+ e_3 \otimes a_{11} + &a_{11} \otimes e_3 + e_4 \otimes a_{22} + a_{22} \otimes e_4\\
\epsilon(e_1)&=1, \epsilon(e_2)=\epsilon(e_3)=\epsilon(e_4)=0\\
S(e_i)&=e_i\hspace{1mm}(i=1,2,3,4), S(a_{ij})=a_{ji}\hspace{1mm}(i,j=1,2)
\end{align*}
Here $e_i$ and $a_{ij}$ denote the following elements:
\begin{align*}
e_1&=1\oplus0\oplus0\oplus0\oplus\begin{bmatrix} 0 & 0 \\ 0 & 0 \end{bmatrix}, e_2=0\oplus1\oplus0\oplus0\oplus\begin{bmatrix} 0 & 0 \\ 0 & 0 \end{bmatrix}\\
e_3&=0\oplus0\oplus1\oplus0\oplus\begin{bmatrix} 0 & 0 \\ 0 & 0 \end{bmatrix}, e_4=0\oplus0\oplus0\oplus1\oplus\begin{bmatrix} 0 & 0 \\ 0 & 0 \end{bmatrix}\\
&\quad\quad a_{ij}=0\oplus0\oplus0\oplus0\oplus E_{ij}\quad(i,j=1,2)
\end{align*}
where the $E_{ij}$'s are matrix units.\par
Let $k$ be a positive integer. Let $\eta=\exp(2\pi\sqrt{-1}/k)$. The \emph{Sekine quantum group} $\mathcal{A}_k$ is a finite quantum group defined as follows. As a $C^*$-algebra $\mathcal{A}_k=(\bigoplus_{i,j\in\mathbb{Z}_k}\mathbb{C}d_{ij})\oplus\mathbb{M}_k$, where the $d_{ij}$'s are projections such that $d_{ij}d_{kl}=\delta_{ik}\delta_{jl}d_{ij}$. Its Hopf $\ast$-algebra structure is given by 
\begin{align*}
\Delta(d_{ij})=&\sum_{m,n\in\mathbb{Z}_k}d_{mn}\otimes d_{i-m, j-n}+\dfrac{1}{k}\sum_{m,n\in\mathbb{Z}_k}\eta^{i(m-n)}e_{mn}\otimes e_{m+j,m+j},\\
\Delta(e_{ij})=&\sum_{m,n\in\mathbb{Z}_k}\eta^{m(i-j)}d_{-m,-n}\otimes e_{i-n,j-n}+\sum_{m,n\in\mathbb{Z}_k}\eta^{m(j-i)}e_{i-n,j-n}\otimes d_{mn},\\
&\quad\quad \epsilon(d_{ij})=\delta_{i0}\delta_{j0},\quad\epsilon(e_{ij})=0,\\
&\quad\quad S(d_{ij})=d_{-i,-j},\quad S(e_{ij})=e_{ji}
\end{align*}
where $i,j\in\mathbb{Z}_k$ and the $e_{ij}$'s are defined in the same way as the Kac--Paljutkin's finite quantum group.
\end{exmp}
The following statement is nontrivial. See \cite{VD} for deatails.
\begin{prop}
For any finite quantum group there exists a state $h$ such that $(\id\otimes h)\Delta(\bullet)=h(\bullet)1$. In addition such a state is unique up to a constant multiple.
\end{prop}
\begin{exmp}(\cite{FG})
The state $h_{\mathcal{A}}\colon\mathcal{A}\to\mathbb{C}$ defined by \[h_{\mathcal{A}}=\dfrac{1}{8}(e_1^*+e_2^*+e_3^*+e_4^*)+\dfrac{1}{4}(a_{11}^*+a_{22}^*)\] is the Haar state of Kac--Paljutkin's finite quantum group, where $d^*_{i}\colon\mathcal{A}\to\mathbb{C}$ is the linear functional such that $e_{i}^*(e_{j})=\delta_{ij}, e^*_{i}(a_{jk})=0$ and $a_{ij}^*\colon\mathcal{A}\to\mathcal{C}$ is the linear functional such that $a_{ij}^*(a_{kl})=\delta_{ik}\delta_{jl}, a^*_{ij}(e_{k})=0$.
\end{exmp}
\begin{exmp}(\cite{Z})
The state $h_{\mathcal{A}_k}=\dfrac{1}{2k^2}\sum_{i,j\in\mathbb{Z}_k}d^{\ast}_{ij}+\dfrac{1}{2k}\sum_{r\in\mathbb{Z}_k}e^*_{rr}$ is the Haar state of Sekine quantum group $\mathcal{A}_k$, where $d^*_{ij}\colon\mathcal{A}_k\to\mathbb{C}$ is the linear functional such that $d_{ij}^*(d_{ab})=\delta_{ia}\delta_{jb}, d^*_{ij}(e_{ab})=0$ and $e_{ij}^*\colon\mathcal{A}_k\to\mathcal{C}$ is the linear functional such that $e_{ij}^*(e_{ab})=\delta_{ia}\delta_{jb}, e^*_{ij}(d_{ab})=0$.
\end{exmp}
\subsection{Coideals, idempotent states and group-like projections}
In this subsection, we introduce the concepts of coideals, idempotent states and group-like projections. We also introduce the concept of integrals for coideals.
\begin{defn}
Let $A$ be a Hopf algebra. A unital subalgebra $L$ is called a \emph{unital left coideal subalgebra} of $A$ if $\Delta(L)\subseteq A\otimes L$. In addition, if $A$ is a Hopf $\ast$-algebra and $L$ is a unital $\ast$-subalgebra satisfying $\Delta(L)\subseteq A\otimes L$, then we say $L$ is a \emph{unital left coideal $\ast$-subalgebra} of $A$.
\end{defn} 
\begin{defn}
Let $L$ be a unital left coideal subalgebra of a Hopf algebra $A$. For $a,b\in A$, we define $a\underset{\text{ad}}{\bullet}b\coloneqq a_{(1)}bS(a_{(2)})$. We say $L$ is \emph{normal} if $a\underset{\text{ad}}{\bullet}x\in L$ for all $a\in A$ and $x\in L$.
\end{defn}
\begin{exmp}
Let $G$ be a finite group and let $H$ be a subgroup of $G$. Then the group algebra $\mathbb{C}[H]$ is a unital left coideal $\ast$-subalgebra of $\mathbb{C}[G]$. It is normal if $H$ is a normal subgroup of $G$. The algebra $C(G/H)$ of continuous functions on the homogeneous space $G/H$ is a unital left coideal $\ast$-subalgebra of $C(G)$, which is normal since $C(G)$ is commutative.
\end{exmp}
\begin{exmp}(cf.~\cite{FG})
Let
\begin{enumerate}
\item $L_1=\mathcal{A}$
\item $L_2=\spn\{e_1+e_2, e_3+e_4, a_{1,1}+a_{2,2}, a_{1,2}-\sqrt{-1}a_{2,1}\}$
\item $L_3=\spn\{e_1+e_4, e_2+e_4, a_{1,1}+a_{2,2}, a_{1,2}+\sqrt{-1}a_{2,1}\}$
\item $L_4=\spn\{e_1+e_4, e_2+e_3, a_{1,1}, a_{2,2}\}$
\item $L_5=\spn\{e_1+e_2+e_3+e_4, a_{11}+a_{2,2}\}$
\item $L_6=\spn\{e_1+e_4+a_{1,1}, e_2+e_3+a_{2,2}\}$
\item $L_7=\spn\{e_1+e_4+a_{2,2}, e_2+e_3+a_{1,1}\}$
\item $L_8=\mathbb{C}1.$
\end{enumerate}
They unital left coideal $\ast$-algebras of Kac--Paljutkin's finite quantum group $\mathcal{A}$. We will show that $L_4$ and $L_5$ are normal in Theorem \ref{KPnilp}.
\end{exmp}
\begin{exmp}(cf.~\cite{Z})
Let $i=1,\ldots,k-1$. Let $\Gamma_i=\{(j,ij)\in\mathbb{Z}_k\times\mathbb{Z}_k\mid j\in\mathbb{Z}_k\}$ and $\Gamma_k=\mathbb{Z}_k\times\mathbb{Z}_k$. Put
\begin{align*}
L'_i&=\spn\{\sum_{(r,s)\in\Gamma_i}d_{p-r,q-s}, \sum_{(r,s)\in\Gamma_i}\eta^{r(q-p)}e_{p-s,q-s}\mid p,q\in\mathbb{Z}_k\},\\
L'_k&=\spn\{\sum_{r,s\in\mathbb{Z}_k}d_{r,s}, \sum_{s\in\mathbb{Z}_k}e_{s,s}\}.
\end{align*}
They are unital left coideal $\ast$-algebras of Sekine quantum group $\mathcal{A}_k$. We will show that $L'_1$ and $L'_k$ are normal (Theorem \ref{Snilp}).
\end{exmp}
\begin{defn}
Let $\phi\ A\to\mathbb{C}$ be a state on a finite quantum group $A$. We say that $\phi$ is an idempotent state if $\phi=\phi\ast\phi\coloneqq(\phi\otimes\phi)\Delta$.
\end{defn}
\begin{exmp}(\cite{FG})
It is known that the idempotent states of Kac--Paljutkin's finite quantum group are the following ones:
\begin{enumerate}
\item $\rho_1=\epsilon$,
\item $\rho_2=\dfrac{1}{2}(e_1^*+e_2^*)$,
\item $\rho_3=\dfrac{1}{2}(e_1^*+e_3^*)$,
\item $\rho_4=\dfrac{1}{2}(e_1^*+e_4^*)$,
\item $\rho_5=\dfrac{1}{4}(e_1^*+e_2^*+e_3^*+e_4^*)$,
\item $\rho_6=\dfrac{1}{4}(e_1^*+e_4^*)+\dfrac{1}{2}a_{11}^*$,
\item $\rho_7=\dfrac{1}{4}(e_1^*+e_4^*)+\dfrac{1}{2}a_{22}^*$ and
\item $\rho_8=h_{\mathcal{A}}$.
\end{enumerate}
\end{exmp}
\begin{exmp}(\cite{Z})
We put $h_i=\dfrac{1}{k}\sum_{(p,q)\in\Gamma_i}d^*_{p,q}$ and $h_k=\dfrac{1}{k^2}\sum_{(p,q)\in\Gamma_k}d_{p,q}^*$. They are examples of idempotent states on Sekine quantum group $\mathcal{A}_k$. 
\end{exmp}
It is a well-known fact that any unital left coideal $\ast$-subalgebra $L$ of a finite quantum group $A$ is of the form $L=(\id\otimes\phi)\Delta(A)$ for some idempotent state $\phi\colon A\to\mathbb{C}$. Furthermore, it is known that
\[\phi\mapsto(\id\otimes\phi)\Delta(A)\]
is a order-preserving bijection between the set of idempotent states on $A$ and unital left coideal $\ast$-subalgebras of $A$. Here we define the order on the set of idempotent states on $A$ by
\[\phi\preccurlyeq\psi\overset{\text{def}}{\iff}\psi=\phi\ast\psi\coloneqq(\phi\otimes\psi)\Delta.\]
For the details of the facts above, see \cite{FS1} for example. It s easy to show that $L_i=(\id\otimes\rho_i)(\mathcal{A})$, $L'_i=(\id\otimes h_i)\Delta(\mathcal{A}_k)$ and $L'_k=(\id\otimes h_k)(\mathcal{A}_k)$.
\begin{defn}
Let $A$ be a finite quantum group. Let $p\in A$ be a projection (i.~e.~ $p=p^*=p^2$). We say $p$ is a \emph{group-like projection} if $\Delta(p)(1\otimes p)=p\otimes p$.
\end{defn}
There is a one to one correspondence between idempotent states and group-like projections:
\begin{prop}(\cite{FS2})
For an idempotent state $\phi$ on a finite quantum group $A$ there is a unique group-like projection $p\in A$ such that $\phi=h(\bullet p)/h(p)$, where $h$ denotes the Haar state on $A$.
\end{prop} 
\begin{defn}
Let $L$ be a unital left coideal subalgebra of a Hopf algebra $A$. An element $l\in L$ is called an \emph{integral} of $L$ if $lx=xl=\epsilon(x)l$ for all $x\in L$. 
\end{defn}
\begin{exmp}
It is easy to see that the group-like projections corresponding to the idempotent states $\rho_1,\ldots,\rho_8$ are the following:
\begin{enumerate}
\item $p_1=e_1$
\item $p_2=e_1+e_2$
\item $p_3=e_1+e_3$
\item $p_4=e_1+e_4$
\item $p_5=e_1+e_2+e_3+e_4$
\item $p_6=e_1+e_4+a_{11}$
\item $p_7=e_1+e_4+a_{22}$
\item $p_8=e_1+e_2+e_3+e_4+a_{11}+a_{22}.$
\end{enumerate}
It is also easy to see that the group-like projections $p_i$'s are the integrals of the coideals $L_i$'s.
\end{exmp}
It is easy to show the following lemmas.
\begin{lem}\label{i}
For any $i=1,\cdots,k-1,$ the element $p'_i\coloneqq\sum_{(p,q)\in\Gamma_i}d_{p,q}$ is a group-like projection satisfying $h_i(\bullet)=\dfrac{h_{\mathcal{A}_k}(\bullet p'_i)}{h_{\mathcal{A}_k}(p'_i)}$ on $\mathcal{A}_k$.
\end{lem}
\begin{lem}\label{k}
The element $p'_k\coloneqq\sum_{(p,q)\in\Gamma_k}d_{p,q}$ is a group-like projection satisfying $h_k(\bullet)=\dfrac{h_{\mathcal{A}_k}(\bullet p'_k)}{h_{\mathcal{A}_k}(p'_k)}$ on $\mathcal{A}_k$.
\end{lem}
\begin{lem}\label{C}
The element $d_{0,0}$ is a group-like projection satisfying $\epsilon(\bullet)=\dfrac{h_{\mathcal{A}_k}(\bullet d_{0,0})}{h_{\mathcal{A}_k}(d_{0,0})}$ on $\mathcal{A}_k$.
\end{lem}
It is also easy to see that group-like projections above are integrals for $L'_i$, $L'_k$ and $\mathcal{A}_k$ respectively.
\subsection{Universal $R$-matrices}
Simply put, a universal $R$-matrix is a solution of quantum Yang--Baxter equation. For details, see \cite{D} and \cite{K}.
\begin{defn}
A Hopf algebra $A$ is called \emph{quasitriangular} if there exists an invertible element $R\in A\otimes A$ such that  $R\Delta(\bullet)R^{-1}=\Delta^{\op}(\bullet)$, $(\Delta\otimes\id)R=R_{13}R_{23}$ and $(\id\otimes\Delta)R=R_{13}R_{12}$, where $\Delta^{\op}$ denotes the composition of $\Delta$ and the flip $a\otimes b\mapsto b\otimes a$. Here we used the leg-numbering notation: $(a\otimes b)_{13}\coloneqq a\otimes1\otimes b$, $(a\otimes b)_{12}\coloneqq(a\otimes b\otimes1)$ and so on. Such an $R$ is called a \emph{universal $R$-matrix}.
\end{defn}
\begin{exmp}(\cite{Suz, W})
Kac--Paljutkin's finite quantum group is quasitriangular. For a direct computation, see Appendix A.
\end{exmp}
\subsection{Nilpotency, solvability and conjugacy classes for finite dimensional Hopf algebras}
In \cite{CW}, Cohen and Westreich found the intrinsic definitions of niloptency and solvability for finite dimensional Hopf algebras. One of the good points of their is that nilpotency implies solvability. (If we define the solvability and nilpotency of Hopf algebras via the solvability and nilpotency of fusion categories, this implication does not hold in general \cite{ENO}.) Furthermore the celebrated Burnside's $p^aq^b$ theorem for Hopf algebras holds in the setting of Cohen--Westreich.
\begin{defn}(\cite[Definition 3.5]{CW}) Let $A$ be a semisimple Hopf algebra. A chain of unital left coideal subalgebras of $A$
\begin{align*}
L_0\subset L_1\subset\cdots\subset L_i
\end{align*}  
is called a \emph{solvable series} if the following conditions are satisfied for all $0\leq j\leq i-1$:
\begin{enumerate}
\item $l_j\in Z(L_{j+1})$, where $l_j$ denotes the integral of $L_j$ and $Z(L_{j+1})$ denotes the center of $L_{j+1}$.
\item $(a\underset{\text{ad}}{\bullet}b)l_j=\epsilon(a)bl_j$ for all $a,b\in L_{j+1}$.
\end{enumerate}
If there is a solvable series such that $L_0=\mathbb
{C}1$ and $L_i=A$ then the Hopf algebra is called \emph{solvable}. 
\end{defn}
\begin{defn}(\cite[Proposition 3.8]{CW})
A semisimple Hopf algebra $A$ is called \emph{nilpotent} if it has a chain of normal left coideal subalgebras
\[\mathbb{C}=L_0\subset L_1\subset\cdots\subset L_i=A\]
satisfying $L_{j+1}l_{j}\subset Z(Al_{j})$ for all $0\leq j\leq i-1$. 
\end{defn}
\begin{Rem}
Cohen--Westreich defined nilpotent Hopf algebras in a different way. By \cite[Proposition 3.8]{CW}, the definition above is equivalent to their original definition. Note that the group algebras of nilpotent or solvable finite groups satisfy the definitions above (\cite[Example 3.6]{CW}). 
\end{Rem}
The following theorems are natural and surprising ones.
\begin{thm}(\cite[Corollary 3.9]{CW})
Semisimple nlpotent Hopf algebras are solvable.
\end{thm}
\begin{thm}(\cite[Theorem 3.11]{CW})
Let $p$ and $q$ be prime numbers. Let $a$ and $b$ be non-negative integers. If $A$ is a quasitriangular semisimple Hopf algebra of dimension $p^aq^b$, then $A$ is solvable. In addition, if $L$ is a unital left coideal subalgebra, then $A$ has a solvable series containing the unital left coideal subalgebra $L$.
\end{thm}
As in the case of groups, there is a notion of conjugacy classes for Hopf algebras. We do not give adefinition of conjugacy classes. For details, see \cite{CW2}. Instead, we introduce properties conjugacy classes have.
\begin{prop}(\cite[Section 1]{CW2})
Let $\mathcal{C}_j$ denote the conjugacy classes oh a Hopf algebra $A$. Let $\mathcal{C}_0=\mathbb{C}$. Then it follows that
\begin{itemize}
\item$\sum_j\dim(\mathcal{C}_j)=\dim(A)$.
\item$\dim(A)$ is devided by each $\dim(\mathcal{C}_j)$.
\end{itemize}
\end{prop}

\section{On nilpotency and solvability}
As stated above, Cohen and Westreich gave a definition and characterizations of nilpotency and solvability for semisimple finite dimensional Hopf algebras. Moreover they showed Bunside theorem for quasi-triangular semisimple finite dimensional Hopf algebras. In this section we classify the series which shows the solvability of Kac--Paljutkin's finite quantum group $\mathcal{A}$ and by using one of that series we show the nilpotency of Kac--Paljutkin's finite quantum group $\mathcal{A}$.\par
First we have the following:
\begin{thm}
Kac--Paljutkkn's finite quantum group $\mathcal{A}$ is solvable.
\end{thm}
\begin{proof}
This is a easy consequence of Cohen--Westreich's Burnside theorem since $\dim\mathcal{A}=8=2^3$ and we know that $\mathcal{A}$ is quasitrangular (\cite{Suz,W} or Theorem \ref{quasi}).
\end{proof}
As a consequence of Cohen--Westreich's Burnside theorem, we can also classify the solvable series.
\begin{thm}
Any solvable series for Kac--Paljutkin's finite quantum group of length 4 is one of the following.
\begin{enumerate}
\item $\mathbb{C}\subset L_5\subset L_2\subset \mathcal{A}$
\item $\mathbb{C}\subset L_5\subset L_3\subset \mathcal{A}$
\item$\mathbb{C}\subset L_5\subset L_4\subset \mathcal{A}$
\item $\mathbb{C}\subset L_6\subset L_4\subset \mathcal{A}$
\item $\mathbb{C}\subset L_7\subset L_4\subset \mathcal{A}$
\end{enumerate}
\end{thm}
\begin{Rem}\label{REM}
    If $\mathbb{C}\subset K_1\subset K_2\subset\cdots K_n\subset\mathcal{A}$ is an increasing sequence of unital left coideal subalgebras, then $n\leq4$. Note that by Cohen--Westreich's Burnside theorem, there exists a solvable series containing the coideal $L_i$ for each $i$. Thus we know all the possible series:
\begin{enumerate}
\item If the series contains $L_2$, then it must be $\mathbb{C}\subset L_2\subset\mathcal{A}$ or $\mathbb{C}\subset L_5\subset L_2\subset\mathcal{A}$.
\item If the series contains $L_3$, then it must be $\mathbb{C}\subset L_3\subset\mathcal{A}$ or $\mathbb{C}\subset L_5\subset L_3\subset\mathcal{A}$.
\item If the series contains $L_4$, then it must be $\mathbb{C}\subset L_4\subset\mathcal{A}$, $\mathbb{C}\subset L_5\subset L_4\subset\mathcal{A}$,  $\mathbb{C}\subset L_6\subset L_4\subset\mathcal{A}$ or $\mathbb{C}\subset L_7\subset L_4\subset\mathcal{A}$.
\item  If the series contains $L_5$, then it must be $\mathbb{C}\subset L_5\subset\mathcal{A}$, $\mathbb{C}\subset L_5\subset L_2\subset\mathcal{A}$ or $\mathbb{C}\subset L_5\subset L_3\subset\mathcal{A}$.
\item If the series contains $L_6$, then it must be $\mathbb{C}\subset L_6\subset\mathcal{A}$ or $\mathbb{C}\subset L_6\subset L_4\subset\mathcal{A}$.
\item If the series contains $L_7$, then it must be $\mathbb{C}\subset L_7\subset\mathcal{A}$ or $\mathbb{C}\subset L_7\subset L_4\subset\mathcal{A}$.
\end{enumerate}
\end{Rem}
\begin{proof}
 It suffices to show that the series listed above are solvable series. We must verify the following items:
 \begin{enumerate}
 \item
 \begin{enumerate}
 \item $p_8\in Z(L_5)$
 \item $p_5\in Z(L_2)$
 \item $p_2\in Z(\mathcal{A})$
 \item $(a\underset{\text{ad}}{\bullet}b)p_2=\epsilon(a)bp_2$ for all $a,b\in\mathcal{A}$
 \item $(a\underset{\text{ad}}{\bullet}b)p_5=\epsilon(a)bp_5$ for all $a,b\in L_2$
 \item $(a\underset{\text{ad}}{\bullet}b)p_8=\epsilon(a)bp_8$ for all $a,b\in L_5$
 \end{enumerate}
 \item
 \begin{enumerate}
 \item $p_8\in Z(L_5)$
 \item $p_5\in Z(L_3)$
 \item $p_3\in Z(\mathcal{A})$
 \item $(a\underset{\text{ad}}{\bullet}b)p_3=\epsilon(a)bp_3$ for all $a,b\in\mathcal{A}$
 \item $(a\underset{\text{ad}}{\bullet}b)p_5=\epsilon(a)bp_5$ for all $a,b\in L_3$
 \item $(a\underset{\text{ad}}{\bullet}b)p_8=\epsilon(a)bp_8$ for all $a,b\in L_5$
 \end{enumerate}
 \item
 \begin{enumerate}
 \item $p_8\in Z(L_5)$
 \item $p_5\in Z(L_4)$
 \item $p_4\in Z(\mathcal{A})$
 \item $(a\underset{\text{ad}}{\bullet}b)p_4=\epsilon(a)bp_4$ for all $a,b\in\mathcal{A}$
 \item $(a\underset{\text{ad}}{\bullet}b)p_5=\epsilon(a)bp_5$ for all $a,b\in L_4$
 \item $(a\underset{\text{ad}}{\bullet}b)p_8=\epsilon(a)bp_8$ for all $a,b\in L_5$
 \end{enumerate}
\item
 \begin{enumerate}
 \item $p_8\in Z(L_6)$
 \item $p_6\in Z(L_4)$
 \item $p_4\in Z(\mathcal{A})$
 \item $(a\underset{\text{ad}}{\bullet}b)p_4=\epsilon(a)bp_4$ for all $a,b\in\mathcal{A}$
 \item $(a\underset{\text{ad}}{\bullet}b)p_6=\epsilon(a)bp_6$ for all $a,b\in L_4$
 \item $(a\underset{\text{ad}}{\bullet}b)p_8=\epsilon(a)bp_8$ for all $a,b\in L_6$
 \end{enumerate}
 \item
 \begin{enumerate}
 \item $p_8\in Z(L_7)$
 \item $p_7\in Z(L_4)$
 \item $p_4\in Z(\mathcal{A})$
 \item $(a\underset{\text{ad}}{\bullet}b)p_4=\epsilon(a)bp_4$ for all $a,b\in\mathcal{A}$
 \item $(a\underset{\text{ad}}{\bullet}b)p_7=\epsilon(a)bp_7$ for all $a,b\in L_4$
 \item $(a\underset{\text{ad}}{\bullet}b)p_8=\epsilon(a)bp_8$ for all $a,b\in L_7$
 \end{enumerate}
 \end{enumerate}

 It is clear that the items (1.a), (1.b), (1.c), (2.a), (2.b), (2.c), (3.a), (3.b), (3.c), (4.a), (4.c), (5.a) and (5.c) hold, since $p_i$ is a central projection if $i=1,2,3,4,5,8$.\par
Next we verify the items (4.b) and (5.b). Note that 
 \begin{align*}
     L_4=\spn\{e_1+e_4, e_2+e_3, a_{11}, a_{22}\}.
 \end{align*}
 We have 
 \begin{align*}
     p_6(e_1+e_4)&=e_1+e_4=(e_1+e_4)p_6\\
     p_6(e_2+e_3)&=0=(e_2+e_3)p_6\\
     p_6a_{11}&=a_{11}=a_{11}p_6\\
     p_6a_{22}&=0=a_{22}p_6\\
      p_7(e_1+e_4)&=e_1+e_4=(e_1+e_4)p_7\\
     p_7(e_2+e_3)&=0=(e_2+e_3)p_7\\
     p_7a_{11}&=0=a_{11}p_7\\
     p_7a_{22}&=a_{22}=a_{22}p_7.
 \end{align*}
 Thus the items (4.b) and (5.b) hold.\par
By Remark \ref{REM} the items (1.d), (1.f), (2.d), (2.f), (3.d), (3.f), (4.d), (4.f), (5.d) and (5.f) hold.\par
Let us verify the item (1.e). It suffices to verify this when $a,b\in\{e_1+e_2, e_3+e_4, a_{11}+a_{22}, a_{12}+\sqrt{-1}a_{21}\}$. If $a=b=e_1+e_2$, then $a\underset{\text{ad}}{\bullet}b=e_1+e_2$ and $\epsilon(a)b=e_1+e_2$. Thus $(a\underset{\text{ad}}{\bullet}b)p_5=\epsilon(a)bp_5$. If $a=e_1+e_2$ and $b=e_3+e_4$, then  $a\underset{\text{ad}}{\bullet}b=e_3+e_4$ and $\epsilon(a)b=e_3+e_4$. Thus $(a\underset{\text{ad}}{\bullet}b)p_5=\epsilon(a)bp_5$. If $a=e_1+e_2$ and $b=a_{11}+a_{22}$, then  $a\underset{\text{ad}}{\bullet}b\in\mathbb{M}_2$ and $\epsilon(a)b\in\mathbb{M}_2$. Thus $(a\underset{\text{ad}}{\bullet}b)p_5=\epsilon(a)bp_5=0$. If $a=e_1+e_2$ and $b=a_{12}+\sqrt{-1}a_{21}$, then $a\underset{\text{ad}}{\bullet}b\in\mathbb{M}_2$ and $\epsilon(a)b\in\mathbb{M}_2$. Thus $(a\underset{\text{ad}}{\bullet}b)p_5=\epsilon(a)bp_5=0$. If $a=e_3+e_4$ and $b=e_1+e_2, e_3+e_4$, then $(a\underset{\text{ad}}{\bullet}b)=\epsilon(a)b=0$. Thus $(a\underset{\text{ad}}{\bullet}b)p_5=\epsilon(a)bp_5=0$. If $a=e_3+e_4$ and $b=a_{11}+a_{22}, a_{12}+\sqrt{-1}a_{21}$, then $(a\underset{\text{ad}}{\bullet}b)\in\mathbb{M}_2$ and $\epsilon(a)b\in\mathbb{M}_2$. Thus $(a\underset{\text{ad}}{\bullet}b)p_5=\epsilon(a)bp_5=0$. If $a=a_{11}+a_{22}, a_{12}-\sqrt{-1}a_{21}$, then $\epsilon(a)=0$ and $\Delta(a)\not\in\spn\{a_{ij}\otimes a_{kl}, e_{p}\otimes e_{q}\mid i,j,k,l=1,2; p,q=1,2,3,4\}$. Thus $(a\underset{\text{ad}}{\bullet}b)=\epsilon(a)b=0$ and we have  $(a\underset{\text{ad}}{\bullet}b)p_5=\epsilon(a)bp_5=0$. In a similar way, we can verify the item (2.e).\par
Next we verify the item (3.e). As before, we may assume that $a,b\in\{e_1+e_4, e_2+e_3, a_{11}, a_{22}\}$. If $a=e_1+e_4$ and $b=e_1+e_4$, then $a\underset{\text{ad}}{\bullet}b=e_1+e_4$ and $\epsilon(a)b=e_1+e_4$. Thus $a\underset{\text{ad}}{\bullet}bp_5=\epsilon(a)bp_5$. If $a=e_1+e_4$ and $b=e_2+e_3$, then $a\underset{\text{ad}}{\bullet}b=e_2+e_3$ and $\epsilon(a)b=e_2+e_3$. Thus $(a\underset{\text{ad}}{\bullet}b)p_5=\epsilon(a)bp_5$. If $a=e_1+e_4$ and $b=a_{11}$, then $a\underset{\text{ad}}{\bullet}b=a_{11}$ and $\epsilon(a)b=a_{11}$. Thus $(a\underset{\text{ad}}{\bullet}b)p_5=\epsilon(a)bp_5$. If $a=e_1+e_4$ and $b=a_{22}$, then $a\underset{\text{ad}}{\bullet}b=a_{22}$ and $\epsilon(a)b=a_{22}$. Thus $(a\underset{\text{ad}}{\bullet}b)p_5=\epsilon(a)bp_5$. If $a=e_2+e_3, a_{11}, a_{22}$, then $a\underset{\text{ad}}{\bullet}b=\epsilon(a)b=0$. Thus $(a\underset{\text{ad}}{\bullet}b)p_5=\epsilon(a)bp_5$. In a similar way, we can verify the itms (4.e) and (5.e).
\end{proof}
Furthermore we can show that Kac--Paljutkin's finite quantum group is nilpotent.
\begin{thm}\label{KPnilp}
Kac--Paljutkin's finite quantum group $\mathcal{A}$ is nilpotent.
\end{thm} 
\begin{proof}
    We claim that the solvable series $\mathbb{C}\subset L_5\subset L_4\subset \mathcal{A}$ shows the nilpotency of Kac--Paljutkin's finite quantum group. We must prove that $L_5e_1\subset Z(\mathcal{A}e_1), L_4p_5\subset Z(\mathcal{A}p_5)$, $\mathcal{A}p_4\subset Z(\mathcal{A}p_4)$, $L_4$ is normal and $L_5$ is normal. The first, second and third ones follow from $L_5e_1=\mathbb{C}e_1$, $L_4p_5=\spn\{e_1+e_4, e_2+e_3\}$, $\mathcal{A}e_1=\mathbb{C}e_1$, $\mathcal{A}p_5=\spn\{e_1, e_2, e_3, e_4\}$ and $\mathcal{A}p_4=\spn\{e_1, e_4\}$. The normality of $L_4$ follows from the following.
\begin{align*}
e_1\underset{\text{ad}}{\bullet}(e_1+e_4)=&e_1+e_4, e_i\underset{\text{ad}}{\bullet}(e_1+e_4)=0 (i=2,3,4),\\
 a_{ij}\underset{\text{ad}}{\bullet}(e_1+e_4)&=0\hspace{1mm}(i,j=1,2), e_1\underset{\text{ad}}{\bullet}(e_2+e_3)=e_2+e_3,\\
e_i\underset{\text{ad}}{\bullet}(e_2+e_3)&=0\hspace{1mm}(i=2,3,4), a_{ij}\underset{\text{ad}}{\bullet}(e_2+e_3)=0\hspace{1mm}(i,j=1,2),\\
e_1\underset{\text{ad}}{\bullet}a_{11}&=\dfrac{1}{2}(a_{11}+a_{22}), e_i\underset{\text{ad}}{\bullet}a_{11}=0\hspace{1mm}(i=2,3),\\ e_4\underset{\text{ad}}{\bullet}a_{11}&=\dfrac{1}{2}(a_{11}-a_{22}), a_{ij}\underset{\text{ad}}{\bullet}a_{11}=0\hspace{1mm}(i,j=1,2),\\ e_1\underset{\text{ad}}{\bullet}a_{22}&=\dfrac{1}{2}(a_{11}+a_{22}), e_i\underset{\text{ad}}{\bullet}a_{22}=0\hspace{1mm}(i=2,3),\\ e_4\underset{\text{ad}}{\bullet}a_{22}&=\dfrac{1}{2}(a_{22}-a_{11}), a_{ij}\underset{\text{ad}}{\bullet}a_{22}=0\hspace{1mm}(i,j=1,2).\\
\end{align*} 
In a similar way one can show the normality of $L_5$.
\end{proof}
The converse of Burnside theorem is not true. (For example, any groups of order $2pq$ are solvable where $p$ and $q$ are odd primes.) We show that each Sekine quantum group $\mathcal{A}_k$ is nilpotent and thus $\mathcal{A}_{15}$ is a counter-example of the converse of Cohen--Westreich's Burnside theorem, which is neither commutative nor cocommutative. 
\begin{thm}\label{Snilp}
Sekine quantum group $\mathcal{A}_{k}$ is nilpotent for each $k$. In particular, Sekine quantum group $\mathcal{A}_{15}$ is a solvable Hopf algebra of dimension $2\times3^2\times5^2$.
\end{thm}
\begin{proof}
Recall that 
\begin{align*}
L'_i&=\spn\{\sum_{(r,s)\in\Gamma_i}d_{p-r,q-s}, \sum_{(r,s)\in\Gamma_i}\eta^{r(q-p)}e_{p-s,q-s}\mid p,q\in\mathbb{Z}_k\}=(\id\otimes h_i)\Delta(\mathcal{A}_k)\\
L'_k&=\spn\{\sum_{r,s\in\mathbb{Z}_k}d_{r,s}, \sum_{s\in\mathbb{Z}_k}e_{s,s}\}=(\id\otimes h_k)\Delta(\mathcal{A}_k).
\end{align*}
We show that the series $\mathbb{C}\subset L'_k\subset L'_1\subset\mathcal{A}_k$ gives the nilpotency of $\mathcal{A}_k$. We must show that
\begin{enumerate}
\item $L'_k$ is normal.
\item $L'_1$ is normal.
\item $L'_kd_{0,0}\subset Z(\mathcal{A}_kd_{0,0})$.
\item $L'_1p'_k\subset Z(\mathcal{A}_kp'_k)$.
\item $\mathcal{A}_kp'_1\subset Z(\mathcal{A}_kp'_1)$.
\end{enumerate}
Notice that $Z(\mathcal{A}_kd_{0,0})=Z(\mathcal{A}_kp'_k)=Z(\mathcal{A}_kp'_i)=\mathcal{A}_k$, then it is easy to see that the items 3, 4 and 5 hold.\par 
Let us show the item 1. We have
\begin{align*}
d_{a,b}\underset{\text{ad}}{\bullet}(\sum_{r,s\in\mathbb{Z}_k}d_{r,s})=&\sum_{r,s\in\mathbb{Z}_k}\sum_{m,n\in\mathbb{Z}_k}d_{m,n}d_{r,s}d_{-a+m,-b+n}\\
=&\sum_{r,s\in\mathbb{Z}_k}\sum_{m,n\in\mathbb{Z}_k}\delta_{m,r}\delta_{n,s}d_{r,s}d_{-a+m,-b+n}\\
=&\sum_{r,s\in\mathbb{Z}_k}d_{r,s}d_{-a+r,-b+s}\\
=&\sum_{r,s\in\mathbb{Z}_k}\delta_{r,-a+r}\delta_{s,-b+s}d_{r,s}\\
=&\delta_{a,0}\delta_{b,0}\sum_{r,s}d_{r,s}.
\end{align*}
Thus we see that $d_{a,b}\underset{\text{ad}}{\bullet}(\sum_{r,s\in\mathbb{Z}_k}d_{r,s})\in L'_k$. We also have
\begin{align*}
d_{a,b}\underset{\text{ad}}{\bullet}(\sum_{s\in\mathbb{Z}_k}e_{s,s})=&\sum_{m,n\in\mathbb{Z}_k}\sum_{s\in\mathbb{Z}_k}\eta^{a(m-n)}e_{m,n}e_{s,s}e_{n+b,m+b}\\
=&\sum_{m,n\in\mathbb{Z}_k}\sum_{s\in\mathbb{Z}_k}\eta^{a(m-n)}_{m,m+b}\delta_{n,s}\delta_{s,n+b}\\
=&\sum_{m,n\in\mathbb{Z}_k}\eta^{a(m-n)}\delta_{0,b}e_{m,m+b}\\
=&\begin{cases}
    k\sum_{m\in\mathbb{Z}_k}e_{m,m} & \text{if $a=0$ and $b=0$,} \\
    0                 & \text{if $a=0$ and $b\neq0$,}\\
    0                 & \text{if $a\neq0$ and $b=0$,}\\
    0                 & \text{if $a\neq0$ and $b\neq0$.}
  \end{cases}
\end{align*}
It is clear that $e_{a,b}\underset{\text{ad}}{\bullet}(\sum_{r,s\in\mathbb{Z}_k}d_{r,s})=e_{a,b}\underset{\text{ad}}{\bullet}(\sum_{s\in\mathbb{Z}_k}e_{s,s})=0$. Hence the coideal $L_k$ is normal.
\par
Let us show the item 2. We have
\begin{align*}
d_{a,b}\underset{\text{ad}}{\bullet}(\sum_{(r,s)\in\Gamma_1}d_{p-r,q-s})=&\sum_{m,n\in\mathbb{Z}_k}\sum_{(r,s)\in\Gamma_1}d_{m,n}d_{p-r,q-s}S(d_{a-m,b-n})\\
=&\sum_{m,n\in\mathbb{Z}_k}\sum_{(r,s)\in\Gamma_1}d_{m,n}d_{p-r,q-s}d_{-a+m,-b+n}\\
=&\sum_{m,n\in\mathbb{Z}_k}\sum_{(r,s)\in\Gamma_1}\delta_{m,p-r}\delta_{n,q-s}d_{m,n}d_{-a+m,-b+n}\\
=&\sum_{(r,s)\in\Gamma_1}d_{p-r,q-s}d_{-a+p-r,-b+q-s}\\
=&\sum_{(r,s)\in\Gamma_1}\delta_{p-r,-a+p-r}\delta_{q-s,-b+q-s}d_{p-r,q-s}.
\end{align*}
Hence we see that
\begin{equation*}
  d_{a,b}\underset{\text{ad}}{\bullet}(\sum_{(r,s)\in\Gamma_1}d_{p-r,q-s})=
  \begin{cases}
    \sum_{(r,s)\in\Gamma_1}d_{p-r,q-s} & \text{if $(a,b)=(0,0)$,} \\
    0                 & \text{if $(a,b)\neq(0,0)$.}
  \end{cases}
\end{equation*}
It is clear that $e_{a,b}\underset{\text{ad}}{\bullet}(\sum_{(r,s)\in\Gamma_1}d_{p-r,q-s})=0$. We also have
\begin{align*}
 &d_{a,b}\underset{\text{ad}}{\bullet}(\sum_{(r,s)\in\Gamma_1}\eta^{r(q-p)}e_{p-r,q-s})\\
=&\dfrac{1}{k}\sum_{m,n\in\mathbb{Z}_k}\sum_{(r,s)\in\Gamma_1}\eta^{a(m-n)}\eta^{r(q-p)}e_{m,n}e_{p-s,q-s}e_{n+b,m+b}\\
=&\dfrac{1}{k}\sum_{m,n\in\mathbb{Z}_k}\sum_{(r,s)\in\Gamma_1}\eta^{a(m-n)}\eta^{r(q-p)}\delta_{n,p-s}\delta_{q-s,n+b}e_{m,m+b}\\
=&\dfrac{1}{k}\sum_{m\in\mathbb{Z}_k}\sum_{(r,s)\in\Gamma_1}\eta^{a(m-p+s)}\eta^{r(q-p)}\delta_{q-s,p-s+b}e_{m,m+b}\\
=&\begin{cases}
    \dfrac{1}{k}\sum_{m\in\mathbb{Z}_k}\sum_{(r,s)\in\Gamma_1}\eta^{a(m-p+s)}\eta^{rb}e_{m,m+b} & \text{if $q=p+b$,} \\
    0                 & \text{if $q\neq p+b$.}
  \end{cases}\\
=&\begin{cases}
    \dfrac{1}{k}\sum_{m\in\mathbb{Z}_k}\sum_{r=0}^{k-1}\eta^{a(m-p)}\eta^{rai}\eta^{rb}e_{m,m+b} & \text{if $q=p+b$,} \\
    0                 & \text{if $q\neq p+b$.}
  \end{cases}\\
=&\begin{cases}
    \sum_{m\in\mathbb{Z}_k}\eta^{a(m-p)}e_{m,m+b} & \text{if $q=p+b$ and $a+b=0$,} \\
    0                 & \text{if $q=p+b$ and $a+b\neq0$,}\\
    0                 & \text{if $q\neq p+b$.}
  \end{cases}
\end{align*}
If $a+b=0$, then we have
\begin{align*}
\eta^{-b^2}\sum_{m=0}^{k-1}\eta^{am}e_{m,m+b}=&\sum_{r=0}^{k-1}\eta^{br}e_{-b-r,-r}\\
=&\sum_{r=0}^{k-1}\eta^{-ar}e_{a-r,-r}\\
=&\sum_{r=0}^{k-1}\eta^{r(0-a)}e_{a-r, 0-r}.
\end{align*}
From this we see that $d_{a,b}\underset{\text{ad}}{\bullet}(\sum_{(r,s)\in\Gamma_1}\eta^{r(q-p)}e_{p-r,q-s})\in L'_1$. It is clear that $e_{a,b}\underset{\text{ad}}{\bullet}(\sum_{(r,s)\in\Gamma_1}\eta^{r(q-p)}e_{p-r,q-s})=0$. Therefore the coideal $L'_1$ is normal.
\end{proof}
\section{Hopf algebras of dimension $2pq$}\label{2pq}
It is well-known that any group of order $2pq$ is not simple. In this section we show that any quasitriangular and semisimple Hopf algebra of dimension $2pq$ has a nontrivial normal unital left coideal subalgebra if $p$ and $q$ are distinct odd primes such that $p-2$ is divided by $5$ and $q\neq5$.
\begin{prop}Let $p$ and $q$ be as above.
Let $A$ be a quasitiangular and semisimple Hopf algebra of dimension $2pq$. Then $A$ has a nontrivial normal unital left coideal.
\end{prop}
\begin{proof}
It suffices to show the proposition when $A$ is non-commutative. Since $1+\sum_{j\geq1}\dim\mathcal{C}_j=2pq(=\dim A)$, there exists some $j_0\geq1$ such that $p\nmid\dim\mathcal{C}_{j_0}$. Since $\dim\mathcal{C}_{j_0}\mid2pq$, we have $\dim\mathcal{C}_j=1, q$ or $2q$. If $\dim\mathcal{C}_{j_0}=1, q$, then apply \cite[Theorem 4.5]{CW2} and we get the result. If $\dim\mathcal{C}_{j_0}=2q$, then there exists some $j_1\neq0, j_{0}$ such that $q\nmid\dim\mathcal{C}_{j_1}$ since $1+2q+\sum_{j\neq0,j_{0}}$. By the fact that $\dim\mathcal{C}_{j_1}\mid2pq$, we have $\dim\mathcal{C}_{j_1}=1,p,2p$. As before, we get the result when $\dim\mathcal{C}_{j_1}=1,p$. If $\dim\mathcal{C}_{j_1}=2p$, there exists some $j_2\neq0,j_0,j_1$ such that $2\nmid\dim\mathcal{C}_{j_2}$ since $1+2p+2q+\sum_{j\neq0,j_0,j_1}\dim\mathcal{C}_j=2pq$. The fact taht $\dim\mathcal{C}_{j_2}\mid2pq$ implies that $\dim\mathcal{C}_{j_2}=1,p,q,pq$. If $\dim\mathcal{C}_{j_2}=1,p,q$, then apply \cite[Theorem 4.5]{CW2} again and we get the result. Suppose $\dim\mathcal{C}_{j_2}=pq$. We have $1+2p+2q+pq+\sum_{j\neq0,j_0,j_1,j_2}\dim\mathcal{C}_j=2pq$. There exists some integer $k$ such that $p=5k+2$ by the assumption. Then we must have $5k(q-2)-5=\sum_{j\neq0,j_0,j_1,j_2}\dim\mathcal{C}_j\mid2pq$, which contradicts the assumption $q\neq5$. Therefore $\dim\mathcal{C}_{j_2}\neq pq$ and this completes the proof.
\end{proof}



\appendix
\section{A direct computation of the universal $R$-matrices for Kac--Paljutkin's finite quantum group\protect\footnote{%
One can check our computation by using \emph{Mathematica} (\cite{Wol}). For the code, see \href{https://github.com/guunterr/QuantumGroups/tree/master/check}{https://github.com/guunterr/QuantumGroups/tree/master/check}.}}
The quasi-triangular structure of Kac--Paljutkin's finite quantum group is well-known (\cite{Suz}, \cite{W}). In this appendix we give a direct computation of the $R$-matrices of Kac--Paljutkin's finite quantum group. Note that this is not a new result. We put 
\[R=\sum_{ij}A_{ij}e_i\otimes e_j+\sum_{ijk}B_{ijk}e_i\otimes a_{jk}+\sum_{ijk}C_{ijk}a_{ij}\otimes e_k+\sum_{ijkl}D_{ijkl}a_{ij}\otimes a_{kl}.\] 
Suppose $R$ be a universal $R$-matrix of Kac--Paljutkin's finite quantum group. 
By $R\Delta=\Delta^{\text{op}}R$, we have $R\Delta(e_i)=\Delta^{\text{op}}(e_i)R$ and $R\Delta(a_{jk})=\Delta^{\text{op}}(a_{jk})R$, where $i=1, 2, 3, 4$ and $j, k=1, 2$. Note that
\begin{align*}
&R\Delta(e_1)\\
=&\sum_{i,j}A_{ij}(e_i\otimes e_j)(e_1\otimes e_1)+\sum_{i,j}A_{ij}(e_i\otimes e_j)(e_2\otimes e_2)\\
+&\sum_{i,j}A_{ij}(e_i\otimes e_j)(e_3\otimes e_3)+\sum_{i,j}A_{ij}(e_i\otimes e_j)(e_4\otimes e_4)\\
+&\dfrac{1}{2}\sum_{i,j,k,l}D_{ijkl}(a_{ij}\otimes a_{kl})(a_{11}\otimes a_{11})+\dfrac{1}{2}\sum_{ijkl}D_{ijkl}(a_{ij}\otimes a_{kl})(a_{12}\otimes a_{12})\\
+&\dfrac{1}{2}\sum_{i,j,k,l}D_{ijkl}(a_{ij}\otimes a_{kl})(a_{21}\otimes a_{21})+\dfrac{1}{2}\sum_{i,j,k,l}D_{ijkl}(a_{ij}\otimes a_{kl})(a_{22}\otimes a_{22})\\
=&A_{11}(e_1\otimes e_1)+A_{22}(e_2\otimes e_2)+A_{33}(e_3\otimes e_3)+A_{44}(e_4\otimes e_4)\\
+&\dfrac{1}{2}D_{1111}(a_{11}\otimes a_{11})+\dfrac{1}{2}D_{1121}(a_{11}\otimes a_{21})+\dfrac{1}{2}D_{2111}(a_{21}\otimes a_{11})\\
+&\dfrac{1}{2}D_{2121}(a_{21}\otimes a_{21})+\dfrac{1}{2}D_{1111}(a_{12}\otimes a_{12})+\dfrac{1}{2}D_{1121}(a_{12}\otimes a_{22})\\
+&\dfrac{1}{2}D_{2111}(a_{22}\otimes a_{12})+\dfrac{1}{2}D_{2121}(a_{22}\otimes a_{22})+\dfrac{1}{2}D_{1212}(a_{11}\otimes a_{11})\\
+&\dfrac{1}{2}D_{1222}(a_{11}\otimes a_{21})+\dfrac{1}{2}D_{2212}(a_{21}\otimes a_{11})+\dfrac{1}{2}D_{2222}(a_{21}\otimes a_{21})\\
+&\dfrac{1}{2}D_{1212}(a_{12}\otimes a_{12})+\dfrac{1}{2}D_{1222}(a_{12}\otimes a_{22})+\dfrac{1}{2}D_{2212}(a_{22}\otimes a_{12})\\
+&\dfrac{1}{2}D_{2222}(a_{22}\otimes a_{22}),
\end{align*}
\begin{align*}
&\Delta^{\text{op}}(e_1)R\\
=&\sum_{i,j}A_{ij}(e_1\otimes e_1)(e_i\otimes e_j)+\sum_{i,j}A_{ij}(e_2\otimes e_2)(e_i\otimes e_j)\\
+&\sum_{i,j}A_{ij}(e_3\otimes e_3)(e_i\otimes e_j)+\sum_{i,j}A_{ij}(e_4\otimes e_4)(e_i\otimes e_j)\\
+&\dfrac{1}{2}\sum_{i,j,k,l}D_{ijkl}(a_{11}\otimes a_{11})(a_{ij}\otimes a_{kl})+\dfrac{1}{2}\sum_{i,j,k,l}D_{ijkl}(a_{12}\otimes a_{12})(a_{ij}\otimes a_{kl})\\
+&\dfrac{1}{2}\sum_{i,j,k,l}D_{ijkl}(a_{21}\otimes a_{21})(a_{ij}\otimes a_{kl})+\dfrac{1}{2}\sum_{i,j,k,l}D_{ijkl}(a_{22}\otimes a_{22})(a_{ij}\otimes a_{kl})
\end{align*}
\begin{align*}
=&A_{11}(e_1\otimes e_1)+A_{22}(e_2\otimes e_2)+A_{33}(e_3\otimes e_3)+A_{44}(e_4\otimes e_4)\\
+&\dfrac{1}{2}D_{1111}(a_{11}\otimes a_{11})+\dfrac{1}{2}D_{1112}(a_{11}\otimes a_{12})+\dfrac{1}{2}D_{1211}(a_{12}\otimes a_{11})\\
+&\dfrac{1}{2}D_{1212}(a_{12}\otimes a_{12})+\dfrac{1}{2}D_{2121}(a_{11}\otimes a_{11})+\dfrac{1}{2}D_{2122}(a_{11}\otimes a_{12})\\
+&\dfrac{1}{2}D_{2221}(a_{12}\otimes a_{11})+\dfrac{1}{2}D_{2222}(a_{12}\otimes a_{12})+\dfrac{1}{2}D_{1111}(a_{21}\otimes a_{21})\\
+&\dfrac{1}{2}D_{1112}(a_{21}\otimes a_{22})+\dfrac{1}{2}D_{1211}(a_{22}\otimes a_{21})+\dfrac{1}{2}D_{1212}(a_{22}\otimes a_{22})\\
+&\dfrac{1}{2}D_{2121}(a_{21}\otimes a_{21})+\dfrac{1}{2}D_{2122}(a_{21}\otimes a_{22})+\dfrac{1}{2}D_{2221}(a_{22}\otimes a_{21})\\
+&\dfrac{1}{2}D_{2222}(a_{22}\otimes a_{22}),
\end{align*}
\begin{align*}
&R\Delta(e_2)\\
=&\sum_{i,j}A_{ij}(e_i\otimes e_j)(e_1\otimes _2)+\sum_{i,j}A_{ij}(e_i\otimes e_j)(e_2\otimes e_1)\\
+&\sum_{i,j}A_{ij}(e_i\otimes e_j)(e_3\otimes e_4)+\sum_{i,j}A_{ij}(e_i\otimes e_j)(e_4\otimes e_3)\\
+&\dfrac{1}{2}\sum_{i,j,k,l}D_{ijkl}(a_{ij}\otimes a_{kl})(a_{11}\otimes a_{22})+\dfrac{1}{2}\sum_{i,j,k,l}D_{ijkl}(a_{ij}\otimes a_{kl})(a_{22}\otimes a_{11})\\
+&\dfrac{\sqrt{-1}}{2}\sum_{i,j,k,l}D_{ijkl}(a_{ij}\otimes a_{kl})(a_{21}\otimes a_{12})
-\dfrac{\sqrt{-1}}{2}\sum_{i,j,k,l}D_{ijkl}(a_{ij}\otimes a_{kl})(a_{12}\otimes a_{21})\\
=&A_{12}(e_1\otimes e_2)+A_{21}(e_2\otimes e_1)+A_{34}(e_3\otimes e_4)+A_{43}(e_4\otimes e_3)\\
+&\dfrac{1}{2}D_{1112}(a_{11}\otimes a_{12})+\dfrac{1}{2}D_{1122}(a_{11}\otimes a_{22})+\dfrac{1}{2}D_{2112}(a_{21}\otimes a_{12})\\
+&\dfrac{1}{2}D_{2122}(a_{21}\otimes a_{22})+\dfrac{1}{2}D_{1211}(a_{12}\otimes a_{11})+\dfrac{1}{2}D_{1221}(a_{12}\otimes a_{21})\\
+&\dfrac{1}{2}D_{2211}(a_{22}\otimes a_{11})+\dfrac{1}{2}D_{2221}(a_{22}\otimes a_{21})+\dfrac{\sqrt{-1}}{2}D_{1211}(a_{11}\otimes a_{12})\\
+&+\dfrac{\sqrt{-1}}{2}D_{1221}(a_{11}\otimes a_{22})+\dfrac{\sqrt{-1}}{2}D_{2211}(a_{21}\otimes a_{12})+\dfrac{\sqrt{-1}}{2}D_{2221}(a_{21}\otimes a_{22})\\
-&\dfrac{\sqrt{-1}}{2}D_{1112}(a_{12}\otimes a_{11})-\dfrac{\sqrt{-1}}{2}D_{1122}(a_{12}\otimes a_{21})-\dfrac{\sqrt{-1}}{2}D_{2112}(a_{22}\otimes a_{11})\\
-&\dfrac{\sqrt{-1}}{2}D_{2122}(a_{22}\otimes a_{21}),
\end{align*}
\begin{align*}
&\Delta^{\text{op}}(e_2)R\\
=&\sum_{i,j}A_{ij}(e_{2}\otimes e_{1})(e_{i}\otimes e_{j})+\sum_{i,j}A_{ij}(e_{1}\otimes e_{2})(e_{i}\otimes e_{j})\\
+&\sum_{i,j}A_{ij}(e_{4}\otimes e_{3})(e_{i}\otimes e_{j})+\sum_{i,j}A_{ij}(e_{3}\otimes e_{4})(e_{i}\otimes e_{j})\\
+&\dfrac{1}{2}\sum_{i,j,k,l}D_{ijkl}(a_{22}\otimes a_{11})(a_{ij}\otimes a_{kl})
+\dfrac{1}{2}\sum_{i,j,k,l}D_{ijkl}(a_{11}\otimes a_{22})(a_{ij}\otimes a_{kl})\\
+&\dfrac{\sqrt{-1}}{2}\sum_{i,j,k,l}D_{ijkl}(a_{12}\otimes a_{21})(a_{ij}\otimes a_{kl})
-\dfrac{\sqrt{-1}}{2}\sum_{i,j,k,l}D_{ijkl}(a_{21}\otimes a_{12})(a_{ij}\otimes a_{kl})\\
=&A_{21}(e_2\otimes e_1)+A_{12}(e_1\otimes e_2)+A_{43}(e_4\otimes e_3)+A_{34}(e_3\otimes e_4)\\
+&\dfrac{1}{2}D_{2111}(a_{21}\otimes a_{11})+\dfrac{1}{2}D_{2112}(a_{21}\otimes a_{12})+\dfrac{1}{2}D_{2211}(a_{22}\otimes a_{11})\\
+&\dfrac{1}{2}D_{2212}(a_{22}\otimes a_{12})+\dfrac{1}{2}D_{1121}(a_{11}\otimes a_{21})+\dfrac{1}{2}D_{1122}(a_{11}\otimes a_{22})\\
+&\dfrac{1}{2}D_{1221}(a_{12}\otimes a_{21})+\dfrac{1}{2}D_{1222}(a_{12}\otimes a_{22})+\dfrac{\sqrt{-1}}{2}D_{2111}(a_{11}\otimes a_{21})\\
+&\dfrac{\sqrt{-1}}{2}D_{2112}(a_{11}\otimes a_{22})+\dfrac{\sqrt{-1}}{2}D_{2211}(a_{12}\otimes a_{21})+\dfrac{\sqrt{-1}}{2}D_{2212}(a_{12}\otimes a_{22})\\
-&\dfrac{\sqrt{-1}}{2}D_{1121}(a_{21}\otimes a_{11})-\dfrac{\sqrt{-1}}{2}D_{1122}(a_{21}\otimes a_{12})-\dfrac{\sqrt{-1}}{2}D_{1221}(a_{22}\otimes a_{11})\\
-&\dfrac{\sqrt{-1}}{2}D_{1222}(a_{22}\otimes a_{12}),
\end{align*}
\begin{align*}
&R\Delta(e_3)\\
=&\sum_{i,j}A_{ij}(e_i\otimes e_j)(e_1\otimes e_3)+\sum_{i,j}A_{ij}(e_i\otimes e_j)(e_3\otimes e_1)\\
+&\sum_{i,j}A_{ij}(e_i\otimes e_j)(e_2\otimes e_4)+\sum_{i,j}A_{ij}(e_i\otimes e_j)(e_4\otimes e_2)\\
+&\dfrac{1}{2}\sum_{i,j,k,l}D_{ijkl}(a_{ij}\otimes a_{kl})(a_{11}\otimes a_{22})
+\dfrac{1}{2}\sum_{i,j,k,l}D_{ijkl}(a_{ij}\otimes a_{kl})(a_{22}\otimes a_{11})\\
-&\dfrac{\sqrt{-1}}{2}\sum_{i,j,k,l}D_{ijkl}(a_{ij}\otimes a_{kl})(a_{21}\otimes a_{12})
+\dfrac{\sqrt{-1}}{2}\sum_{i,j,k,l}D_{ijkl}(a_{ij}\otimes a_{kl})(a_{12}\otimes a_{21})
\end{align*}
\begin{align*}
=&A_{13}(e_1\otimes e_3)+A_{31}(e_3\otimes e_1)+A_{24}(e_2\otimes e_4)+A_{42}(e_4\otimes e_2)\\
+&\dfrac{1}{2}D_{1112}(a_{11}\otimes a_{12})+\dfrac{1}{2}D_{1122}(a_{11}\otimes a_{22})+\dfrac{1}{2}D_{2112}(a_{21}\otimes a_{12})\\
+&\dfrac{1}{2}D_{2122}(a_{21}\otimes a_{22})+\dfrac{1}{2}D_{1211}(a_{12}\otimes a_{11})+\dfrac{1}{2}D_{1221}(a_{12}\otimes a_{21})\\
+&\dfrac{1}{2}D_{2211}(a_{22}\otimes a_{11})+\dfrac{1}{2}D_{2221}(a_{22}\otimes a_{21})-\dfrac{\sqrt{-1}}{2}D_{1211}(a_{11}\otimes a_{12})\\
-&\dfrac{\sqrt{-1}}{2}D_{1221}(a_{11}\otimes a_{22})-\dfrac{\sqrt{-1}}{2}D_{2211}(a_{21}\otimes a_{12})-\dfrac{\sqrt{-1}}{2}D_{2221}(a_{21}\otimes a_{22})\\
+&\dfrac{\sqrt{-1}}{2}D_{1112}(a_{12}\otimes a_{11})+\dfrac{\sqrt{-1}}{2}D_{1122}(a_{12}\otimes a_{21})+\dfrac{\sqrt{-1}}{2}D_{2112}(a_{22}\otimes a_{11})\\
+&\dfrac{\sqrt{-1}}{2}D_{2122}(a_{22}\otimes a_{21}),\\
&\Delta^{\text{op}}(e_3)R\\
=&\sum_{i,j}A_{ij}(e_3\otimes e_1)(e_i\otimes e_j)+\sum_{i,j}A_{ij}(e_1\otimes e_3)(e_i\otimes e_j)\\
+&\sum_{i,j}A_{ij}(e_4\otimes e_2)(e_i\otimes e_j)+\sum_{i,j}A_{ij}(e_2\otimes e_4)(e_i\otimes e_j)\\
+&\dfrac{1}{2}\sum_{i,j,k,l}D_{ijkl}(a_{22}\otimes a_{11})(a_{ij}\otimes a_{kl})+\dfrac{1}{2}\sum_{i,j,k,l}D_{ijkl}(a_{11}\otimes a_{22})(a_{ij}\otimes a_{kl})\\
-&\dfrac{\sqrt{-1}}{2}\sum_{i,j,k,l}D_{ijkl}(a_{12}\otimes a_{21})(a_{ij}\otimes a_{kl})+\dfrac{\sqrt{-1}}{2}\sum_{i,j,k,l}D_{ijkl}(a_{21}\otimes a_{12})(a_{ij}\otimes a_{kl})\\
=&A_{31}(e_3\otimes e_1)+A_{13}(e_1\otimes e_3)+A_{42}(e_4\otimes e_2)+A_{24}(e_2\otimes e_4)\\
+&\dfrac{1}{2}D_{2111}(a_{21}\otimes a_{11})+\dfrac{1}{2}D_{2112}(a_{21}\otimes a_{12})+\dfrac{1}{2}D_{2211}(a_{22}\otimes a_{11})\\
+&\dfrac{1}{2}D_{2212}(a_{22}\otimes a_{12})+\dfrac{1}{2}D_{1121}(a_{11}\otimes a_{21})+\dfrac{1}{2}D_{1122}(a_{11}\otimes a_{22})\\
+&\dfrac{1}{2}D_{1221}(a_{12}\otimes a_{21})+\dfrac{1}{2}D_{1222}(a_{12}\otimes a_{22})-\dfrac{\sqrt{-1}}{2}D_{2111}(a_{11}\otimes a_{21})\\
-&\dfrac{\sqrt{-1}}{2}D_{2112}(a_{11}\otimes a_{22})-\dfrac{\sqrt{-1}}{2}D_{2211}(a_{12}\otimes a_{21})-\dfrac{\sqrt{-1}}{2}D_{2212}(a_{12}\otimes a_{22})\\
+&\dfrac{\sqrt{-1}}{2}D_{1121}(a_{21}\otimes a_{11})+\dfrac{\sqrt{-1}}{2}D_{1122}(a_{21}\otimes a_{12})+\dfrac{\sqrt{-1}}{2}D_{1221}(a_{22}\otimes a_{11})\\
+&\dfrac{\sqrt{-1}}{2}D_{1222}(a_{22}\otimes a_{12}),\\
&R\Delta(e_4)\\
=&\sum_{i,j}A_{ij}(e_i\otimes e_j)(e_1\otimes e_4)+\sum_{i,j}A_{ij}(e_i\otimes e_j)(e_4\otimes e_1)\\
+&\sum_{i,j}A_{ij}(e_i\otimes e_j)(e_2\otimes e_3)+\sum_{i,j}A_{ij}(e_i\otimes e_j)(e_3\otimes e_2)
\end{align*}
\begin{align*}
+&\dfrac{1}{2}\sum_{i,j,k,l}D_{ijkl}(a_{ij}\otimes a_{kl})(a_{11}\otimes a_{11})+\dfrac{1}{2}\sum_{i,j,k,l}D_{ijkl}(a_{ij}\otimes a_{kl})(a_{22}\otimes a_{22})\\
-&\dfrac{1}{2}\sum_{i,j,k,l}D_{ijkl}(a_{ij}\otimes a_{kl})(a_{12}\otimes a_{12})-\dfrac{1}{2}\sum_{i,j,k,l}D_{ijkl}(a_{ij}\otimes a_{kl})(a_{21}\otimes a_{21})\\
=&A_{14}e_1\otimes e_4+A_{41}e_4\otimes e_1+A_{23}e_2\otimes e_3+A_{32}e_3\otimes e_2\\
+&\dfrac{1}{2}D_{1111}(a_{11}\otimes a_{11})+\dfrac{1}{2}D_{1121}(a_{11}\otimes a_{21})+\dfrac{1}{2}D_{2111}(a_{21}\otimes a_{11})\\
+&\dfrac{1}{2}D_{2121}(a_{21}\otimes a_{21})+\dfrac{1}{2}D_{1212}(a_{12}\otimes a_{12})+\dfrac{1}{2}D_{1222}(a_{12}\otimes a_{22})\\
+&\dfrac{1}{2}D_{2212}(a_{22}\otimes a_{12})+\dfrac{1}{2}D_{2222}(a_{22}\otimes a_{22})-\dfrac{1}{2}D_{1111}(a_{12}\otimes a_{12})\\
-&\dfrac{1}{2}D_{1121}(a_{12}\otimes a_{22})-\dfrac{1}{2}D_{2111}(a_{22}\otimes a_{12})-\dfrac{1}{2}D_{2121}(a_{22}\otimes a_{22})\\
-&\dfrac{1}{2}D_{1212}(a_{11}\otimes a_{11})-\dfrac{1}{2}D_{1222}(a_{11}\otimes a_{21})-\dfrac{1}{2}D_{2212}(a_{21}\otimes a_{11})\\
-&\dfrac{1}{2}D_{2222}(a_{21}\otimes a_{21}),\\
&\Delta^{\text{op}}(e_4)R\\
=&\sum_{i,j}A_{ij}(e_4\otimes e_1)(e_i\otimes e_j)+\sum_{i,j}A_{ij}(e_1\otimes e_4)(e_i\otimes e_j)\\
+&\sum_{i,j}A_{ij}(e_3\otimes e_2)(e_i\otimes e_j)+\sum_{i,j}A_{ij}(e_2\otimes e_3)(e_i\otimes e_j)\\
+&\dfrac{1}{2}\sum_{i,j,k,l}D_{ijkl}(a_{11}\otimes a_{11})(a_{ij}\otimes a_{kl})+\dfrac{1}{2}\sum_{i,j,k,l}D_{ijkl}(a_{22}\otimes a_{22})(a_{ij}\otimes a_{kl})\\
-&\dfrac{1}{2}\sum_{i,j,k,l}D_{ijkl}(a_{12}\otimes a_{12})(a_{ij}\otimes a_{kl})-\dfrac{1}{2}\sum_{i,j,k,l}D_{ijkl}(a_{21}\otimes a_{21})(a_{ij}\otimes a_{kl})\\
=&A_{41}e_4\otimes e_1+A_{14}e_1\otimes e_4+A_{32}e_3\otimes e_2+A_{23}e_2\otimes e_3\\
+&\dfrac{1}{2}D_{1111}(a_{11}\otimes a_{11})+\dfrac{1}{2}D_{1112}(a_{11}\otimes a_{12})+\dfrac{1}{2}D_{1211}(a_{12}\otimes a_{11})\\
+&\dfrac{1}{2}D_{1212}(a_{12}\otimes a_{12})+\dfrac{1}{2}D_{2121}(a_{21}\otimes a_{21})+\dfrac{1}{2}D_{2122}(a_{21}\otimes a_{22})\\
+&\dfrac{1}{2}D_{2221}(a_{22}\otimes a_{21})+\dfrac{1}{2}D_{2222}(a_{22}\otimes a_{22})-\dfrac{1}{2}D_{2121}(a_{11}\otimes a_{11})\\
-&\dfrac{1}{2}D_{2122}(a_{11}\otimes a_{12})-\dfrac{1}{2}D_{2221}(a_{12}\otimes a_{11})-\dfrac{1}{2}D_{2222}(a_{12}\otimes a_{12})\\
-&\dfrac{1}{2}D_{1111}(a_{21}\otimes a_{21})-\dfrac{1}{2}D_{1112}(a_{21}\otimes a_{22})-\dfrac{1}{2}D_{1211}(a_{22}\otimes a_{21})\\
-&\dfrac{1}{2}D_{1212}(a_{22}\otimes a_{22}),
\end{align*}
\begin{align*}
&R\Delta(a_{11})\\
=&\sum_{j,k}B_{1jk}(e_1\otimes a_{jk}a_{11})+\sum_{i,j}C_{ij1}(a_{ij}a_{11}\otimes e_1)+\sum_{j,k}B_{2jk}(e_2\otimes a_{jk}a_{22})\\
+&\sum_{i,j}C_{ij2}(a_{ij}a_{22}\otimes e_2)+\sum_{j,k}B_{3jk}(e_3\otimes a_{jk}a_{22})+\sum_{i,j}C_{ij3}(a_{ij}a_{22}\otimes e_3)\\
+&\sum_{j,k}B_{4jk}(e_4\otimes a_{jk}a_{11})+\sum_{i,j}C_{ij4}(a_{ij}a_{11}\otimes e_4)\\
=&B_{111}(e_1\otimes a_{11})+B_{121}(e_1\otimes a_{21})+C_{111}(a_{11}\otimes e_1)+C_{211}(a_{21}\otimes e_1)\\
+&B_{212}(e_2\otimes a_{12})+B_{222}(e_2\otimes a_{22})+C_{121}(a_{12}\otimes e_2)+C_{221}(a_{22}\otimes e_2)\\
+&B_{312}(e_3\otimes a_{12})+B_{322}(e_3\otimes a_{22})+C_{123}(a_{12}\otimes e_3)+C_{223}(a_{22}\otimes e_3)\\
+&B_{411}(e_4\otimes a_{11})+B_{421}(e_4\otimes a_{21})+C_{114}(a_{11}\otimes e_4)+C_{214}(a_{21}\otimes e_4),\\
&\Delta^{\text{op}}(a_{11})R\\
=&\sum_{i,j}C_{ij1}(a_{11}a_{ij}\otimes e_1)+\sum_{j,k}B_{1jk}(e_1\otimes a_{11}a_{jk})+\sum_{i,j}C_{ij2}(a_{22}a_{ij}\otimes e_2)\\
+&\sum_{j,k}B_{2jk}(e_2\otimes a_{22}a_{jk})+\sum_{i,j}C_{ij3}(a_{22}a_{ij}\otimes e_3)+\sum_{j,k}B_{3jk}(e_3\otimes a_{22}a_{jk})\\
+&\sum_{i,j}C_{ij4}(a_{11}a_{ij}\otimes e_4)+\sum_{j,k}(e_4\otimes a_{11}a_{jk})\\
=&C_{111}(a_{11}\otimes e_1)+C_{121}(a_{12}\otimes e_1)+B_{111}(e_1\otimes a_{11})+B_{112}(e_1\otimes a_{12})\\
+&C_{212}(a_{21}\otimes e_2)+C_{222}(a_{22}\otimes e_2)+B_{221}(e_2\otimes a_{21})+B_{222}(e_2\otimes a_{22})\\
+&C_{213}(a_{21}\otimes e_3)+C_{223}(a_{22}\otimes e_3)+B_{321}(e_3\otimes a_{21})+B_{322}(e_3\otimes a_{22})\\
+&C_{114}(a_{11}\otimes e_4)+C_{124}(a_{12}\otimes e_4)+B_{411}(e_4\otimes a_{11})+B_{412}(e_4\otimes a_{12}),\\
&R\Delta(a_{12})\\
=&\sum_{j,k}B_{1jk}(e_1\otimes a_{jk}a_{12})+\sum_{i,j}C_{ij1}(a_{ij}a_{12}\otimes e_1)+\sqrt{-1}\sum_{j,k}B_{2jk}(e_2\otimes a_{jk}a_{21})\\
-&\sqrt{-1}\sum_{i,j}C_{ij2}(a_{ij}a_{21}\otimes e_2)-\sqrt{-1}\sum_{j,k}B_{3jk}(e_3\otimes a_{jk}a_{21})\\
+&\sqrt{-1}\sum_{i,j}C_{ij3}(a_{ij}a_{21}\otimes e_3)-\sum_{j,k}B_{4jk}(e_4\otimes a_{jk}a_{12})-\sum_{i,j}C_{ij4}(a_{ij}a_{12}\otimes e_4)\\
=&B_{111}(e_{1}\otimes a_{12})+B_{121}(e_1\otimes a_{22})+C_{111}(a_{12}\otimes e_1)+C_{211}(a_{22}\otimes e_1)\\
+&\sqrt{-1}B_{212}(e_2\otimes a_{11})+\sqrt{-1}B_{222}(e_2\otimes a_{21})-\sqrt{-1}C_{122}(a_{11}\otimes e_2)\\
-&\sqrt{-1}C_{222}(a_{21}\otimes e_2)-\sqrt{-1}B_{312}(e_3\otimes a_{11})-\sqrt{-1}B_{322}(e_3\otimes a_{21})\\
+&\sqrt{-1}C_{123}(a_{11}\otimes e_3)+\sqrt{-1}C_{223}(a_{21}\otimes e_3)\\
-&B_{411}(e_4\otimes a_{12})-B_{421}(e_4\otimes a_{22})-C_{114}(a_{12}\otimes e_4)-C_{214}(a_{22}\otimes e_4),
\end{align*}
\begin{align*}
&\Delta^{\text{op}}(a_{12})R\\
=&\sum_{i,j}C_{ij1}(a_{12}a_{ij}\otimes e_1)+\sum_{j,k}B_{1jk}(e_1\otimes a_{12}a_{jk})+\sqrt{-1}\sum_{i,j}C_{ij2}(a_{21}a_{ij}\otimes e_2)\\
-&\sqrt{-1}\sum_{j,k}B_{2jk}(e_2\otimes a_{21}a_{jk})-\sqrt{-1}\sum_{i,j}C_{ij3}(a_{21}a_{ij}\otimes e_3)\\
+&\sqrt{-1}\sum_{j,k}B_{3jk}(e_3\otimes a_{21}a_{jk})-\sum_{i,j}C_{ij4}(a_{12}a_{ij}\otimes e_4)-\sum_{j,k}B_{4jk}(e_4\otimes a_{12}a_{jk})\\
=&C_{211}(a_{11}\otimes e_1)+C_{211}(a_{12}\otimes e_1)+B_{121}(e_1\otimes a_{11})+B_{122}(e_1\otimes a_{12})\\
+&\sqrt{-1}C_{112}(a_{21}\otimes e_2)+\sqrt{-1}C_{122}(a_{22}\otimes e_2)-\sqrt{-1}B_{211}(e_2\otimes a_{21})\\
-&\sqrt{-1}B_{212}(e_2\otimes a_{22})-\sqrt{-1}C_{113}(a_{21}\otimes e_3)-\sqrt{-1}C_{123}(a_{22}\otimes e_3)\\
+&\sqrt{-1}B_{311}(e_3\otimes a_{21})+\sqrt{-1}B_{312}(e_3\otimes a_{22})\\
-&C_{214}(a_{11}\otimes e_4)-C_{224}(a_{12}\otimes e_4)-B_{421}(e_4\otimes a_{11})-B_{422}(e_4\otimes a_{12}),\\
&R\Delta(a_{21})\\
=&\sum_{j,k}B_{1jk}(e_1\otimes a_{jk}a_{21})+\sum_{i,j}C_{ij1}(a_{ij}a_{21}\otimes e_1)-\sqrt{-1}\sum_{j,k}B_{2jk}(e_2\otimes a_{jk}a_{12})\\
+&\sqrt{-1}\sum_{i,j}C_{ij2}(a_{ij}a_{12}\otimes e_2)+\sqrt{-1}\sum_{j,k}B_{3jk}(e_3\otimes a_{jk}a_{12})\\
-&\sqrt{-1}\sum_{i,j}C_{ij3}(a_{ij}a_{12}\otimes e_3)-\sum_{j,k}B_{4jk}(e_4\otimes a_{jk}a_{21})-\sum_{i,j}C_{ij4}(a_{ij}a_{21}\otimes e_4)\\
=&B_{112}(e_1\otimes a_{11})+B_{122}(e_1\otimes a_{21})+C_{121}(a_{11}\otimes e_1)+C_{221}(a_{21}\otimes e_1)\\
-&\sqrt{-1}B_{211}(e_2\otimes a_{12})-\sqrt{-1}B_{221}(e_2\otimes a_{22})+\sqrt{-1}C_{112}(a_{12}\otimes e_2)\\
+&\sqrt{-1}C_{212}(a_{22}\otimes e_2)+\sqrt{-1}B_{311}(e_3\otimes a_{12})+\sqrt{-1}B_{321}(e_3\otimes a_{22})\\
-&\sqrt{-1}C_{113}(a_{12}\otimes e_3)-\sqrt{-1}C_{213}(a_{22}\otimes e_{3})\\
-&B_{412}(e_4\otimes a_{11})-B_{422}(e_4\otimes a_{21})-C_{124}(a_{11}\otimes e_4)-C_{224}(a_{21}\otimes e_4),\\
&\Delta^{\text{op}}(a_{21})R\\
=&\sum_{i,j}C_{ij1}(a_{21}a_{ij}\otimes e_1)+\sum_{j,k}B_{1jk}(e_1\otimes a_{21}a_{jk})-\sqrt{-1}\sum_{i,j}C_{ij2}(a_{12}a_{ij}\otimes e_2)\\
+&\sqrt{-1}\sum_{j,k}B_{2jk}(e_2\otimes a_{12}a_{jk})+\sqrt{-1}\sum_{i,j}C_{ij3}(a_{12}a_{ij}\otimes e_3)\\
-&\sqrt{-1}\sum_{j,k}B_{3jk}(e_3\otimes a_{12}a_{jk})-\sum_{i,j}C_{ij4}(a_{21}a_{ij}\otimes e_4)-\sum_{j,k}B_{4jk}(e_4\otimes a_{21}a_{jk})\\
=&C_{111}(a_{21}\otimes e_1)+C_{121}(a_{22}\otimes e_1)+B_{111}(e_1\otimes a_{21})+B_{112}(e_1\otimes a_{22})\\
-&\sqrt{-1}C_{212}(a_{11}\otimes e_2)-\sqrt{-1}C_{222}(a_{12}\otimes e_2)+\sqrt{-1}B_{221}(e_2\otimes a_{11})\\
+&\sqrt{-1}B_{222}(e_2\otimes a_{12})+\sqrt{-1}C_{213}(a_{11}\otimes e_3)+\sqrt{-1}C_{223}(a_{12}\otimes e_3)\\
-&\sqrt{-1}B_{321}(e_3\otimes a_{11})-\sqrt{-1}B_{322}(e_3\otimes a_{12})\\
-&C_{114}(a_{21}\otimes e_4)-C_{124}(a_{22}\otimes e_4)-B_{411}(e_4\otimes a_{21})-B_{412}(e_4\otimes a_{22}),
\end{align*}
\begin{align*}
&R\Delta(a_{22})\\
=&\sum_{j,k}B_{1jk}(e_1\otimes a_{jk}a_{22})+\sum_{i,j}C_{i,j}(a_{ij}a_{22}\otimes e_1)+\sum_{j,k}B_{2jk}(e_2\otimes a_{jk}a_{11})\\
+&\sum_{i,j}C_{ij2}(a_{ij}a_{11}\otimes e_2)+\sum_{j,k}B_{3jk}(e_3\otimes a_{jk}a_{11})+\sum_{i,j}C_{ij3}(a_{ij}a_{11}\otimes e_3)\\
+&\sum_{j,k}B_{4jk}(e_4\otimes a_{jk}a_{22})+\sum_{i,j}C_{ij4}(a_{ij}a_{22}\otimes e_4),\\
=&B_{112}(e_1\otimes a_{12})+B_{122}(e_1\otimes a_{22})+C_{121}(a_{12}\otimes e_1)+C_{221}(a_{22}\otimes e_1)\\
+&B_{211}(e_2\otimes a_{11})+B_{221}(e_2\otimes a_{21})+C_{112}(a_{11}\otimes e_2)+C_{212}(a_{21}\otimes e_2)\\
+&B_{311}(e_3\otimes a_{11})+B_{321}(e_3\otimes a_{21})+C_{113}(a_{11}\otimes e_3)+C_{213}(a_{21}\otimes e_3)\\
+&B_{412}(e_4\otimes a_{12})+B_{422}(e_4\otimes a_{22})+C_{124}(a_{12}\otimes e_4)+C_{224}(a_{22}\otimes e_4)\\
\text{and}&\\
&\Delta^{\text{op}}(a_{22})R\\
=&\sum_{i,j}C_{ij1}(a_{22}a_{ij}\otimes e_{1})+\sum_{j,k}B_{1jk}(e_1\otimes a_{22}a_{jk})+\sum_{i,j}C_{ij2}(a_{11}a_{ij}\otimes e_2)\\
+&\sum_{j,k}B_{2jk}(e_2\otimes a_{11}a_{jk})+\sum_{i,j}C_{ij3}(a_{11}a_{ij}\otimes e_3)+\sum_{j,k}B_{3jk}(e_3\otimes a_{11}a_{jk})\\
+&\sum_{i,j}C_{ij4}(a_{22}a_{ij}\otimes e_4)+\sum_{j,k}B_{4jk}(e_4\otimes a_{22}a_{jk})\\
=&C_{211}(a_{21}\otimes e_1)+C_{221}(a_{22}\otimes e_1)+B_{121}(e_1\otimes a_{21})+B_{122}(e_1\otimes a_{22})\\
+&C_{112}(a_{11}\otimes e_2)+C_{122}(a_{12}\otimes e_2)+B_{211}(e_2\otimes a_{11})+B_{212}(e_2\otimes a_{12})\\
+&C_{113}(a_{11}\otimes e_3)+C_{123}(a_{12}\otimes e_3)+B_{311}(e_3\otimes a_{11})+B_{312}(e_3\otimes a_{12})\\
+&C_{214}(a_{21}\otimes e_4)+C_{224}(a_{22}\otimes e_4)+B_{421}(e_4\otimes a_{21})+B_{422}(e_4\otimes a_{22}).
\end{align*}
We compare the coefficients.
\begin{align*}
&D_{1111}+D_{1212}=D_{1111}+D_{2121}, D_{1121}+D_{1222}=0, D_{1112}+D_{2122}=0,\\
&D_{1211}+D_{2221}=0, D_{1111}+D_{1212}=D_{1212}+D_{2222}, D_{1121}+D_{1222}=0,\\
&D_{2111}+D_{2212}=0, D_{2121}+D_{2222}=D_{1111}+D_{2121}, D_{2111}+D_{2212}=0,\\
&D_{1112}+D_{2122}=0, D_{1211}+D_{2221}=0, D_{2121}+D_{2222}=D_{1212}+D_{2222},\\
&\qquad\qquad D_{1112}+\sqrt{-1}D_{1211}=0, D_{1121}+\sqrt{-1}D_{2111}=0,\\
&\qquad\qquad D_{1122}+\sqrt{-1}D_{1221}=D_{1122}+\sqrt{-1}D_{2112},\\
&\qquad\qquad D_{1221}-\sqrt{-1}D_{1122}=D_{1221}+\sqrt{-1}D_{2211},\\ 
&\qquad\qquad D_{1211}-\sqrt{-1}D_{1112}=0, D_{1222}+\sqrt{-1}D_{2212}=0,\\
&\qquad\qquad D_{2112}+\sqrt{-1}D_{2211}=D_{2112}-\sqrt{-1}D_{1122},\\
&\qquad\qquad D_{2111}-\sqrt{-1}D_{1121}=0, D_{2122}+\sqrt{-1}D_{2221}=0,\\
&\qquad\qquad D_{2211}-\sqrt{-1}D_{2112}=D_{2211}-\sqrt{-1}D_{1221},\\
&\qquad\qquad D_{2221}-\sqrt{-1}D_{2122}=0, D_{2212}-\sqrt{-1}D_{1222}=0,
\end{align*}
\begin{align*}
&\qquad\qquad D_{1112}-\sqrt{-1}D_{1211}=0, D_{1121}-\sqrt{-1}D_{2111}=0,\\
&\qquad\qquad D_{1122}-\sqrt{-1}D_{1221}=D_{1122}-\sqrt{-1}D_{2112},\\
&\qquad\qquad D_{1221}-\sqrt{-1}D_{2211}=D_{1221}+\sqrt{-1}D_{1122},\\
&\qquad\qquad D_{1211}+\sqrt{-1}D_{1112}=0, D_{1222}-\sqrt{-1}D_{2212}=0,\\
&\qquad\qquad D_{2112}-\sqrt{-1}D_{2211}=D_{2112}+\sqrt{-1}D_{1122},\\
&\qquad\qquad D_{2122}-\sqrt{-1}D_{2221}=0, D_{2111}+\sqrt{-1}D_{1121}=0,\\
&\qquad\qquad D_{2211}+\sqrt{-1}D_{2112}=D_{2211}+\sqrt{-1}D_{1221},\\
&\qquad\qquad D_{2221}+\sqrt{-1}D_{2122}=0, D_{2212}+\sqrt{-1}D_{1222}=0,\\
&D_{1111}-D_{1212}=D_{1111}-D_{2121}, D_{1121}-D_{1222}=0, D_{1112}-D_{2122}=0,\\
&D_{1212}-D_{1111}=D_{1212}-D_{2222}, D_{1222}-D_{1121}=0, D_{1211}-D_{2221}=0,\\
&D_{2121}-D_{2222}=D_{2121}-D_{1111}, D_{2111}-D_{2212}=0, D_{2122}-D_{1112}=0,\\
&D_{2212}-D_{2111}=0, D_{2221}-D_{1211}=0, D_{2222}-D_{2121}=D_{2222}-D_{1212},\\
&B_{111}=B_{111}, B_{121}=0, B_{112}=0, C_{111}=C_{111}, C_{211}=0, C_{121}=0,\\
&B_{212}=0, B_{221}=0, B_{222}=B_{222}, C_{121}=0, C_{212}=0, C_{221}=C_{222},\\
&C_{123}=0, C_{213}=0, C_{223}=C_{223}, B_{321}=0, B_{312}=0, B_{322}=B_{322},\\
&B_{411}=B_{411}, B_{421}=0, B_{412}=0, C_{114}=C_{114}, C_{214}=0, C_{124}=0,\\
&B_{111}=B_{122}, B_{121}=0, C_{111}=C_{221}, C_{211}=0, B_{212}=0, B_{222}=-B_{211},\\
&C_{122}=0, -C_{222}=C_{112}, B_{312}=0, -B_{322}=B_{311}, C_{123}=0, C_{223}=-C_{113},\\
&B_{411}=B_{422}, B_{421}=0, C_{114}=C_{224}, C_{214}=0, C_{211}=0, B_{121}=0,\\
&C_{122}=0, B_{212}=0, C_{123}=0, B_{312}=0, C_{214}=0, B_{421}=0,\\
&B_{112}=0, B_{122}=B_{111}, C_{121}=0, C_{221}=C_{111}, -B_{211}=B_{222}, B_{221}=0,\\
&C_{112}=-C_{222}, C_{212}=0, B_{311}=-B_{322}, B_{321}=0, -C_{113}=C_{223}, C_{213}=0,\\
&B_{412}=0, -B_{422}=-B_{411},  C_{124}=0, -C_{224}=-C_{114}, B_{112}=0, C_{121}=0, \\
&B_{221}=0, C_{212}=0, B_{321}=0, C_{213}=0, B_{412}=0, C_{124}=0,\\
&B_{112}=0, B_{122}=B_{122}, C_{121}=0, C_{221}=C_{221}, B_{211}=B_{211}, B_{221}=0,\\
&C_{112}=C_{112}, C_{212}=0, B_{311}=B_{311}, B_{321}=0, C_{113}=C_{113}, C_{213}=0,\\
&B_{412}=0, B_{422}=B_{422}, C_{124}=0, C_{224}=C_{224}, B_{121}=0, C_{211}=0, B_{212}=0,\\
&C_{122}=0, B_{312}=0, C_{123}=0, B_{421}=0, C_{214}=0.
\end{align*}
Thus $R$ must be of the form
\begin{align*}
R=&A_{11}e_{1}\otimes e_{1}+A_{12}e_{1}\otimes e_{2}+A_{13}e_{1}\otimes e_{3}+A_{14}e_{1}\otimes e_{4}\\
+&A_{21}e_{2}\otimes e_{1}+A_{22}e_{2}\otimes e_{2}+A_{23}e_{2}\otimes e_{3}+A_{24}e_{2}\otimes e_{4}\\
+&A_{31}e_{3}\otimes e_{1}+A_{32}e_{3}\otimes e_{2}+A_{33}e_{3}\otimes e_{3}+A_{34}e_{3}\otimes e_{4}\\
+&A_{41}e_{4}\otimes e_{1}+A_{42}e_{4}\otimes e_{2}+A_{43}e_{4}\otimes e_{3}+A_{44}e_{4}\otimes e_{4}\\
+&B_1e_1\otimes a_{11}+B_1e_1\otimes a_{22}+B_2e_2\otimes a_{11}-B_2e_2\otimes a_{22}\\
+&B_3e_3\otimes a_{11}-B_3e_3\otimes a_{22}+B_4e_4\otimes a_{11}+B_4e_4\otimes a_{22}\\
+&C_1a_{11}\otimes e_1+C_1 a_{22}\otimes e_1+C_2a_{11}\otimes e_2-C_2a_{22}\otimes e_2\\
+&C_3a_{11}\otimes e_3-C_3a_{22}\otimes e_3+C_4a_{11}\otimes e_4+C_4a_{22}\otimes e_4\\
+&D_{1111}a_{11}\otimes a_{11}+D_{1122}a_{11}\otimes a_{22}+D_{1212}a_{12}\otimes a_{12}+D_{1221}a_{12}\otimes a_{21}\\+&D_{1221}a_{21}\otimes a_{12}+D_{1212}a_{21}\otimes a_{21}-D_{1122}a_{22}\otimes a_{11}+D_{1111}a_{22}\otimes a_{22},
\end{align*}
where $B_i=B_{i11}$ and $C_{i}=C_{11i}$ $(i=1,2,3,4)$.\par
Next by \cite[Theorem V\hspace{0.1mm}I\hspace{0.1mm}I\hspace{0.1mm}I.2.4]{K} we have $(\id\otimes\epsilon)R=(\epsilon\otimes\id)R=1$ and thus $A_{1i}=A_{i1}=1$ for all $i$ and $B_1=C_1=1$. We compute $(\id\otimes\Delta)(R)$.
\begin{align*}
&(\id\otimes\Delta)(R)\\
=&A_{11}e_1\otimes e_1\otimes e_1+A_{12}e_1\otimes e_1\otimes e_2+A_{13}e_1\otimes e_1\otimes e_3+A_{14}e_1\otimes e_1\otimes e_4\\
+&B_1e_1\otimes e_1\otimes a_{11}+B_1e_1\otimes e_1\otimes a_{22}\\
+&A_{12}e_1\otimes e_2\otimes e_1+A_{11}e_1\otimes e_2\otimes e_2+A_{14}e_1\otimes e_2\otimes e_3+A_{13}e_1\otimes e_2\otimes e_4\\
+&B_1e_1\otimes e_2\otimes a_{11}+B_1e_1\otimes e_2\otimes a_{22}\\
+&A_{13}e_1\otimes e_3\otimes e_1+A_{14}e_1\otimes e_3\otimes e_2+A_{11}e_1\otimes e_3\otimes e_3+A_{12}e_1\otimes e_3\otimes e_4\\
+&B_1e_1\otimes e_3\otimes a_{11}+B_1e_1\otimes e_3\otimes a_{22}\\
+&A_{14}e_1\otimes e_4\otimes e_1+A_{13}e_1\otimes e_4\otimes e_2+A_{12}e_1\otimes e_4\otimes e_3+A_{11}e_1\otimes e_4\otimes e_4\\
+&B_1e_1\otimes e_4\otimes a_{11}+B_1e_1\otimes e_4\otimes a_{22}\\
+&B_1e_1\otimes a_{11}\otimes e_1+B_1e_1\otimes a_{11}\otimes e_2+B_1e_1\otimes a_{11}\otimes e_3+B_1e_1\otimes a_{11}\otimes e_4\\
+&\dfrac{1}{2}(A_{11}+A_{14})e_1\otimes a_{11}\otimes a_{11}+\dfrac{1}{2}(A_{12}+A_{13})e_1\otimes a_{11}\otimes a_{22}\\
+&\dfrac{1}{2}(A_{11}-A_{14})e_1\otimes a_{12}\otimes a_{12}+\dfrac{1}{2}(-\sqrt{-1}A_{12}+\sqrt{-1}A_{13})e_1\otimes a_{12}\otimes a_{21}\\
+&\dfrac{1}{2}(\sqrt{-1}A_{12}-\sqrt{-1}A_{13})e_1\otimes a_{21}\otimes a_{12}+\dfrac{1}{2}(A_{11}-A_{14})e_{1}\otimes a_{21}\otimes a_{21}\\
+&B_1e_1\otimes a_{22}\otimes e_1+B_1e_1\otimes a_{22}\otimes e_2+B_1e_1\otimes a_{22}\otimes e_3+B_1e_1\otimes a_{22}\otimes e_4\\
+&\dfrac{1}{2}(A_{12}+A_{13})e_1\otimes a_{22}\otimes a_{11}+\dfrac{1}{2}(A_{11}+A_{14})e_1\otimes a_{22}\otimes a_{22}\\
+&A_{21}e_2\otimes e_1\otimes e_1+A_{22}e_2\otimes e_1\otimes e_2+A_{23}e_2\otimes e_1\otimes e_3+A_{24}e_2\otimes e_1\otimes e_4\\
+&B_2e_2\otimes e_1\otimes a_{11}-B_2e_2\otimes e_1\otimes a_{22}
\end{align*}
\begin{align*}
+&A_{22}e_2\otimes e_2\otimes e_1+A_{21}e_2\otimes e_2\otimes e_2+A_{24}e_2\otimes e_2\otimes e_3+A_{23}e_2\otimes e_2\otimes e_4\\
-&B_2e_2\otimes e_2\otimes a_{11}-B_2e_2\otimes e_2\otimes a_{22}\\
+&A_{23}e_2\otimes e_3\otimes e_1+A_{24}e_2\otimes e_3\otimes e_2+A_{21}e_2\otimes e_3\otimes e_3+A_{22}e_2\otimes e_3\otimes e_4\\
-&B_2e_2\otimes e_3\otimes a_{11}+B_2e_2\otimes e_3\otimes a_{22}\\
+&A_{24}e_2\otimes e_4\otimes e_1+A_{23}e_2\otimes e_4\otimes e_2+A_{22}e_2\otimes e_4\otimes e_3+A_{21}e_2\otimes e_4\otimes e_4\\
+&B_2e_2\otimes e_4\otimes a_{11}-B_2e_2\otimes e_4\otimes a_{22}\\
+&B_2e_2\otimes a_{11}\otimes e_1-B_2e_2\otimes a_{11}\otimes e_2-B_2e_2\otimes a_{11}\otimes e_3+B_2e_2\otimes a_{11}\otimes e_4\\
+&\dfrac{1}{2}(A_{21}+A_{24})e_2\otimes a_{11}\otimes a_{11}+\dfrac{1}{2}(A_{22}+A_{23})e_2\otimes a_{11}\otimes a_{22}\\
+&\dfrac{1}{2}(A_{21}-A_{24})e_2\otimes a_{12}\otimes a_{12}+\dfrac{1}{2}(-\sqrt{-1}A_{22}+\sqrt{-1}A_{23})e_2\otimes a_{12}\otimes a_{21}\\
+&\dfrac{1}{2}(\sqrt{-1}A_{22}-\sqrt{-1}A_{23})e_2\otimes a_{21}\otimes a_{12}+\dfrac{1}{2}(A_{21}-A_{24})e_2\otimes a_{21}\otimes a_{21}\\
-&B_2e_2\otimes a_{22}\otimes e_1+B_2e_2\otimes a_{22}\otimes e_2+B_2e_2\otimes a_{22}\otimes e_3-B_2e_2\otimes a_{22}\otimes e_4\\
+&\dfrac{1}{2}(A_{22}+A_{23})e_2\otimes a_{22}\otimes a_{11}+\dfrac{1}{2}(A_{21}+A_{24})e_2\otimes a_{22}\otimes a_{22}\\
+&A_{31}e_3\otimes e_1\otimes e_1+A_{32}e_3\otimes e_1\otimes e_2+A_{33}e_3\otimes e_1\otimes e_3+A_{34}e_3\otimes e_1\otimes e_4\\
+&B_3e_3\otimes e_1\otimes a_{11}-B_3e_3\otimes e_1\otimes a_{22}\\
+&A_{32}e_3\otimes e_2\otimes e_1+A_{31}e_3\otimes e_2\otimes e_2+A_{34}e_3\otimes e_2\otimes e_3+A_{33}e_3\otimes e_2\otimes e_4\\
-&B_3e_3\otimes e_2\otimes a_{11}+B_3e_3\otimes e_2\otimes a_{22}\\
+&A_{33}e_3\otimes e_3\otimes e_1+A_{34}e_3\otimes e_3\otimes e_2+A_{31}e_3\otimes e_3\otimes e_3+A_{32}e_3\otimes e_3\otimes e_4\\
-&B_3e_3\otimes e_3\otimes a_{11}+B_3e_3\otimes e_3\otimes a_{22}\\
+&A_{34}e_3\otimes e_4\otimes e_1+A_{33}e_3\otimes e_4\otimes e_2+A_{32}e_3\otimes e_4\otimes e_3+A_{31}e_3\otimes e_4\otimes e_4\\
+&B_3e_3\otimes e_4\otimes a_{11}-B_3e_3\otimes e_4\otimes a_{22}\\
+&B_3e_3\otimes a_{11}\otimes e_1-B_3e_3\otimes a_{11}\otimes e_2-B_3e_3\otimes a_{11}\otimes e_3+B_3e_3\otimes a_{11}\otimes e_4\\
+&\dfrac{1}{2}(A_{31}+A_{34})e_3\otimes a_{11}\otimes a_{11}+\dfrac{1}{2}(A_{32}+A_{33})e_3\otimes a_{11}\otimes a_{22}\\
+&\dfrac{1}{2}(A_{31}-A_{34})e_3\otimes a_{12}\otimes a_{12}+\dfrac{1}{2}(-\sqrt{-1}A_{32}+\sqrt{-1}A_{33})e_3\otimes a_{12}\otimes a_{21}\\
+&\dfrac{1}{2}(\sqrt{-1}A_{32}-\sqrt{-1}A_{33})e_3\otimes a_{21}\otimes a_{12}+\dfrac{1}{2}(A_{31}-A_{34})e_3\otimes a_{21}\otimes a_{21}\\
-&B_3e_3\otimes a_{22}\otimes e_1+B_3e_3\otimes a_{22}\otimes e_2+B_3e_3\otimes a_{22}\otimes e_3-B_3e_3\otimes a_{22}\otimes e_4\\
+&\dfrac{1}{2}(A_{32}+A_{33})e_3\otimes a_{22}\otimes a_{11}+\dfrac{1}{2}(A_{31}+A_{34})e_3\otimes a_{22}\otimes a_{22}\\
+&A_{41}e_4\otimes e_1\otimes e_1+A_{42}e_4\otimes e_1\otimes e_2+A_{43}e_4\otimes e_1\otimes e_3+A_{44} e_4\otimes e_1\otimes e_4\\
+&B_4e_4\otimes e_1\otimes a_{11}+B_4e_4\otimes e_1\otimes a_{22}\\
+&A_{42}e_4\otimes e_2\otimes e_1+A_{41}e_4\otimes e_2\otimes e_2+A_{44}e_4\otimes e_2\otimes e_3+A_{43}e_4\otimes e_2\otimes e_4\\
+&B_4e_4\otimes e_2\otimes a_{11}+B_4e_4\otimes e_2\otimes a_{22}\\
+&A_{43}e_4\otimes e_3\otimes  e_1+A_{44}e_4\otimes e_3\otimes e_2+A_{41}e_4\otimes e_3\otimes e_3+A_{42}e_4\otimes e_3\otimes e_4\\
+&B_4e_4\otimes e_3\otimes a_{11}+B_4e_4\otimes e_3\otimes a_{22}
\end{align*}
\begin{align*}
+&A_{44}e_4\otimes e_4\otimes e_1+A_{43}e_4\otimes e_4\otimes e_2+A_{42}e_4\otimes e_4\otimes e_3+A_{41}e_4\otimes e_4\otimes e_4\\
+&B_4e_4\otimes e_4\otimes a_{11}+B_4e_4\otimes e_4\otimes a_{22}\\
+&B_4e_4\otimes a_{11}\otimes e_1+B_4e_4\otimes a_{11}\otimes e_2+B_4e_4\otimes a_{11}\otimes e_3+B_4e_4\otimes a_{11}\otimes e_4\\
+&\dfrac{1}{2}(A_{41}+A_{44})e_4\otimes a_{11}\otimes a_{11}+\dfrac{1}{2}(A_{42}+A_{43})e_4\otimes a_{11}\otimes a_{22}\\
+&\dfrac{1}{2}(A_{41}-A_{44})e_4\otimes a_{12}\otimes a_{12}+\dfrac{1}{2}(-\sqrt{-1}A_{42}+\sqrt{-1}A_{43})e_4\otimes a_{12}\otimes a_{21}\\
+&\dfrac{1}{2}(\sqrt{-1}A_{42}-\sqrt{-1}A_{43})e_4\otimes a_{21}\otimes a_{12}+\dfrac{1}{2}(A_{41}-A_{44})e_4\otimes a_{21}\otimes a_{21}\\
+&B_4e_4\otimes a_{22}\otimes e_1+B_4e_4\otimes a_{22}\otimes e_2+B_4e_4\otimes a_{22}\otimes e_3+B_4e_4\otimes a_{22}\otimes e_4\\
+&\dfrac{1}{2}(A_{42}+A_{43})e_4\otimes a_{22}\otimes a_{11}+\dfrac{1}{2}(A_{41}+A_{44})e_4\otimes a_{22}\otimes a_{22}\\
+&C_1a_{11}\otimes e_1\otimes e_1+C_2a_{11}\otimes e_1\otimes e_2+C_3a_{11}\otimes e_1\otimes e_3+C_4a_{11}\otimes e_1\otimes e_4\\
+&D_{1111}a_{11}\otimes e_1\otimes a_{11}+D_{1122}a_{11}\otimes e_1\otimes a_{22}\\
+&C_2a_{11}\otimes e_2\otimes e_1+C_1a_{11}\otimes e_2\otimes e_2+C_4a_{11}\otimes e_2\otimes e_3+C_3a_{11}\otimes e_2\otimes e_4\\
+&D_{1122}a_{11}\otimes e_2\otimes a_{11}+D_{1111}a_{11}\otimes e_2\otimes a_{22}\\
+&C_3a_{11}\otimes e_3\otimes e_1+C_4a_{11}\otimes e_3\otimes e_2+C_1a_{11}\otimes e_3\otimes e_3+C_2a_{11}\otimes e_3\otimes e_4\\
+&D_{1122}a_{11}\otimes e_3\otimes a_{11}+D_{1111}a_{11}\otimes e_3\otimes a_{22}\\
+&C_4a_{11}\otimes e_4\otimes e_1+C_3a_{11}\otimes e_4\otimes e_2+C_2a_{11}\otimes e_4\otimes e_3+C_1a_{11}\otimes e_4\otimes e_4\\
+&D_{1111}a_{11}\otimes e_4\otimes a_{11}+D_{1122}a_{11}\otimes e_4\otimes a_{22}\\
+&D_{1111}a_{11}\otimes a_{11}\otimes e_1+D_{1122}a_{11}\otimes a_{11}\otimes e_2\\
+&D_{1122}a_{11}\otimes a_{11}\otimes e_3+D_{1111}a_{11}\otimes a_{11}\otimes e_4\\
+&\dfrac{1}{2}(C_1+C_4)a_{11}\otimes a_{11}\otimes a_{11}+\dfrac{1}{2}(C_2+C_3)a_{11}\otimes a_{11}\otimes a_{22}\\
+&\dfrac{1}{2}(C_1-C_4)a_{11}\otimes a_{12}\otimes a_{12}+\dfrac{1}{2}(-\sqrt{-1}C_2+\sqrt{-1}C_3)a_{11}\otimes a_{12}\otimes a_{21}\\
+&\dfrac{1}{2}(\sqrt{-1}C_2-\sqrt{-1}C_3)a_{11}\otimes a_{21}\otimes a_{12}+\dfrac{1}{2}(C_1-C_4)a_{11}\otimes a_{21}\otimes a_{21}\\
+&D_{1122}a_{11}\otimes a_{22}\otimes e_1+D_{1111}a_{11}\otimes a_{22}\otimes e_2\\
+&D_{1111}a_{11}\otimes a_{22}\otimes e_3+D_{1122}a_{11}\otimes a_{22}\otimes e_4\\
+&\dfrac{1}{2}(C_2+C_3)a_{11}\otimes a_{22}\otimes a_{11}+\dfrac{1}{2}(C_1+C_4)a_{11}\otimes a_{22}\otimes a_{22}\\
+&D_{1212}a_{12}\otimes e_1\otimes a_{12}+D_{1221}a_{12}\otimes e_1\otimes a_{21}\\
-&\sqrt{-1}D_{1221}a_{12}\otimes e_2\otimes a_{12}+\sqrt{-1}D_{1212}a_{12}\otimes e_2\otimes a_{21}\\
+&\sqrt{-1}D_{1221}a_{12}\otimes e_3\otimes a_{12}-\sqrt{-1}D_{1212}a_{12}\otimes e_3\otimes a_{21}\\
-&D_{1212}a_{12}\otimes e_4\otimes a_{12}-D_{1221}a_{12}\otimes e_4\otimes a_{21}\\
+&D_{1212}a_{12}\otimes a_{12}\otimes e_1+\sqrt{-1}D_{1221}a_{12}\otimes a_{12}\otimes e_2\\
-&\sqrt{-1}D_{1221}a_{12}\otimes a_{12}\otimes e_3-D_{1212}a_{12}\otimes a_{12}\otimes e_4
\end{align*}
\begin{align*}
+&D_{1221}a_{12}\otimes a_{21}\otimes e_1-\sqrt{-1}D_{1212}a_{12}\otimes a_{21}\otimes e_2\\
+&\sqrt{-1}D_{1212}a_{12}\otimes a_{21}\otimes e_3-D_{1221}a_{12}\otimes a_{21}\otimes e_4\\
+&D_{1221}a_{21}\otimes e_1\otimes a_{12}+D_{1212} a_{21}\otimes e_1\otimes a_{21}\\
-&\sqrt{-1}D_{1212}a_{21}\otimes e_2\otimes a_{12}+\sqrt{-1}D_{1221}a_{21}\otimes e_2\otimes a_{21}\\
+&\sqrt{-1}D_{1212}a_{21}\otimes e_3\otimes a_{12}-\sqrt{-1}D_{1221}a_{21}\otimes e_3\otimes a_{21}\\
-&D_{1221}a_{21}\otimes e_4\otimes a_{12}-D_{1212}a_{21}\otimes e_4\otimes a_{21}\\
+&D_{1221}a_{21}\otimes a_{12}\otimes e_1+\sqrt{-1}D_{1212}a_{21}\otimes a_{12}\otimes e_2\\
-&\sqrt{-1}D_{1212}a_{21}\otimes a_{12}\otimes e_3-D_{1221}a_{21}\otimes a_{12}\otimes e_4\\
+&D_{1212}a_{21}\otimes a_{21}\otimes e_1-\sqrt{-1}D_{1221}a_{21}\otimes a_{21}\otimes e_2\\
+&\sqrt{-1}D_{1221}a_{21}\otimes a_{21}\otimes e_3-D_{1212}a_{21}\otimes a_{21}\otimes e_4\\
+&C_1a_{22}\otimes e_1\otimes e_1-C_2a_{22}\otimes e_1\otimes e_2-C_3a_{22}\otimes e_1\otimes e_3+C_4a_{22}\otimes e_1\otimes e_4\\
-&D_{1122}a_{22}\otimes e_1\otimes a_{11}+D_{1111}a_{22}\otimes e_1\otimes a_{22}\\
-&C_2a_{22}\otimes e_2\otimes e_1+C_1a_{22}\otimes e_2\otimes e_2+C_4a_{22}\otimes e_2\otimes e_3-C_3a_{22}\otimes e_2\otimes e_4\\
+&D_{1111}a_{22}\otimes e_2\otimes a_{11}-D_{1122}a_{22}\otimes e_2\otimes a_{22}\\
-&C_3a_{22}\otimes e_3\otimes e_1+C_4a_{22}\otimes e_3\otimes e_2+C_1a_{22}\otimes e_3\otimes e_3-C_2a_{22}\otimes e_3\otimes e_4\\
+&D_{1111}a_{22}\otimes e_3\otimes a_{11}-D_{1122}a_{22}\otimes e_3\otimes a_{22}\\
+&C_4a_{22}\otimes e_4\otimes e_1-C_3a_{22}\otimes e_4\otimes e_2-C_2a_{22}\otimes e_4\otimes e_3+C_1a_{22}\otimes e_4\otimes e_4\\
-&D_{1122}a_{22}\otimes e_4\otimes a_{11}+D_{1111}a_{22}\otimes e_4\otimes a_{22}\\
-&D_{1122}a_{22}\otimes a_{11}\otimes e_1+D_{1111}a_{22}\otimes a_{11}\otimes e_2\\
+&D_{1111}a_{22}\otimes a_{11}\otimes e_3-D_{1122}a_{22}\otimes a_{11}\otimes e_4\\
+&\dfrac{1}{2}(C_1+C_4)a_{22}\otimes a_{11}\otimes a_{11}+\dfrac{1}{2}(-C_2-C_3)a_{22}\otimes a_{11}\otimes a_{22}\\
+&\dfrac{1}{2}(C_1-C_4)a_{22}\otimes a_{12}\otimes a_{12}+\dfrac{1}{2}(\sqrt{-1}C_2-\sqrt{-1}C_3)a_{22}\otimes a_{12}\otimes a_{21}\\
+&\dfrac{1}{2}(-\sqrt{-1}C_2+\sqrt{-1}C_3)a_{22}\otimes a_{21}\otimes a_{12}+\dfrac{1}{2}(C_1-C_4)a_{22}\otimes a_{21}\otimes a_{21}\\
+&D_{1111}a_{22}\otimes a_{22}\otimes e_1-D_{1122}a_{22}\otimes a_{22}\otimes e_2\\
-&D_{1122}a_{22}\otimes a_{22}\otimes e_3+D_{1111}a_{22}\otimes a_{22}\otimes e_4\\
+&\dfrac{1}{2}(-C_2-C_3)a_{22}\otimes a_{22}\otimes a_{11}+\dfrac{1}{2}(C_1+C_4)a_{22}\otimes a_{22}\otimes a_{22}.
\end{align*}
We compute $R_{13}R_{12}$.
\begin{align*}
&R_{13}R_{12}\\
=&A_{11}A_{11}e_1\otimes e_1\otimes e_1+A_{12}A_{11}e_1\otimes e_1\otimes e_2+A_{13}A_{11}e_1\otimes e_1\otimes e_3\\+&A_{14}A_{11}e_1\otimes e_1\otimes e_4+B_1A_{11}e_1\otimes e_1\otimes a_{11}+B_1A_{11}e_1\otimes e_1\otimes a_{22}\\
+&A_{11}A_{12}e_1\otimes e_2\otimes e_1+A_{12}A_{12}e_1\otimes e_2\otimes e_2+A_{13}A_{12}e_1\otimes e_2\otimes e_3\\
+&A_{14}A_{12}e_1\otimes e_2\otimes e_4+B_1A_{12}e_1\otimes e_2\otimes a_{11}+B_1A_{12}e_1\otimes e_2\otimes a_{22}\\
\end{align*}
\begin{align*}
+&A_{11}A_{13}e_1\otimes e_3\otimes e_1+A_{12}A_{13}e_1\otimes e_3\otimes e_2+A_{13}A_{13}e_1\otimes e_3\otimes e_3\\+&A_{14}A_{13}e_1\otimes e_3\otimes e_4+B_1A_{13}e_1\otimes e_3\otimes a_{11}+B_1A_{13}e_1\otimes e_3\otimes a_{22}\\
+&A_{11}A_{14}e_1\otimes e_4\otimes e_1+A_{12}A_{14}e_1\otimes e_4\otimes e_2+A_{13}A_{14}e_1\otimes e_4\otimes e_3\\
+&A_{14}A_{14}e_1\otimes e_4\otimes e_4+B_1A_{14}e_1\otimes e_4\otimes a_{11}+B_1A_{14}e_1\otimes e_4\otimes a_{22}\\
+&A_{11}B_1e_1\otimes a_{11}\otimes e_1+A_{12}B_1e_1\otimes a_{11}\otimes e_2+A_{13}B_1e_1\otimes a_{11}\otimes e_3\\
+&A_{14}B_1e_1\otimes a_{11}\otimes e_4+B_1^2e_1\otimes a_{11}\otimes a_{11}+B_1^2e_1\otimes a_{11}\otimes a_{22}\\
+&A_{11}B_1e_1\otimes a_{22}\otimes e_1+A_{12}B_1e_1\otimes a_{22}\otimes e_2+A_{13}B_1e_1\otimes a_{22}\otimes e_3\\
+&A_{14}B_1e_1\otimes a_{22}\otimes e_4+B_1^2e_1\otimes a_{22}\otimes a_{11}+B_1^2e_1\otimes a_{22}\otimes a_{22}\\
+&A_{21}A_{21}e_2\otimes e_1\otimes e_1+A_{22}A_{21}e_2\otimes e_1\otimes e_2+A_{23}A_{21}e_2\otimes e_1\otimes e_3\\+&A_{24}A_{21}e_2\otimes e_1\otimes e_4+B_2A_{21}e_2\otimes e_1\otimes a_{11}-B_2A_{21}e_2\otimes e_1\otimes a_{22}\\
+&A_{21}A_{22}e_2\otimes e_2\otimes e_1+A_{22}A_{22}e_2\otimes e_2\otimes e_2+A_{23}A_{22}e_2\otimes e_2\otimes e_3\\
+&A_{24}A_{22}e_2\otimes e_2\otimes e_4+B_2A_{22}e_2\otimes e_2\otimes a_{11}-B_2A_{22}e_2\otimes e_2\otimes a_{22}\\
+&A_{21}A_{23}e_2\otimes e_3\otimes e_1+A_{22}A_{23}e_2\otimes e_3\otimes e_2+A_{23}A_{23}e_2\otimes e_3\otimes e_3\\+&A_{24}A_{23}e_2\otimes e_3\otimes e_4+B_2A_{23}e_2\otimes e_3\otimes a_{11}-B_2A_{23}e_2\otimes e_3\otimes a_{22}\\
+&A_{21}A_{24}e_2\otimes e_4\otimes e_1+A_{22}A_{24}e_2\otimes e_4\otimes e_2+A_{23}A_{24}e_2\otimes e_4\otimes e_3\\
+&A_{24}A_{24}e_2\otimes e_4\otimes e_4+B_2A_{24}e_2\otimes e_4\otimes a_{11}-B_2A_{24}e_2\otimes e_4\otimes a_{22}\\
+&A_{21}B_2e_2\otimes a_{11}\otimes e_1+A_{22}B_2e_2\otimes a_{11}\otimes e_2+A_{23}B_2e_2\otimes a_{11}\otimes e_3\\
+&A_{24}B_2e_2\otimes a_{11}\otimes e_4+B_2^2e_2\otimes a_{11}\otimes a_{11}-B_2^2e_2\otimes a_{11}\otimes a_{22}\\
-&A_{21}B_2e_2\otimes a_{22}\otimes e_1-A_{22}B_2e_2\otimes a_{22}\otimes e_2-A_{23}B_2e_2\otimes a_{22}\otimes e_3\\
-&A_{24}B_2e_2\otimes a_{22}\otimes e_4-B_2^2e_2\otimes a_{22}\otimes a_{11}+B_2^2e_2\otimes a_{22}\otimes a_{22}\\
+&A_{31}A_{31}e_3\otimes e_1\otimes e_1+A_{32}A_{31}e_3\otimes e_1\otimes e_2+A_{33}A_{31}e_3\otimes e_1\otimes e_3\\+&A_{34}A_{31}e_3\otimes e_1\otimes e_4+B_3A_{31}e_3\otimes e_1\otimes a_{11}-B_3A_{31}e_3\otimes e_1\otimes a_{22}\\
+&A_{31}A_{32}e_3\otimes e_2\otimes e_1+A_{32}A_{32}e_3\otimes e_2\otimes e_2+A_{33}A_{32}e_3\otimes e_2\otimes e_3\\
+&A_{34}A_{32}e_3\otimes e_2\otimes e_4+B_3A_{32}e_3\otimes e_2\otimes a_{11}-B_3A_{32}e_3\otimes e_2\otimes a_{22}\\
+&A_{31}A_{33}e_3\otimes e_3\otimes e_1+A_{32}A_{33}e_3\otimes e_3\otimes e_2+A_{33}A_{33}e_3\otimes e_3\otimes e_3\\+&A_{34}A_{33}e_3\otimes e_3\otimes e_4+B_3A_{33}e_3\otimes e_3\otimes a_{11}-B_3A_{33}e_3\otimes e_3\otimes a_{22}\\
+&A_{31}A_{34}e_3\otimes e_4\otimes e_1+A_{32}A_{34}e_3\otimes e_4\otimes e_2+A_{33}A_{34}e_3\otimes e_4\otimes e_3\\
+&A_{34}A_{34}e_3\otimes e_4\otimes e_4+B_3A_{34}e_3\otimes e_4\otimes a_{11}-B_3A_{34}e_3\otimes e_4\otimes a_{22}\\
+&A_{31}B_3e_3\otimes a_{11}\otimes e_1+A_{32}B_3e_3\otimes a_{11}\otimes e_2+A_{33}B_3e_3\otimes a_{11}\otimes e_3\\
+&A_{34}B_3e_3\otimes a_{11}\otimes e_4+B_3^2e_3\otimes a_{11}\otimes a_{11}-B_3^2e_1\otimes a_{11}\otimes a_{22}\\
-&A_{31}B_3e_3\otimes a_{22}\otimes e_1-A_{32}B_3e_3\otimes a_{22}\otimes e_2-A_{33}B_3e_3\otimes a_{22}\otimes e_3\\
-&A_{34}B_3e_3\otimes a_{22}\otimes e_4-B_3^2e_3\otimes a_{22}\otimes a_{11}+B_3^2e_3\otimes a_{22}\otimes a_{22}\\
+&A_{41}A_{41}e_4\otimes e_1\otimes e_1+A_{42}A_{41}e_4\otimes e_1\otimes e_2+A_{43}A_{41}e_4\otimes e_1\otimes e_3\\+&A_{44}A_{41}e_4\otimes e_1\otimes e_4+B_4A_{41}e_4\otimes e_1\otimes a_{11}+B_4A_{41}e_4\otimes e_1\otimes a_{22}\\
+&A_{41}A_{42}e_4\otimes e_2\otimes e_1+A_{42}A_{42}e_4\otimes e_2\otimes e_2+A_{43}A_{42}e_4\otimes e_2\otimes e_3\\
+&A_{44}A_{42}e_4\otimes e_2\otimes e_4+B_4A_{42}e_4\otimes e_2\otimes a_{11}+B_4A_{42}e_4\otimes e_2\otimes a_{22}
\end{align*}
\begin{align*}
+&A_{41}A_{43}e_4\otimes e_3\otimes e_1+A_{42}A_{43}e_4\otimes e_3\otimes e_2+A_{43}A_{43}e_4\otimes e_3\otimes e_3\\+&A_{44}A_{43}e_4\otimes e_3\otimes e_4+B_4A_{43}e_4\otimes e_3\otimes a_{11}+B_4A_{43}e_4\otimes e_3\otimes a_{22}\\
+&A_{41}A_{44}e_4\otimes e_4\otimes e_1+A_{42}A_{44}e_4\otimes e_4\otimes e_2+A_{43}A_{44}e_4\otimes e_4\otimes e_3\\
+&A_{44}A_{44}e_4\otimes e_4\otimes e_4+B_4A_{44}e_4\otimes e_4\otimes a_{11}+B_4A_{44}e_4\otimes e_4\otimes a_{22}\\
+&A_{41}B_4e_4\otimes a_{11}\otimes e_1+A_{42}B_4e_4\otimes a_{11}\otimes e_2+A_{43}B_4e_4\otimes a_{11}\otimes e_3\\
+&A_{44}B_4e_4\otimes a_{11}\otimes e_4+B_4^2e_4\otimes a_{11}\otimes a_{11}+B_4^2e_4\otimes a_{11}\otimes a_{22}\\
+&A_{41}B_4e_4\otimes a_{22}\otimes e_1+A_{42}B_4e_4\otimes a_{22}\otimes e_2+A_{43}B_4e_4\otimes a_{22}\otimes e_3\\
+&A_{44}B_4e_4\otimes a_{22}\otimes e_4+B_4^2e_4\otimes a_{22}\otimes a_{11}+B_4^2e_4\otimes a_{22}\otimes a_{22}\\
+&C_1^2a_{11}\otimes e_1\otimes e_1+C_2C_1a_{11}\otimes e_1\otimes e_2+C_3C_1a_{11}\otimes e_1\otimes e_3\\
+&C_4C_1a_{11}\otimes e_1\otimes e_4+D_{1111}C_1a_{11}\otimes e_1\otimes a_{11}+D_{1122}C_1a_{11}\otimes e_1\otimes a_{22}\\
+&C_1C_2a_{11}\otimes e_2\otimes e_1+C_2^2a_{11}\otimes e_2\otimes e_2+C_3C_2a_{11}\otimes e_2\otimes e_3\\
+& C_4C_2a_{11}\otimes e_2\otimes e_4+D_{1111}C_2 a_{11}\otimes e_2\otimes a_{11}+D_{1122}C_2a_{11}\otimes e_2\otimes a_{22}\\
+&C_1C_3a_{11}\otimes e_3\otimes e_1+C_2C_3a_{11}\otimes e_3\otimes e_2+C_3^2a_{11}\otimes e_3\otimes e_3\\
+&C_4C_3a_{11}\otimes e_3\otimes e_4+D_{1111}C_3a_{11}\otimes e_3\otimes a_{11}+D_{112}C_3a_{11}\otimes e_3\otimes a_{22}\\
+&C_1C_4a_{11}\otimes e_4\otimes e_1+C_2C_4a_{11}\otimes e_4\otimes e_2+C_3C_4a_{11}\otimes e_4\otimes e_3\\
+&C_4^2a_{11}\otimes e_4\otimes e_4+D_{1111}C_4a_{11}\otimes e_4\otimes a_{11}+D_{1122}C_4a_{11}\otimes e_4\otimes a_{22}\\
+&C_1D_{1111}a_{11}\otimes a_{11}\otimes e_1+C_2D_{1111}a_{11}\otimes a_{11}\otimes e_2+C_3D_{1111}a_{11}\otimes a_{11}\otimes e_3\\
+&C_4D_{1111}a_{11}\otimes a_{11}\otimes e_4+D_{1111}^2a_{11}\otimes a_{11}\otimes a_{11}+D_{1122}D_{1111}a_{11}\otimes a_{11}\otimes a_{22}\\
+&D_{1212}D_{1221}a_{11}\otimes a_{12}\otimes a_{12}+D_{1221}^2a_{11}\otimes a_{12}\otimes a_{21}\\
+&D_{1212}^2a_{11}\otimes a_{21}\otimes a_{12}+D_{1221}D_{1212}a_{11}\otimes a_{21}\otimes a_{21}\\
+&C_1D_{1122}a_{11}\otimes a_{22}\otimes e_1+C_2D_{1122}a_{11}\otimes a_{22}\otimes e_2+C_3D_{1122}a_{11}\otimes a_{22}\otimes e_3\\
+&C_4D_{1122}a_{11}\otimes a_{22}\otimes e_4+D_{1111}D_{1122}a_{11}\otimes a_{22}\otimes a_{11}+D_{1122}^2a_{11}\otimes a_{22}\otimes a_{22}\\
+&D_{1212}C_1a_{12}\otimes e_1\otimes a_{12}+D_{1221}C_1a_{12}\otimes e_1\otimes a_{21}\\
-&D_{1212}C_2a_{12}\otimes e_2\otimes a_{12}-D_{1221}C_2a_{12}\otimes e_2\otimes a_{21}\\
-&D_{1212}C_3a_{12}\otimes e_3\otimes a_{12}-D_{1221}C_3a_{12}\otimes e_3\otimes a_{21}\\
+&D_{1212}C_4a_{12}\otimes e_4\otimes a_{12}+D_{1221}C_4a_{12}\otimes e_4\otimes a_{21}\\
-&D_{1212}D_{1122}a_{12}\otimes a_{11}\otimes a_{12}-D_{1221}D_{1122}a_{12}\otimes a_{11}\otimes a_{21}\\
+&C_1D_{1212}a_{12}\otimes a_{12}\otimes e_1+C_2D_{1212}a_{12}\otimes a_{12}\otimes e_2\\
+&C_3D_{1212}a_{12}\otimes a_{12}\otimes e_3+C_4D_{1212}a_{12}\otimes a_{12}\otimes e_4\\
+&D_{1111}D_{1212}a_{12}\otimes a_{12}\otimes a_{11}+D_{1122}D_{1212}a_{12}\otimes a_{12}\otimes a_{22}\\
+&C_1D_{1221}a_{12}\otimes a_{21}\otimes e_1+C_2D_{1221}a_{12}\otimes a_{21}\otimes e_2\\
+&C_3D_{1221}a_{12}\otimes a_{21}\otimes e_3+C_4D_{1221}a_{12}\otimes a_{21}\otimes e_4\\
+&D_{1111}D_{1221}a_{12}\otimes a_{21}\otimes a_{11}+D_{1122}D_{1221}a_{12}\otimes a_{21}\otimes a_{22}\\
+&D_{1212}D_{1111}a_{12}\otimes a_{22}\otimes a_{12}+D_{1221}D_{1111}a_{12}\otimes a_{22}\otimes a_{21}\\
+&D_{1221}C_1a_{21}\otimes e_1\otimes a_{12}+D_{1212}C_1a_{21}\otimes e_1\otimes a_{21}\\
+&D_{1221}C_2a_{21}\otimes e_2\otimes a_{12}+D_{1212}C_2a_{21}\otimes e_2\otimes a_{21}
\end{align*}
\begin{align*}
+&D_{1221}C_3a_{21}\otimes e_3\otimes a_{12}+D_{1212}C_3a_{21}\otimes e_3\otimes a_{21}\\
+&D_{1221}C_4a_{21}\otimes e_4\otimes a_{12}+D_{1212}C_4a_{21}\otimes e_4\otimes a_{21}\\
+&D_{1221}D_{1111}a_{21}\otimes a_{11}\otimes a_{12}+D_{1212}D_{1111}a_{21}\otimes a_{11}\otimes a_{21}\\
+&C_1D_{1221}a_{21}\otimes a_{12}\otimes e_1-C_2D_{1221}a_{21}\otimes a_{12}\otimes e_2\\
-&C_3D_{1221}a_{21}\otimes a_{12}\otimes e_3+C_4D_{1221}a_{21}\otimes a_{12}\otimes e_4\\
-&D_{1122}D_{1221}a_{21}\otimes a_{12}\otimes a_{11}+D_{1111}D_{1221}a_{21}\otimes a_{12}\otimes a_{22}\\
+&C_1D_{1212}a_{21}\otimes a_{21}\otimes e_1-C_2D_{1212}a_{21}\otimes a_{21}\otimes e_2\\
-&C_3D_{1212}a_{21}\otimes a_{21}\otimes e_3+C_4D_{1212}a_{21}\otimes a_{21}\otimes e_4\\
-&D_{1122}D_{1212}a_{21}\otimes a_{21}\otimes a_{11}+D_{1111}D_{1212}a_{21}\otimes a_{21}\otimes a_{22}\\
+&D_{1221}D_{1122}a_{21}\otimes a_{22}\otimes a_{12}+D_{1212}D_{1122}a_{21}\otimes a_{22}\otimes a_{21}\\
+&C_1^2a_{22}\otimes e_1\otimes e_1-C_2C_1a_{22}\otimes e_1\otimes e_2-C_3C_1a_{22}e_1\otimes e_3\\
+&C_4C_1a_{22}\otimes e_1\otimes e_4-D_{1122}C_1a_{22}\otimes e_1\otimes a_{11}+D_{1111}C_1a_{22}\otimes e_1\otimes a_{22}\\
-&C_1C_2a_{22}\otimes e_2\otimes e_1+C_2^2a_{22}\otimes e_2\otimes e_2+C_3C_2a_{22}\otimes e_2\otimes e_3\\
-&C_4C_2a_{22}\otimes e_2\otimes e_4+D_{1122}C_2a_{22}\otimes e_2\otimes a_{11}-D_{1111}C_2a_{22}\otimes e_2\otimes a_{22}\\
-&C_1C_3a_{22}\otimes e_3\otimes e_1+C_2C_3a_{22}\otimes e_3\otimes e_2+C_3^2a_{22}\otimes e_3\otimes e_3\\
-&C_4C_3a_{22}\otimes e_3\otimes e_4+D_{1122}C_3a_{22}\otimes e_3\otimes a_{11}-D_{1111}C_3a_{22}\otimes e_3\otimes a_{22}\\
+&C_1C_4a_{22}\otimes e_4\otimes e_1-C_2C_4a_{22}\otimes e_4\otimes e_2-C_3C_4a_{22}\otimes e_4\otimes e_3\\
+&C_4^2a_{22}\otimes e_4\otimes e_4-D_{1122}C_4a_{22}\otimes e_4\otimes a_{11}+D_{1111}C_4a_{22}\otimes e_4\otimes a_{22}\\
-&C_1D_{1122}a_{22}\otimes a_{11}\otimes e_1+C_2D_{1122}a_{22}\otimes a_{11}\otimes e_2+C_3D_{1122}a_{22}\otimes a_{11}\otimes e_3\\
-&C_4D_{1122}a_{22}\otimes a_{11}\otimes e_4+D_{1122}^2a_{22}\otimes a_{11}\otimes a_{11}-D_{1111}D_{1122}a_{22}\otimes a_{11}\otimes a_{22}\\
+&D_{1221}D_{1212}a_{22}\otimes a_{12}\otimes a_{12}+D_{1212}^2a_{22}\otimes a_{12}\otimes a_{21}\\
+&D_{1221}^2a_{22}\otimes a_{21}\otimes a_{12}+D_{1212}D_{1221}a_{22}\otimes a_{21}\otimes a_{21}\\
+&C_1D_{1111}a_{22}\otimes a_{22}\otimes e_1-C_2D_{1111}a_{22}\otimes a_{22}\otimes e_2-C_3D_{1111}a_{22}\otimes a_{22}\otimes e_3\\
+&C_4D_{1111}a_{2}\otimes a_{22}\otimes e_4-D_{1122}D_{1111}a_{22}\otimes a_{22}\otimes a_{11}+D_{1111}^2a_{22}\otimes a_{22}\otimes a_{22}.
\end{align*}
By comparing the coefficients of $a_{11}\otimes e_2\otimes e_2$, $a_{11}\otimes e_3\otimes e_3$ and $a_{11}\otimes e_4\otimes e_4$, we see that $C_2^2=C_3^2=C_4^2=C_1=1$. By comparing the coefficients of $a_{11}\otimes e_2\otimes e_3$ we have $C_3C_2=C_4$. Thus $(C_2,C_3,C_4)=(1,1,1)$, $(1,-1,-1)$, $(-1,1,-1)$ or $(-1,-1,1)$. Let us consider the following four cases.
\begin{case} $(C_2,C_3,C_4)=(1,1,1)$\\
By comparing the coefficients of $a_{11}\otimes a_{11}\otimes e_3$, we have $C_3D_{1111}=D_{1122}$ and thus $D_{1111}=D_{1122}$. By comparing the coefficients of $a_{12}\otimes a_{12}\otimes e_3$, we have $C_3D_{1212}=-\sqrt{-1}D_{1221}$ and thus $D_{1221}=\sqrt{-1}D_{1212}$. By comparing the coefficients of $a_{21}\otimes a_{21}\otimes e_4$, we have $C_4D_{1212}=-D_{1212}$ and thus $D_{1212}=D_{1221}=0$. By comparing the coefficients of $a_{11}\otimes a_{11}\otimes a_{11}$, we have $D_{1111}^2=\dfrac{1}{2}(C_1+C_4)=1$ and thus $D_{1111}=\pm1$. By comparing the coefficients of $a_{11}\otimes a_{11}\otimes a_{22}$, we have $D_{1122}D_{1111}=\dfrac{1}{2}(C_2+C_3)=1$ and thus $(D_{1111},D_{1122},D_{1212},D_{1221})=(1,1,0,0)$ or $(-1,-1,0,0)$.
\end{case}
\begin{case} $(C_2,C_3,C_4)=(1,-1,-1)$\\
Let $\lambda=D_{1212}$. By comparing the coefficients of $a_{11}\otimes a_{11}\otimes e_3$, we have $C_3D_{1111}=D_{1122}$ and thus $D_{1111}=-D_{1122}$. By comparing the coefficients of $a_{12}\otimes a_{12}\otimes e_3$, we have $C_3D_{1212}=-\sqrt{-1}D_{1221}$ and thus $D_{1221}=-\sqrt{-1}D_{1212}$. By comparing the coefficients of $a_{11}\otimes a_{22}\otimes e_4$, we have $C_4D_{1122}=D_{1122}$ and thus $(D_{1111},D_{1122},D_{1212},D_{1221})=(0,0,\lambda,-\sqrt{-1}\lambda)$.
\end{case}
\begin{case} $(C_2,C_3,C_4)=(-1,1,-1)$\\
Let $\lambda=D_{1212}$. By comparing the coefficients of $a_{11}\otimes a_{11}\otimes e_3$, we have $C_3D_{1111}=D_{1122}$ and thus $D_{1111}=D_{1122}$. By comparing the coefficients of $a_{12}\otimes a_{12}\otimes e_3$, we have $C_3D_{1212}=-\sqrt{-1}D_{1221}$ and thus $D_{1221}=-\sqrt{-1}D_{1212}$. By comparing the coefficients of $a_{11}\otimes a_{22}\otimes e_4$, we have $C_4D_{1122}=D_{1122}$ and thus $(D_{1111},D_{1122},D_{1212},D_{1221})=(0,0,\lambda,\sqrt{-1}\lambda)$.
\end{case}
\begin{case} $(C_2,C_3,C_4)=(-1,-1,1)$\\
By comparing the coefficients of $a_{11}\otimes a_{11}\otimes e_3$, we have $C_3D_{1111}=D_{1122}$ and thus $D_{1111}=-D_{1122}$. By comparing the coefficients of $a_{12}\otimes a_{12}\otimes e_3$, we have $C_3D_{1212}=-\sqrt{-1}D_{1221}$ and thus $D_{1221}=-\sqrt{-1}D_{1212}$. By comparing the coefficients of $a_{21}\otimes a_{21}\otimes e_4$, we have $C_4D_{1212}=-D_{1212}$ and thus $D_{1212}=D_{1221}=0$. By comparing the coefficients of $a_{11}\otimes a_{11}\otimes a_{11}$, we have $D_{1111}^2=\dfrac{1}{2}(C_1+C_4)=1$ and thus $D_{1111}=\pm1$. By comparing the coefficients of $a_{11}\otimes a_{11}\otimes a_{22}$, we have $D_{1122}D_{1111}=\dfrac{1}{2}(C_2+C_3)=-1$ and thus $(D_{1111},D_{1122},D_{1212},D_{1221})=(1,-1,0,0)$ or $(-1,1,0,0)$.
\end{case}
Let us compute $(\Delta\otimes\id)(R)$.
\begin{align*}
&(\Delta\otimes\id)(R)\\
=&A_{11}e_1\otimes e_1\otimes e_1+A_{12}e_1\otimes e_1\otimes e_2+A_{13}e_1\otimes e_1\otimes e_3+A_{14}e_1\otimes e_1\otimes e_4\\
+&B_1e_1\otimes e_1\otimes a_{11}+B_1e_1\otimes e_1\otimes a_{22}\\
+&A_{21}e_1\otimes e_2\otimes e_1+A_{22}e_1\otimes e_2\otimes e_2+A_{23}e_1\otimes e_2\otimes e_3+A_{24}e_1\otimes e_2\otimes e_4\\
+&B_2e_1\otimes e_2\otimes a_{11}-B_2e_1\otimes e_2\otimes a_{22}\\
+&A_{31}e_1\otimes e_3\otimes e_1+A_{32}e_1\otimes e_3\otimes e_2+A_{33}e_1\otimes e_3\otimes e_3+A_{34}e_1\otimes e_3\otimes e_4\\
+&B_3e_1\otimes e_3\otimes a_{11}-B_3e_1\otimes e_3\otimes a_{22}\\
+&A_{41}e_1\otimes e_4\otimes e_1+A_{42}e_1\otimes e_4\otimes e_2+A_{43}e_1\otimes e_4\otimes e_3+A_{44}e_1\otimes e_4\otimes e_4\\
+&B_4e_1\otimes e_4\otimes a_{11}+B_4e_1\otimes e_4\otimes a_{22}\\
+&C_1e_1\otimes a_{11}\otimes e_1+C_2e_1\otimes a_{11}\otimes e_2+C_3e_1\otimes a_{11}\otimes e_3+C_4e_1\otimes a_{11}\otimes e_4\\
+&D_{1111}e_1\otimes a_{11}\otimes a_{11}+D_{1122}e_1\otimes a_{11}\otimes a_{22}\\
+&D_{1212}e_1\otimes a_{12}\otimes a_{12}+D_{1221}e_1\otimes a_{12}\otimes a_{21}\\
+&D_{1221}e_1\otimes a_{21}\otimes a_{12}+D_{1212}e_1\otimes a_{21}\otimes a_{21}\\
+&C_1e_1\otimes a_{22}\otimes e_1-C_2e_1\otimes a_{22}\otimes e_2-C_3e_1\otimes a_{22}\otimes e_3+C_4e_1\otimes a_{22}\otimes e_4\\
-&D_{1122}e_1\otimes a_{22}\otimes a_{11}+D_{1111}e_1\otimes a_{22}\otimes a_{22}\\
+&A_{21}e_2\otimes e_1\otimes e_1+A_{22}e_2\otimes e_1\otimes e_2+A_{23}e_2\otimes e_1\otimes e_3+A_{24}e_2\otimes e_1\otimes e_4
\end{align*}
\begin{align*}
+&B_2e_2\otimes e_1\otimes a_{11}-B_2e_2\otimes e_1\otimes a_{22}\\
+&A_{11}e_2\otimes e_2\otimes e_1+A_{12}e_2\otimes e_2\otimes e_2+A_{13}e_2\otimes e_2\otimes e_3+A_{14}e_2\otimes e_2\otimes e_4\\
+&B_1e_2\otimes e_2\otimes a_{11}+B_1e_2\otimes e_2\otimes a_{22}\\
+&A_{41}e_2\otimes e_3\otimes e_1+A_{42}e_2\otimes e_3\otimes e_2+A_{43}e_2\otimes e_3\otimes e_3+A_{44}e_2\otimes e_3\otimes e_4\\
+&B_4e_2\otimes e_3\otimes a_{11}+B_4e_2\otimes e_3\otimes a_{22}\\
+&A_{31}e_2\otimes e_4\otimes e_1+A_{32}e_2\otimes e_4\otimes e_2+A_{33}e_2\otimes e_4\otimes e_3+A_{34}e_2\otimes e_4\otimes e_4\\
+&B_3e_2\otimes e_4\otimes a_{11}-B_3e_2\otimes e_4\otimes a_{22}\\
+&C_1e_2\otimes a_{11}\otimes e_1-C_2e_2\otimes a_{11}\otimes e_2-C_3e_2\otimes a_{11}\otimes e_3+C_4e_2\otimes a_{11}\otimes e_4\\
-&D_{1122}e_2\otimes a_{11}\otimes a_{11}+D_{1111}e_2\otimes a_{11}\otimes a_{22}\\
-&\sqrt{-1}D_{1221}e_2\otimes a_{12}\otimes a_{12}-\sqrt{-1}D_{1212}e_2\otimes a_{12}\otimes a_{21}\\
+&\sqrt{-1}D_{1212}e_2\otimes a_{21}\otimes a_{12}+\sqrt{-1}D_{1221}e_2\otimes a_{21}\otimes a_{21}\\
+&C_1e_2\otimes a_{22}\otimes e_1+C_2e_2\otimes a_{22}\otimes e_2+C_3e_2\otimes a_{22}\otimes e_3+C_4e_2\otimes a_{22}\otimes e_4\\
+&D_{1111}e_2\otimes a_{22}\otimes a_{11}+D_{1122}e_2\otimes a_{22}\otimes a_{22}\\
+&A_{31}e_3\otimes e_1\otimes e_1+A_{32}e_3\otimes e_1\otimes e_2+A_{33}e_3\otimes e_1\otimes e_3+A_{34}e_3\otimes e_1\otimes e_4\\
+&B_3e_3\otimes e_1\otimes a_{11}-B_3e_3\otimes e_1\otimes a_{22}\\
+&A_{41}e_3\otimes e_2\otimes e_1+A_{42}e_3\otimes e_2\otimes e_2+A_{43}e_3\otimes e_2\otimes e_3+A_{44}e_3\otimes e_2\otimes e_4\\
+&B_4e_3\otimes e_2\otimes a_{11}+B_4e_3\otimes e_2\otimes a_{22}\\
+&A_{11}e_3\otimes e_3\otimes e_1+A_{12}e_3\otimes e_3\otimes e_2+A_{13}e_3\otimes e_3\otimes e_3+A_{14}e_3\otimes e_3\otimes e_4\\
+&B_1e_3\otimes e_3\otimes a_{11}+B_1e_3\otimes e_3\otimes a_{22}\\
+&A_{21}e_3\otimes e_4\otimes e_1+A_{22}e_3\otimes e_4\otimes e_2+A_{23}e_3\otimes e_4\otimes e_3+A_{24}e_3\otimes e_4\otimes e_4\\
+&B_2e_3\otimes e_4\otimes a_{11}-B_2e_3\otimes e_4\otimes a_{22}\\
+&C_1e_3\otimes a_{11}\otimes e_1-C_2e_3\otimes a_{11}\otimes e_2-C_3e_3\otimes a_{11}\otimes e_3+C_4e_3\otimes a_{11}\otimes e_4\\
-&D_{1122}e_3\otimes a_{11}\otimes a_{11}+D_{1111}e_3\otimes a_{11}\otimes a_{22}\\
+&\sqrt{-1}D_{1221}e_3\otimes a_{12}\otimes a_{12}+\sqrt{-1}D_{1212}e_3\otimes a_{12}\otimes a_{21}\\
-&\sqrt{-1}D_{1212}e_3\otimes a_{21}\otimes a_{12}-\sqrt{-1}D_{1221}e_3\otimes a_{21}\otimes a_{21}\\
+&C_1e_3\otimes a_{22}\otimes e_1+C_2e_3\otimes a_{22}\otimes e_2+C_3e_3\otimes a_{22}\otimes e_3+C_4e_3\otimes a_{22}\otimes e_4\\
+&D_{1111}e_3\otimes a_{22}\otimes a_{11}+D_{1122}e_3\otimes a_{22}\otimes a_{22}\\
+&A_{41}e_4\otimes e_1\otimes e_1+A_{42}e_4\otimes e_1\otimes e_2+A_{43}e_4\otimes e_1\otimes e_3+A_{44}e_4\otimes e_1\otimes e_4\\
+&B_4e_4\otimes e_1\otimes a_{11}+B_4e_4\otimes e_1\otimes a_{22}\\
+&A_{31}e_4\otimes e_2\otimes e_1+A_{32}e_4\otimes e_2\otimes e_2+A_{33}e_4\otimes e_2\otimes e_3+A_{34}e_4\otimes e_2\otimes e_4\\
+&B_3e_4\otimes e_2\otimes a_{11}-B_3e_4\otimes e_2\otimes a_{22}\\
+&A_{21}e_4\otimes e_3\otimes e_1+A_{22}e_4\otimes e_3\otimes e_2+A_{23}e_4\otimes e_3\otimes e_3+A_{24}e_4\otimes e_3\otimes e_4\\
+&B_2e_4\otimes e_3\otimes a_{11}-B_2e_4\otimes e_3\otimes a_{22}\\
+&A_{11}e_4\otimes e_4\otimes e_1+A_{12}e_4\otimes e_4\otimes e_2+A_{13}e_4\otimes e_4\otimes e_3+A_{14}e_4\otimes e_4\otimes e_4\\
\end{align*}
\begin{align*}
+&B_1e_4\otimes e_4\otimes a_{11}+B_1e_4\otimes e_4\otimes a_{22}\\
+&C_1e_4\otimes a_{11}\otimes e_1+C_2e_4\otimes a_{11}\otimes e_2+C_3e_4\otimes a_{11}\otimes e_3+C_4e_4\otimes a_{11}\otimes e_4\\
+&D_{1111}e_4\otimes a_{11}\otimes a_{11}+D_{1122}e_4\otimes a_{11}\otimes a_{22}\\
-&D_{1212}e_4\otimes a_{12}\otimes a_{12}-D_{1221}e_4\otimes a_{12}\otimes a_{21}\\
-&D_{1221}e_4\otimes a_{21}\otimes a_{12}-D_{1212}e_4\otimes a_{21}\otimes a_{21}\\
+&C_1e_4\otimes a_{22}\otimes e_1-C_2e_4\otimes a_{22}\otimes e_2-C_3e_4\otimes a_{22}\otimes e_3+C_4e_4\otimes a_{22}\otimes e_4\\
-&D_{1122}e_4\otimes a_{22}\otimes a_{11}+D_{1111}e_4\otimes a_{22}\otimes a_{22}\\
+&C_1a_{11}\otimes e_1\otimes e_1+C_2a_{11}\otimes e_1\otimes e_2+C_3a_{11}\otimes e_1\otimes e_3+C_4a_{11}\otimes e_1\otimes e_4\\
+&D_{1111}a_{11}\otimes e_1\otimes a_{11}+D_{1122}a_{11}\otimes e_1\otimes a_{22}\\
+&C_1a_{11}\otimes e_2\otimes e_1-C_2a_{11}\otimes e_2\otimes e_2-C_3a_{11}\otimes e_2\otimes e_3+C_4a_{11}\otimes e_2\otimes e_4\\
-&D_{1122}a_{11}\otimes e_2\otimes a_{11}+D_{1111}a_{11}\otimes e_2\otimes a_{22}\\
+&C_1a_{11}\otimes e_3\otimes e_1-C_2a_{11}\otimes e_3\otimes e_2-C_3a_{11}\otimes e_3\otimes e_3+C_4a_{11}\otimes e_3\otimes e_4\\
-&D_{1122}a_{11}\otimes e_3\otimes a_{11}+D_{1111}a_{11}\otimes e_3\otimes a_{22}\\
+&C_1a_{11}\otimes e_4\otimes e_1+C_2a_{11}\otimes e_4\otimes e_2+C_3a_{11}\otimes e_4\otimes e_3+C_4a_{11}\otimes e_4\otimes e_4\\
+&D_{1111}a_{11}\otimes e_4\otimes a_{11}+D_{1122}a_{11}\otimes e_4\otimes a_{22}\\
+&\dfrac{1}{2}(A_{11}+A_{41})a_{11}\otimes a_{11}\otimes e_1+\dfrac{1}{2}(A_{12}+A_{42})a_{11}\otimes a_{11}\otimes e_2\\
+&\dfrac{1}{2}(A_{13}+A_{43})a_{11}\otimes a_{11}\otimes e_3+\dfrac{1}{2}(A_{14}+A_{44})a_{11}\otimes a_{11}\otimes e_4\\
+&\dfrac{1}{2}(B_1+B_4)a_{11}\otimes a_{11}\otimes a_{11}+\dfrac{1}{2}(B_1+B_4)a_{11}\otimes a_{11}\otimes a_{22}\\
+&\dfrac{1}{2}(A_{21}+A_{31})a_{11}\otimes a_{22}\otimes e_1+\dfrac{1}{2}(A_{22}+A_{32})a_{11}\otimes a_{22}\otimes e_2\\
+&\dfrac{1}{2}(A_{23}+A_{33})a_{11}\otimes a_{22}\otimes e_3+\dfrac{1}{2}(A_{24}+A_{34})a_{11}\otimes a_{22}\otimes e_4\\
+&\dfrac{1}{2}(B_2+B_3)a_{11}\otimes a_{22}\otimes a_{11}+\dfrac{1}{2}(-B_2-B_3)a_{11}\otimes a_{22}\otimes a_{22}\\
+&D_{1212}a_{12}\otimes e_1\otimes a_{12}+D_{1221}a_{12}\otimes e_1\otimes a_{21}\\
+&\sqrt{-1}D_{1221}a_{12}\otimes e_2\otimes a_{12}+\sqrt{-1}D_{1212}a_{12}\otimes e_2\otimes a_{21}\\
-&\sqrt{-1}D_{1221}a_{12}\otimes e_3\otimes a_{12}-\sqrt{-1}D_{1212}a_{12}\otimes e_3\otimes a_{21}\\
-&D_{1212}a_{12}\otimes e_4\otimes a_{12}-D_{1221}a_{12}\otimes e_4\otimes a_{21}\\
+&\dfrac{1}{2}(A_{11}-A_{41})a_{12}\otimes a_{12}\otimes e_1+\dfrac{1}{2}(A_{12}-A_{42})a_{12}\otimes a_{12}\otimes e_2\\
+&\dfrac{1}{2}(A_{13}-A_{43})a_{12}\otimes a_{12}\otimes e_3+\dfrac{1}{2}(A_{14}-A_{44})a_{12}\otimes a_{12}\otimes e_4\\
+&\dfrac{1}{2}(B_1-B_4)a_{12}\otimes a_{12}\otimes a_{11}+\dfrac{1}{2}(B_1-B_4)a_{12}\otimes a_{12}\otimes a_{22}\\
+&\dfrac{1}{2}\sqrt{-1}(-A_{21}+A_{31})a_{12}\otimes a_{21}\otimes e_1+\dfrac{1}{2}\sqrt{-1}(-A_{22}+A_{32})a_{12}\otimes a_{21}\otimes e_2\\
+&\dfrac{1}{2}\sqrt{-1}(-A_{23}+A_{33})a_{12}\otimes a_{21}\otimes e_3+\dfrac{1}{2}\sqrt{-1}(-A_{24}+A_{34})a_{12}\otimes a_{21}\otimes e_4
\end{align*}
\begin{align*}
+&\dfrac{1}{2}\sqrt{-1}(-B_2+B_3)a_{12}\otimes a_{21}\otimes a_{11}+\dfrac{1}{2}\sqrt{-1}(B_2-B_3)a_{12}\otimes a_{21}\otimes a_{22}\\
+&D_{1221}a_{21}\otimes e_1\otimes a_{12}+D_{1212}a_{21}\otimes e_1\otimes a_{21}\\
-&\sqrt{-1}D_{1212}a_{21}\otimes e_2\otimes a_{12}-\sqrt{-1}D_{1221}a_{21}\otimes e_2\otimes a_{21}\\
+&\sqrt{-1}D_{1212}a_{21}\otimes e_3\otimes a_{12}+\sqrt{-1}D_{1221}a_{21}\otimes e_3\otimes a_{21}\\
-&D_{1221}a_{21}\otimes e_4\otimes a_{12}-D_{1212}a_{21}\otimes e_4\otimes a_{21}\\
+&\dfrac{1}{2}\sqrt{-1}(A_{21}-A_{31})a_{21}\otimes a_{12}\otimes e_1+\dfrac{1}{2}\sqrt{-1}(A_{22}-A_{32})a_{21}\otimes a_{12}\otimes e_2\\
+&\dfrac{1}{2}\sqrt{-1}(A_{23}-A_{33})a_{21}\otimes a_{12}\otimes e_3+\dfrac{1}{2}\sqrt{-1}(A_{24}-A_{34})a_{21}\otimes a_{12}\otimes e_4\\
+&\dfrac{1}{2}\sqrt{-1}(B_2-B_3)a_{21}\otimes a_{12}\otimes a_{11}+\dfrac{1}{2}\sqrt{-1}(-B_2+B_3)a_{21}\otimes a_{12}\otimes a_{22}\\
+&\dfrac{1}{2}(A_{11}-A_{41})a_{21}\otimes a_{21}\otimes e_1+\dfrac{1}{2}(A_{12}-A_{42})a_{21}\otimes a_{21}\otimes e_2\\
+&\dfrac{1}{2}(A_{13}-A_{43})a_{21}\otimes a_{21}\otimes e_3+\dfrac{1}{2}(A_{14}-A_{44})a_{21}\otimes a_{21}\otimes e_4\\
+&\dfrac{1}{2}(B_1-B_4)a_{21}\otimes a_{21}\otimes a_{11}+\dfrac{1}{2}(B_1-B_4)a_{21}\otimes a_{21}\otimes a_{22}\\
+&C_1a_{22}\otimes e_1\otimes e_1-C_2a_{22}\otimes e_1\otimes e_2-C_3a_{22}\otimes e_1\otimes e_3+C_4a_{22}\otimes e_1\otimes e_4\\
+&D_{1111}a_{22}\otimes e_1\otimes a_{11}+D_{1122}a_{22}\otimes e_1\otimes a_{22}\\
+&C_1a_{22}\otimes e_2\otimes e_1+C_2a_{22}\otimes e_2\otimes e_2+C_3a_{22}\otimes e_2\otimes e_3+C_4a_{22}\otimes e_2\otimes e_4\\
+&D_{1111}a_{22}\otimes e_2\otimes a_{11}+D_{1122}a_{22}\otimes e_2\otimes a_{22}\\
+&C_1a_{22}\otimes e_3\otimes e_1+C_2a_{22}\otimes e_3\otimes e_2+C_3a_{22}\otimes e_3\otimes e_3+C_4a_{22}\otimes e_3\otimes e_4\\
+&D_{1111}a_{22}\otimes e_3\otimes a_{11}+D_{1122}a_{22}\otimes e_3\otimes a_{22}\\
+&C_1a_{22}\otimes e_4\otimes e_1-C_2a_{22}\otimes e_4\otimes e_2-C_3a_{22}\otimes e_4\otimes e_3+C_4a_{22}\otimes e_4\otimes e_4\\
-&D_{1122}a_{22}\otimes e_4\otimes a_{11}+D_{1111}a_{22}\otimes e_4\otimes a_{22}\\
+&\dfrac{1}{2}(A_{21}+A_{31})a_{22}\otimes a_{11}\otimes e_1+\dfrac{1}{2}(A_{22}+A_{32})a_{22}\otimes a_{11}\otimes e_2\\
+&\dfrac{1}{2}(A_{23}+A_{33})a_{22}\otimes a_{11}\otimes e_3+\dfrac{1}{2}(A_{24}+A_{34})a_{22}\otimes a_{11}\otimes e_4\\
+&\dfrac{1}{2}(B_2+B_3)a_{22}\otimes a_{11}\otimes a_{11}+\dfrac{1}{2}(-B_2-B_3)a_{22}\otimes a_{11}\otimes a_{22}\\
+&\dfrac{1}{2}(A_{11}+A_{41})a_{22}\otimes a_{22}\otimes e_1+\dfrac{1}{2}(A_{12}+A_{42})a_{22}\otimes a_{22}\otimes e_2\\
+&\dfrac{1}{2}(A_{13}+A_{43})a_{22}\otimes a_{22}\otimes e_3+\dfrac{1}{2}(A_{14}+A_{44})a_{22}\otimes a_{22}\otimes e_4\\
+&\dfrac{1}{2}(B_1+B_4)a_{22}\otimes a_{22}\otimes a_{11}+\dfrac{1}{2}(B_1+B_4)a_{22}\otimes a_{22}\otimes a_{22}\\
\end{align*}
We compute $R_{13}R_{23}$.
\begin{align*}
&R_{13}R_{23}\\
=&A_{11}A_{11}e_1\otimes e_1\otimes e_1 +A_{12}A_{12}e_1\otimes e_1\otimes e_2 +A_{13}A_{13}e_1\otimes e_1\otimes e_3\\+&A_{14}A_{14}e_1\otimes e_1\otimes e_4+B_1^2e_1\otimes e_1\otimes a_{11}+B_1^2e_1\otimes e_1\otimes a_{22}\\
+&A_{11}A_{21}e_1\otimes e_2\otimes e_1+A_{12}A_{22}e_1\otimes e_2\otimes e_2+A_{13}A_{23}e_1\otimes e_2\otimes e_3\\
+&A_{14}A_{24}e_1\otimes e_2\otimes e_4+B_1B_2e_1\otimes e_2\otimes a_{11}-B_1B_2e_1\otimes e_2\otimes a_{22}\\
+&A_{11}A_{31}e_1\otimes e_3\otimes e_1+A_{12}A_{32}e_1\otimes e_3\otimes e_2+A_{13}A_{33}e_1\otimes e_3\otimes e_3\\
+&A_{14}A_{34}e_1\otimes e_3\otimes e_4+B_1B_3e_1\otimes e_3\otimes a_{11}-B_1B_3e_1\otimes e_3\otimes a_{22}\\
+&A_{11}A_{41}e_1\otimes e_4\otimes e_1+A_{12}A_{42}e_1\otimes e_4\otimes e_2+A_{13}A_{43}e_1\otimes e_4\otimes e_3\\
+&A_{14}A_{44}e_1\otimes e_4\otimes e_4+B_1B_4e_1\otimes e_4\otimes a_{11}+B_1B_4e_1\otimes e_4\otimes a_{22}\\
+&A_{11}C_1e_1\otimes a_{11}\otimes e_1+A_{12}C_2e_1\otimes a_{11}\otimes e_2+A_{13}C_3e_1\otimes a_{11}\otimes e_3\\
+&A_{14}C_4e_1\otimes a_{11}\otimes e_4+B_1D_{1111}e_1\otimes a_{11}\otimes a_{11}+B_1D_{1122}e_1\otimes a_{11}\otimes a_{22}\\
+&B_1D_{1212}e_1\otimes a_{12}\otimes a_{12}+B_1D_{1221}e_1\otimes a_{12}\otimes a_{21}\\
+&B_1D_{1221}e_1\otimes a_{21}\otimes a_{12}+B_1D_{1212}e_1\otimes a_{21}\otimes a_{21}\\
+&A_{11}C_1e_1\otimes a_{22}\otimes e_1-A_{12}C_2e_1\otimes a_{22}\otimes e_2-A_{13}C_3e_1\otimes a_{22}\otimes e_3\\
+&A_{14}C_4e_1\otimes a_{22}\otimes e_4-B_1D_{1122}e_1\otimes a_{22}\otimes a_{11}+B_1D_{1111}e_1\otimes a_{22}\otimes a_{22}\\
+&A_{21}A_{11}e_2\otimes e_1\otimes e_1+A_{22}A_{12}e_2\otimes e_1\otimes e_2+A_{23}A_{13}e_2\otimes e_1\otimes e_3\\
+&A_{24}A_{14}e_2\otimes e_1\otimes e_4+B_2B_1e_2\otimes e_1\otimes a_{11}-B_2B_1e_2\otimes e_1\otimes a_{22}\\
+&A_{21}A_{21}e_2\otimes e_2\otimes e_1+A_{22}A_{22}e_2\otimes e_2\otimes e_2+A_{23}A_{23}e_2\otimes e_2\otimes e_3\\
+&A_{24}A_{24}e_2\otimes e_2\otimes e_4+B_2^2e_2\otimes e_2\otimes a_{11}+B_2^2e_2\otimes e_2\otimes a_{22}\\
+&A_{21}A_{31}e_2\otimes e_3\otimes e_1+A_{22}A_{32}e_2\otimes e_3\otimes e_2+A_{23}A_{33}e_2\otimes e_3\otimes e_3\\
+&A_{24}A_{34}e_2\otimes e_3\otimes e_4+B_2B_3e_2\otimes e_3\otimes a_{11}+B_2B_3e_2\otimes e_3\otimes a_{22}\\
+&A_{21}A_{41}e_2\otimes e_4\otimes e_1+A_{22}A_{42}e_2\otimes e_4\otimes e_2+A_{23}A_{43}e_2\otimes e_4\otimes e_3\\
+&A_{24}A_{44}e_2\otimes e_4\otimes e_4+B_2B_4e_2\otimes e_4\otimes a_{11}-B_2B_4e_2\otimes e_4\otimes a_{22}\\
+&A_{21}C_1e_2\otimes a_{11}\otimes e_1+A_{22}C_2e_2\otimes a_{11}\otimes e_2+A_{23}C_3e_2\otimes a_{11}\otimes e_3\\
+&A_{24}C_4e_2\otimes a_{11}\otimes e_4+B_2D_{1111}e_2\otimes a_{11}\otimes a_{11}-B_2D_{1122}e_2\otimes a_{11}\otimes a_{22}\\
+&B_2D_{1212}e_2\otimes a_{12}\otimes a_{12}-B_2D_{1221}e_2\otimes a_{12}\otimes a_{21}\\
+&B_2D_{1221}e_2\otimes a_{21}\otimes a_{12}-B_2D_{1212}e_2\otimes a_{21}\otimes a_{21}\\
+&A_{21}C_1e_2\otimes a_{22}\otimes e_1-A_{22}C_2e_2\otimes a_{22}\otimes e_2-A_{23}C_3e_2\otimes a_{22}\otimes e_3\\
+&A_{24}C_4e_2\otimes a_{22}\otimes e_4-B_2D_{1122}e_2\otimes a_{22}\otimes a_{11}-B_2D_{1111}e_2\otimes a_{22}\otimes a_{22}\\
+&A_{31}A_{11}e_3\otimes e_1\otimes e_1+A_{32}A_{12}e_3\otimes e_1\otimes e_2+A_{33}A_{13}e_3\otimes e_1\otimes e_3\\
+&A_{34}A_{14}e_3\otimes e_1\otimes e_4+B_3B_1e_3\otimes e_1\otimes a_{11}-B_3B_1e_3\otimes e_1\otimes a_{22}\\
+&A_{31}A_{21}e_3\otimes e_2\otimes e_1+A_{32}A_{22}e_3\otimes e_2\otimes e_2+A_{33}A_{23}e_3\otimes e_2\otimes e_3\\
+&A_{34}A_{24}e_3\otimes e_2\otimes e_4+B_3B_2e_3\otimes e_2\otimes a_{11}+B_3B_2e_3\otimes e_2\otimes a_{22}\\
+&A_{31}A_{31}e_3\otimes e_3\otimes e_1+A_{32}A_{32}e_3\otimes e_3\otimes e_2+A_{33}A_{33}e_3\otimes e_3\otimes e_3\\
+&A_{34}A_{34}e_3\otimes e_3\otimes e_4+B_3^2e_3\otimes e_3\otimes a_{11}+B_3^2e_3\otimes e_3\otimes a_{22}
\end{align*}
\begin{align*}
+&A_{31}A_{41}e_3\otimes e_4\otimes e_1+A_{32}A_{42}e_3\otimes e_4\otimes e_2+A_{33}A_{43}e_3\otimes e_4\otimes e_3\\
+&A_{34}A_{44}e_3\otimes e_4\otimes e_4+B_3B_4e_3\otimes e_4\otimes a_{11}-B_3B_4e_3\otimes e_4\otimes a_{22}\\
+&A_{31}C_1e_3\otimes a_{11}\otimes e_1+A_{32}C_2e_3\otimes a_{11}\otimes e_2+A_{33}C_3e_3\otimes a_{11}\otimes e_3\\
+&A_{34}C_4e_3\otimes a_{11}\otimes e_4+B_3D_{1111}e_3\otimes a_{11}\otimes a_{11}-B_3D_{1122}e_3\otimes a_{11}\otimes a_{22}\\
+&B_3D_{1212}e_3\otimes a_{12}\otimes a_{12}-B_3D_{1221}e_3\otimes a_{12}\otimes a_{21}\\
+&B_3D_{1221}e_3\otimes a_{21}\otimes a_{12}-B_3D_{1212}e_3\otimes a_{21}\otimes a_{21}\\
+&A_{31}C_1e_3\otimes a_{22}\otimes e_1-A_{32}C_2e_3\otimes a_{22}\otimes e_2-A_{33}C_3e_3\otimes a_{22}\otimes e_3\\
+&A_{34}C_4e_3\otimes a_{22}\otimes e_4-B_3D_{1122}e_3\otimes a_{22}\otimes a_{11}-B_3D_{1111}e_3\otimes a_{22}\otimes a_{22}\\
+&A_{41}A_{11}e_4\otimes e_1\otimes e_1+A_{42}A_{12}e_4\otimes e_1\otimes e_2+A_{43}A_{13}e_4\otimes e_1\otimes e_3\\
+&A_{44}A_{14}e_4\otimes e_1\otimes e_4+B_4B_1e_4\otimes e_1\otimes a_{11}+B_4B_1e_4\otimes e_1\otimes a_{22}\\
+&A_{41}A_{21}e_4\otimes e_2\otimes e_1+A_{42}A_{22}e_4\otimes e_2\otimes e_2+A_{43}A_{23}e_4\otimes e_2\otimes e_3\\
+&A_{44}A_{24}e_4\otimes e_2\otimes e_4+B_4B_2e_4\otimes e_2\otimes a_{11}-B_4B_2e_4\otimes e_2\otimes a_{22}\\
+&A_{41}A_{31}e_4\otimes e_3\otimes e_1+A_{42}A_{32}e_4\otimes e_3\otimes e_2+A_{43}A_{33}e_4\otimes e_3\otimes e_3\\
+&A_{44}A_{34}e_4\otimes e_3\otimes e_4+B_4B_3e_4\otimes e_3\otimes a_{11}-B_4B_3e_4\otimes e_3\otimes a_{22}\\
+&A_{41}A_{41}e_4\otimes e_4\otimes e_1+A_{42}A_{42}e_4\otimes e_4\otimes e_2+A_{43}A_{43}e_4\otimes e_4\otimes e_3\\
+&A_{44}A_{44}e_4\otimes e_4\otimes e_4+B_4^2e_4\otimes e_4\otimes a_{11}+B_4^2e_4\otimes e_4\otimes a_{22}\\
+&A_{41}C_1e_4\otimes a_{11}\otimes e_1+A_{42}C_2e_4\otimes a_{11}\otimes e_2+A_{43}C_3e_4\otimes a_{11}\otimes e_3\\
+&A_{44}C_4e_4\otimes a_{11}\otimes e_4+B_4D_{1111}e_4\otimes a_{11}\otimes a_{11}+B_4D_{1122}e_4\otimes a_{11}\otimes a_{22}\\
+&B_4D_{1212}e_4\otimes a_{12}\otimes a_{12}+B_4D_{1221}e_4\otimes a_{12}\otimes a_{21}\\
+&B_4D_{1221}e_4\otimes a_{21}\otimes a_{12}+B_4D_{1212}e_4\otimes a_{21}\otimes a_{21}\\
+&A_{41}C_1e_4\otimes a_{22}\otimes e_1-A_{42}C_2e_4\otimes a_{22}\otimes e_2-A_{43}C_3e_4\otimes a_{22}\otimes e_3\\
+&A_{44}C_4e_4\otimes a_{22}\otimes e_4-B_4D_{1122}e_4\otimes a_{22}\otimes a_{11}+B_4D_{1111}e_4\otimes a_{22}\otimes a_{22}\\
+&C_1A_{11}a_{11}\otimes e_1\otimes e_1+C_2A_{12}a_{11}\otimes e_1\otimes e_2+C_3A_{13}a_{11}\otimes e_1\otimes e_3\\
+&C_4A_{14}a_{11}\otimes e_1\otimes e_4+D_{1111}B_1a_{11}\otimes e_1\otimes a_{11}+D_{1122}B_1a_{11}\otimes e_1\otimes a_{22}\\
+&C_1A_{21}a_{11}\otimes e_2\otimes e_1+C_2A_{22}a_{11}\otimes e_2\otimes e_2+C_3A_{23}a_{11}\otimes e_2\otimes e_3\\
+&C_4A_{24}a_{11}\otimes e_2\otimes e_4+D_{1111}B_2a_{11}\otimes e_2\otimes a_{11}-D_{1122}B_2a_{11}\otimes e_2\otimes a_{22}\\
+&C_1A_{31}a_{11}\otimes e_3\otimes e_1+C_2A_{32}a_{11}\otimes e_3\otimes e_2+C_3A_{33}a_{11}\otimes e_3\otimes e_3\\
+&C_4A_{34}a_{11}\otimes e_3\otimes e_4+D_{1111}B_3a_{11}\otimes e_3\otimes a_{11}-D_{1122}B_3a_{11}\otimes e_3\otimes a_{22}\\
+&C_1A_{41}a_{11}\otimes e_4\otimes e_1+C_2A_{42}a_{11}\otimes e_4\otimes e_2+C_3A_{43}a_{11}\otimes e_4\otimes e_3\\
+&C_4A_{44}a_{11}\otimes e_4\otimes e_4+D_{1111}B_4a_{11}\otimes e_4\otimes a_{11}+D_{1122}B_4a_{11}\otimes e_4\otimes a_{22}\\
+&C_1^2a_{11}\otimes a_{11}\otimes e_1+C_2^2a_{11}\otimes a_{11}\otimes e_2+C_3^2a_{11}\otimes a_{11}\otimes e_3\\
+&C_4^2a_{11}\otimes a_{11}\otimes e_4+D_{1111}^2a_{11}\otimes a_{11}\otimes a_{11}+D_{1122}^2a_{11}\otimes a_{11}\otimes a_{22}\\
+&D_{1111}D_{1212}a_{11}\otimes a_{12}\otimes a_{12}+D_{1122}D_{1221}a_{11}\otimes a_{12}\otimes a_{21}\\
+&D_{1111}D_{1221}a_{11}\otimes a_{21}\otimes a_{12}+D_{1122}D_{1212}a_{11}\otimes a_{21}\otimes a_{21}\\
+&C_1^2a_{11}\otimes a_{22}\otimes e_1-C_2^2a_{11}\otimes a_{22}\otimes e_2-C_3^2a_{11}\otimes a_{22}\otimes e_3\\
+&C_4^2a_{11}\otimes a_{22}\otimes e_4-D_{1111}D_{1122}a_{11}\otimes a_{22}\otimes a_{11}+D_{1122}D_{1111}a_{11}\otimes a_{22}\otimes a_{22}
\end{align*}
\begin{align*}
+&D_{1212}B_1a_{12}\otimes e_1\otimes a_{12}+D_{1221}B_1a_{12}\otimes e_1\otimes a_{21}\\
-&D_{1212}B_2a_{12}\otimes e_2\otimes a_{12}+D_{1221}B_2a_{12}\otimes e_2\otimes a_{21}\\
-&D_{1212}B_3a_{12}\otimes e_3\otimes a_{12}+D_{1221}B_3a_{12}\otimes e_3\otimes a_{21}\\
+&D_{1212}B_4a_{12}\otimes e_4\otimes a_{12}+D_{1221}B_4a_{12}\otimes e_4\otimes a_{21}\\
+&D_{1212}D_{1122}a_{12}\otimes a_{11}\otimes a_{12}+D_{1221}D_{1111}a_{12}\otimes a_{11}\otimes a_{21}\\
+&D_{1212}D_{1221}a_{12}\otimes a_{12}\otimes a_{11}+D_{1221}D_{1212}a_{12}\otimes a_{12}\otimes a_{22}\\
+&D_{1212}D_{1212}a_{12}\otimes a_{21}\otimes a_{11}+D_{1221}D_{1221}a_{12}\otimes a_{21}\otimes a_{22}\\
+&D_{1212}D_{1111}a_{12}\otimes a_{22}\otimes a_{12}-D_{1221}D_{1122}a_{12}\otimes a_{22}\otimes a_{21}\\
+&D_{1221}B_1a_{21}\otimes e_1\otimes a_{12}+D_{1212}B_1a_{21}\otimes e_1\otimes a_{21}\\
-&D_{1221}B_2a_{21}\otimes e_2\otimes a_{12}+D_{1212}B_2a_{21}\otimes e_2\otimes a_{21}\\
-&D_{1221}B_3a_{21}\otimes e_3\otimes a_{12}+D_{1212}B_3a_{21}\otimes e_3\otimes a_{21}\\
+&D_{1221}B_4a_{21}\otimes e_4\otimes a_{12}+D_{1212}B_4a_{21}\otimes e_4\otimes a_{21}\\
+&D_{1221}D_{1122}a_{21}\otimes a_{11}\otimes a_{12}+D_{1212}D_{1111}a_{21}\otimes a_{11}\otimes a_{21}\\
+&D_{1221}D_{1221}a_{21}\otimes a_{12}\otimes a_{11}+D_{1212}D_{1212}a_{21}\otimes a_{12}\otimes a_{22}\\
+&D_{1221}D_{1212}a_{21}\otimes a_{21}\otimes a_{11}+D_{1212}D_{1221}a_{21}\otimes a_{21}\otimes a_{22}\\
+&D_{1221}D_{1111}a_{21}\otimes a_{22}\otimes a_{12}-D_{1212}D_{1122}a_{21}\otimes a_{22}\otimes a_{21}\\
+&C_1A_{11}a_{22}\otimes e_1\otimes e_1-C_2A_{12}a_{22}\otimes e_1\otimes e_2-C_3A_{13}a_{22}\otimes e_1\otimes e_3\\
+&C_4A_{14}a_{22}\otimes e_1\otimes e_4-D_{1122}B_1a_{22}\otimes e_1\otimes a_{11}+D_{1111}B_1a_{22}\otimes e_1\otimes a_{22}\\
+&C_1A_{21}a_{22}\otimes e_2\otimes e_1-C_2A_{22}a_{22}\otimes e_2\otimes e_2-C_3A_{23}a_{22}\otimes e_2\otimes e_3\\
+&C_4A_{24}a_{22}\otimes e_2\otimes e_4-D_{1122}B_2a_{22}\otimes e_2\otimes a_{11}-D_{1111}B_2a_{22}\otimes e_2\otimes a_{22}\\
+&C_1A_{31}a_{22}\otimes e_3\otimes e_1-C_2A_{32}a_{22}\otimes e_3\otimes e_2-C_3A_{33}a_{22}\otimes e_3\otimes e_3\\
+&C_4A_{34}a_{22}\otimes e_3\otimes e_4-D_{1122}B_3a_{22}\otimes e_3\otimes a_{11}-D_{1111}B_3a_{22}\otimes e_3\otimes a_{22}\\
+&C_1A_{41}a_{22}\otimes e_4\otimes e_1-C_2A_{42}a_{22}\otimes e_4\otimes e_2-C_3A_{43}a_{22}\otimes e_4\otimes e_3\\
+&C_4A_{44}a_{22}\otimes e_4\otimes e_4-D_{1122}B_4a_{22}\otimes e_4\otimes a_{11}+D_{1111}B_4a_{22}\otimes e_4\otimes a_{22}\\
+&C_1^2a_{22}\otimes a_{11}\otimes e_1-C_2^2a_{22}\otimes a_{11}\otimes e_2-C_3^2a_{22}\otimes a_{11}\otimes e_3\\
+&C_4^2a_{22}\otimes a_{11}\otimes e_4-D_{1122}D_{1111}a_{22}\otimes a_{11}\otimes a_{11}+D_{1111}D_{1122}a_{22}\otimes a_{11}\otimes a_{22}\\
-&D_{1122}D_{1212}a_{22}\otimes a_{12}\otimes a_{12}+D_{1111}D_{1221}a_{22}\otimes a_{12}\otimes a_{21}\\
-&D_{1122}D_{1221}a_{22}\otimes a_{21}\otimes a_{12}+D_{1111}D_{1212}a_{22}\otimes a_{21}\otimes a_{21}\\
+&C_1^2a_{22}\otimes a_{22}\otimes e_1+C_2^2a_{22}\otimes a_{22}\otimes e_2+C_3^2a_{22}\otimes a_{22}\otimes e_3\\
+&C_4^2a_{22}\otimes a_{22}\otimes e_4+D_{1122}D_{1122}a_{22}\otimes a_{22}\otimes a_{11}+D_{1111}D_{1111}a_{22}\otimes a_{22}\otimes a_{22}.
\end{align*}
In Case $1$, by comparing the coefficients of $e_3\otimes a_{11}\otimes a_{22}$ we have $-B_3D_{1122}=D_{1111}$ and thus $B_3=-1$. By comparing the coefficients of $e_4\otimes a_{11}\otimes a_{22}$, we have $B_4D_{1122}=D_{1122}$ and thus $B_4=1$. From these we see that $B_2=-1$. Hence \[\begin{bmatrix}
   B_1 \\
   B_2 \\
   B_3 \\
   B_4
\end{bmatrix}=\begin{bmatrix}
   1 \\
   -1 \\
   -1 \\
   1
\end{bmatrix}_, \begin{bmatrix}
   C_1 \\
   C_2 \\
   C_3 \\
   C_4
\end{bmatrix}=\begin{bmatrix}
   1 \\
   1 \\
   1 \\
   1
\end{bmatrix}_, \begin{bmatrix}
   D_{1111} \\
   D_{1122} \\
   D_{1212} \\
   D_{1221}
\end{bmatrix}=\begin{bmatrix}
  \pm1 \\
  \pm1 \\
   0 \\
   0
\end{bmatrix}_.\]
In Case $2$, by comparng the coefficients of $a_{11}\otimes a_{11}\otimes a_{11}$, we have $D_{1111}^2=\dfrac{1}{2}(B_1+B_4)$ and thus $B_4=-1$. By comparing the coefficients of $a_{12}\otimes a_{12}\otimes a_{22}$, we have $D_{1221}D_{1212}=\dfrac{1}{2}(B_1-B_4)$ and thus $\lambda^2=\sqrt{-1}$. By comparing the coefficients of $e_2\otimes a_{21}\otimes a_{21}$, we have $-B_2D_{1212}=\sqrt{-1}D_{1221}$ and thus $B_2=-1$. Hence\[
\begin{bmatrix}
   B_1 \\
   B_2 \\
   B_3 \\
   B_4
\end{bmatrix}=\begin{bmatrix}
   1 \\
   -1 \\
   1 \\
   -1
\end{bmatrix}_, \begin{bmatrix}
   C_1 \\
   C_2 \\
   C_3 \\
   C_4
\end{bmatrix}=\begin{bmatrix}
   1 \\
   1 \\
   -1 \\
   -1
\end{bmatrix}_, \begin{bmatrix}
   D_{1111} \\
   D_{1122} \\
   D_{1212} \\
   D_{1221}
\end{bmatrix}=\begin{bmatrix}
  0 \\
  0 \\
   \lambda \\
   -\sqrt{-1}\lambda
\end{bmatrix}_.
\]
In Case $3$, by comparing the coefficients of $a_{11}\otimes a_{11}\otimes a_{11}$ we have $D_{1111}^2=\dfrac{1}{2}(B_1+B_4)$ and thus $B_4=-1$. By comparing the coefficients of $a_{12}\otimes a_{12}\otimes a_{22}$, we have $D_{1221}D_{1212}=\dfrac{1}{2}(B_1-B_4)$ and thus $\lambda^2=-\sqrt{-1}$. Hence\[
\begin{bmatrix}
   B_1 \\
   B_2 \\
   B_3 \\
   B_4
\end{bmatrix}=\begin{bmatrix}
   1 \\
   1 \\
   -1 \\
   -1
\end{bmatrix}_, \begin{bmatrix}
   C_1 \\
   C_2 \\
   C_3 \\
   C_4
\end{bmatrix}=\begin{bmatrix}
   1 \\
   -1 \\
   1 \\
   -1
\end{bmatrix}_, \begin{bmatrix}
   D_{1111} \\
   D_{1122} \\
   D_{1212} \\
   D_{1221}
\end{bmatrix}=\begin{bmatrix}
  0 \\
  0 \\
   \lambda \\
   \sqrt{-1}\lambda
\end{bmatrix}_.\]
In Case $4$, by comparing the coefficients of $e_3\otimes a_{11}\otimes a_{22}$ we have $-B_3D_{1122}=D_{1111}$ and thus $B_3=1$. By comparing the coefficients of $e_4\otimes a_{11}\otimes a_{22}$, we have $B_4D_{1122}=D_{1122}$ and thus $B_4=1$. Hence\[
\begin{bmatrix}
   B_1 \\
   B_2 \\
   B_3 \\
   B_4
\end{bmatrix}=\begin{bmatrix}
   1 \\
   1 \\
   1 \\
   1
\end{bmatrix}_, \begin{bmatrix}
   C_1 \\
   C_2 \\
   C_3 \\
   C_4
\end{bmatrix}=\begin{bmatrix}
   1 \\
   -1 \\
   -1 \\
   1
\end{bmatrix}_, \begin{bmatrix}
   D_{1111} \\
   D_{1122} \\
   D_{1212} \\
   D_{1221}
\end{bmatrix}=\begin{bmatrix}
  \pm1 \\
  \mp1 \\
   0 \\
   0
\end{bmatrix}_.\]
In any cases, by comparing the coefficients of 
\begin{align*}
&a_{11}\otimes a_{22}\otimes e_2, a_{11}\otimes a_{22}\otimes e_3, a_{11}\otimes a_{22}\otimes e_4,\\
&a_{12}\otimes a_{21}\otimes e_2, a_{12}\otimes a_{21}\otimes e_3, a_{12}\otimes a_{21}\otimes e_4,\\
&a_{22}\otimes a_{22}\otimes e_2, a_{22}\otimes a_{22}\otimes e_3, a_{22}\otimes a_{22}\otimes e_4
\end{align*}
we have
\begin{align*}
A_{22}+A_{32}&=-2C_2^2=-2, A_{23}+A_{33}=-2C_3^2=-2, A_{24}+A_{34}=2C_4^2,\\
-A_{22}+A_{32}&=0, -A_{23}+A_{33}=0, -A_{24}+A_{34}=0,\\
A_{12}+A_{42}&=2C_2^2=2, A_{13}+A_{43}=2C_3^2=2, A_{14}+A_{44}=2C_4^2=2.
\end{align*}
Thus we have
\begin{align*}
\begin{bmatrix}
   A_{11} & A_{12} & A_{13} & A_{14}\\
   A_{21} & A_{22} & A_{23} & A_{24}\\
   A_{31} & A_{32} & A_{33} & A_{34}\\
   A_{41} & A_{42} & A_{43} & A_{44}
\end{bmatrix}=\begin{bmatrix}
   1 & 1 & 1& 1\\
   1 & -1 & -1 & 1\\
   1 & -1 & -1 & 1\\
   1 & 1 & 1& 1
\end{bmatrix}
\end{align*}
\par
Conversely suppose $R$ satisfies one of these conditions. To show that $R$ is a universal $R$-matrix, it suffices to show that $R$ is invertible, $(\Delta\otimes\id)R=R_{13}R_{23}$ and $(\id\otimes\Delta)R=R_{13}R_{12}$. As for the invertibility, we have \begin{align*}
RR^*=&\sum_{i,j}|A_{ij}|^2e_i\otimes e_j\\
+&\sum_i|B_i|^2(e_i\otimes a_{11}+e_i\otimes a_{22})+\sum_{i}|C_i|^2(a_{11}\otimes e_i+a_{22}\otimes e_i)\\
+&(|D_{1111}|^2+|D_{1212}|^2)a_{11}\otimes a_{11}+(|D_{1122}|^2+|D_{1221}|^2)a_{11}\otimes a_{22}\\
+&(D_{1111}\compconj{D_{1212}}+D_{1212}\compconj{D_{1111}})a_{12}\otimes a_{12}\\
+&(D_{1122}\compconj{D_{1221}}-D_{1221}\compconj{D_{1122}})a_{12}\otimes a_{21}\\
+&(D_{1221}\compconj{D_{1122}}-D_{1122}\compconj{D_{1221}})a_{21}\otimes a_{12}\\
+&(D_{1212}\compconj{D_{1111}}+D_{1111}\compconj{D_{1212}})a_{21}\otimes a_{21}\\
+&(|D_{1122}|^2+|D_{1221}|^2)a_{22}\otimes a_{11}+(|D_{1111}|^2+|D_{1212}|^2)a_{22}\otimes a_{22}\\
=&1\otimes 1.
\end{align*} It is easy to check the other formulas.\par
Thus we have shown the following theorem.
\begin{thm}\label{quasi}(cf.\ \cite{Suz} and \cite{W})
\item The universal $R$-matrices of Kac--Paljutkin's finite quantum group $\mathcal{A}$ are of the form 
\begin{align*}
R=&A_{11}e_{1}\otimes e_{1}+A_{12}e_{1}\otimes e_{2}+A_{13}e_{1}\otimes e_{3}+A_{14}e_{1}\otimes e_{4}\\
+&A_{21}e_{2}\otimes e_{1}+A_{22}e_{2}\otimes e_{2}+A_{23}e_{2}\otimes e_{3}+A_{24}e_{2}\otimes e_{4}\\
+&A_{31}e_{3}\otimes e_{1}+A_{32}e_{3}\otimes e_{2}+A_{33}e_{3}\otimes e_{3}+A_{34}e_{3}\otimes e_{4}\\
+&A_{41}e_{4}\otimes e_{1}+A_{42}e_{4}\otimes e_{2}+A_{43}e_{4}\otimes e_{3}+A_{44}e_{4}\otimes e_{4}\\
+&B_1e_1\otimes a_{11}+B_1e_1\otimes a_{22}+B_2e_2\otimes a_{11}-B_2e_2\otimes a_{22}\\
+&B_3e_3\otimes a_{11}-B_3e_3\otimes a_{22}+B_4e_4\otimes a_{11}+B_4e_4\otimes a_{22}\\
+&C_1a_{11}\otimes e_1+C_1 a_{22}\otimes e_1+C_2a_{11}\otimes e_2-C_2a_{22}\otimes e_2\\
+&C_3a_{11}\otimes e_3-C_3a_{22}\otimes e_3+C_4a_{11}\otimes e_4+C_4a_{22}\otimes e_4\\
+&D_{1111}a_{11}\otimes a_{11}+D_{1122}a_{11}\otimes a_{22}+D_{1212}a_{12}\otimes a_{12}+D_{1221}a_{12}\otimes a_{21}\\+&D_{1221}a_{21}\otimes a_{12}+D_{1212}a_{21}\otimes a_{21}-D_{1122}a_{22}\otimes a_{11}+D_{1111}a_{22}\otimes a_{22},\end{align*} where 
\begin{align*}
\begin{bmatrix}
   A_{11} & A_{12} & A_{13} & A_{14}\\
   A_{21} & A_{22} & A_{23} & A_{24}\\
   A_{31} & A_{32} & A_{33} & A_{34}\\
   A_{41} & A_{42} & A_{43} & A_{44}
\end{bmatrix}=\begin{bmatrix}
   1 & 1 & 1& 1\\
   1 & -1 & -1 & 1\\
   1 & -1 & -1 & 1\\
   1 & 1 & 1& 1
\end{bmatrix}
\end{align*}
and other coefficients are given by one of the following eight cases:
\begin{itemize}
\item $\begin{bmatrix}
   B_1 \\
   B_2 \\
   B_3 \\
   B_4
\end{bmatrix}=\begin{bmatrix}
   1 \\
   -1 \\
   -1 \\
   1
\end{bmatrix}_, \begin{bmatrix}
   C_1 \\
   C_2 \\
   C_3 \\
   C_4
\end{bmatrix}=\begin{bmatrix}
   1 \\
   1 \\
   1 \\
   1
\end{bmatrix}_, \begin{bmatrix}
   D_{1111} \\
   D_{1122} \\
   D_{1212} \\
   D_{1221}
\end{bmatrix}=\begin{bmatrix}
  \pm1 \\
  \pm1 \\
   0 \\
   0
\end{bmatrix}$
\item $\begin{bmatrix}
   B_1 \\
   B_2 \\
   B_3 \\
   B_4
\end{bmatrix}=\begin{bmatrix}
   1 \\
   1 \\
   1 \\
   1
\end{bmatrix}_, \begin{bmatrix}
   C_1 \\
   C_2 \\
   C_3 \\
   C_4
\end{bmatrix}=\begin{bmatrix}
   1 \\
   -1 \\
   -1 \\
   1
\end{bmatrix}_, \begin{bmatrix}
   D_{1111} \\
   D_{1122} \\
   D_{1212} \\
   D_{1221}
\end{bmatrix}=\begin{bmatrix}
  \pm1 \\
  \mp1 \\
   0 \\
   0
\end{bmatrix}$
\item $\begin{bmatrix}
   B_1 \\
   B_2 \\
   B_3 \\
   B_4
\end{bmatrix}=\begin{bmatrix}
   1 \\
   -1 \\
   1 \\
   -1
\end{bmatrix}_, \begin{bmatrix}
   C_1 \\
   C_2 \\
   C_3 \\
   C_4
\end{bmatrix}=\begin{bmatrix}
   1 \\
   1 \\
   -1 \\
   -1
\end{bmatrix}_, \begin{bmatrix}
   D_{1111} \\
   D_{1122} \\
   D_{1212} \\
   D_{1221}
\end{bmatrix}=\begin{bmatrix}
  0 \\
  0 \\
   \lambda \\
   -\sqrt{-1}\lambda
\end{bmatrix}$ \\($\lambda$ is a square root of $\sqrt{-1}$.)
\item $\begin{bmatrix}
   B_1 \\
   B_2 \\
   B_3 \\
   B_4
\end{bmatrix}=\begin{bmatrix}
   1 \\
   1 \\
   -1 \\
   -1
\end{bmatrix}_, \begin{bmatrix}
   C_1 \\
   C_2 \\
   C_3 \\
   C_4
\end{bmatrix}=\begin{bmatrix}
   1 \\
   -1 \\
   1 \\
   -1
\end{bmatrix}_, \begin{bmatrix}
   D_{1111} \\
   D_{1122} \\
   D_{1212} \\
   D_{1221}
\end{bmatrix}=\begin{bmatrix}
  0 \\
  0 \\
   \lambda \\
   \sqrt{-1}\lambda
\end{bmatrix}$ \\($\lambda$ is a square root of $\sqrt{-1}$.)
\end{itemize}
\end{thm}
By the computations above we also have the following:
\begin{cor}
The $R$-matrices above are unitary.
\end{cor}
Let $A$ be a quasitriangular Hopf algebra and let $R$ be an $R$-matrix of $A$. We put $R_{(l)}=\{(\id\otimes\omega)R\mid\omega\in A^*\}$ and $R_{(r)}=\{(\omega\otimes\id)R\mid\omega\in A^*\}$. Let $A_R$ denote the sub-Hopf algebra of $A$ generated by $R_{(l)}+R_{(r)}$. If $A=A_R$, then $(A,R)$ (or $R$) is called \emph{minimal} (\cite{Rad}). Let us determine which $R$-matrices in Theorem \ref{quasi} are minimal. If our $R$-matrix satisfies one of the following
\begin{itemize}
\item $\begin{bmatrix}
   B_1 \\
   B_2 \\
   B_3 \\
   B_4
\end{bmatrix}=\begin{bmatrix}
   1 \\
   -1 \\
   -1 \\
   1
\end{bmatrix}_, \begin{bmatrix}
   C_1 \\
   C_2 \\
   C_3 \\
   C_4
\end{bmatrix}=\begin{bmatrix}
   1 \\
   1 \\
   1 \\
   1
\end{bmatrix}_, \begin{bmatrix}
   D_{1111} \\
   D_{1122} \\
   D_{1212} \\
   D_{1221}
\end{bmatrix}=\begin{bmatrix}
  \pm1 \\
  \pm1 \\
   0 \\
   0
\end{bmatrix}$
\item $\begin{bmatrix}
   B_1 \\
   B_2 \\
   B_3 \\
   B_4
\end{bmatrix}=\begin{bmatrix}
   1 \\
   1 \\
   1 \\
   1
\end{bmatrix}_, \begin{bmatrix}
   C_1 \\
   C_2 \\
   C_3 \\
   C_4
\end{bmatrix}=\begin{bmatrix}
   1 \\
   -1 \\
   -1 \\
   1
\end{bmatrix}_, \begin{bmatrix}
   D_{1111} \\
   D_{1122} \\
   D_{1212} \\
   D_{1221}
\end{bmatrix}=\begin{bmatrix}
  \pm1 \\
  \mp1 \\
   0 \\
   0
\end{bmatrix}_{,}$
\end{itemize}
then
\begin{align*}
R=&A_{11}e_{1}\otimes e_{1}+A_{12}e_{1}\otimes e_{2}+A_{13}e_{1}\otimes e_{3}+A_{14}e_{1}\otimes e_{4}\\
+&A_{21}e_{2}\otimes e_{1}+A_{22}e_{2}\otimes e_{2}+A_{23}e_{2}\otimes e_{3}+A_{24}e_{2}\otimes e_{4}\\
+&A_{31}e_{3}\otimes e_{1}+A_{32}e_{3}\otimes e_{2}+A_{33}e_{3}\otimes e_{3}+A_{34}e_{3}\otimes e_{4}\\
+&A_{41}e_{4}\otimes e_{1}+A_{42}e_{4}\otimes e_{2}+A_{43}e_{4}\otimes e_{3}+A_{44}e_{4}\otimes e_{4}\\
+&B_1e_1\otimes a_{11}+B_1e_1\otimes a_{22}+B_2e_2\otimes a_{11}-B_2e_2\otimes a_{22}\\
+&B_3e_3\otimes a_{11}-B_3e_3\otimes a_{22}+B_4e_4\otimes a_{11}+B_4e_4\otimes a_{22}\\
+&C_1a_{11}\otimes e_1+C_1 a_{22}\otimes e_1+C_2a_{11}\otimes e_2-C_2a_{22}\otimes e_2\\
+&C_3a_{11}\otimes e_3-C_3a_{22}\otimes e_3+C_4a_{11}\otimes e_4+C_4a_{22}\otimes e_4\\
+&D_{1111}a_{11}\otimes a_{11}+D_{1122}a_{11}\otimes a_{22}-D_{1122}a_{22}\otimes a_{11}+D_{1111}a_{22}\otimes a_{22}\end{align*}
and it is clear that $a_{12}, a_{21}\not\in A_R$. Thus such $R$-matrices are not minimal. If our $R$-matrix satisfies one of the following
\begin{itemize}
\item $\begin{bmatrix}
   B_1 \\
   B_2 \\
   B_3 \\
   B_4
\end{bmatrix}=\begin{bmatrix}
   1 \\
   -1 \\
   1 \\
   -1
\end{bmatrix}_, \begin{bmatrix}
   C_1 \\
   C_2 \\
   C_3 \\
   C_4
\end{bmatrix}=\begin{bmatrix}
   1 \\
   1 \\
   -1 \\
   -1
\end{bmatrix}_, \begin{bmatrix}
   D_{1111} \\
   D_{1122} \\
   D_{1212} \\
   D_{1221}
\end{bmatrix}=\begin{bmatrix}
  0 \\
  0 \\
   \lambda \\
   -\sqrt{-1}\lambda
\end{bmatrix}$ \\($\lambda$ is a square root of $\sqrt{-1}$.)
\item $\begin{bmatrix}
   B_1 \\
   B_2 \\
   B_3 \\
   B_4
\end{bmatrix}=\begin{bmatrix}
   1 \\
   1 \\
   -1 \\
   -1
\end{bmatrix}_, \begin{bmatrix}
   C_1 \\
   C_2 \\
   C_3 \\
   C_4
\end{bmatrix}=\begin{bmatrix}
   1 \\
   -1 \\
   1 \\
   -1
\end{bmatrix}_, \begin{bmatrix}
   D_{1111} \\
   D_{1122} \\
   D_{1212} \\
   D_{1221}
\end{bmatrix}=\begin{bmatrix}
  0 \\
  0 \\
   \lambda \\
   \sqrt{-1}\lambda
\end{bmatrix}$ \\($\lambda$ is a square root of $\sqrt{-1}$),
\end{itemize}
then
\begin{align*}
R=&A_{11}e_{1}\otimes e_{1}+A_{12}e_{1}\otimes e_{2}+A_{13}e_{1}\otimes e_{3}+A_{14}e_{1}\otimes e_{4}\\
+&A_{21}e_{2}\otimes e_{1}+A_{22}e_{2}\otimes e_{2}+A_{23}e_{2}\otimes e_{3}+A_{24}e_{2}\otimes e_{4}\\
+&A_{31}e_{3}\otimes e_{1}+A_{32}e_{3}\otimes e_{2}+A_{33}e_{3}\otimes e_{3}+A_{34}e_{3}\otimes e_{4}\\
+&A_{41}e_{4}\otimes e_{1}+A_{42}e_{4}\otimes e_{2}+A_{43}e_{4}\otimes e_{3}+A_{44}e_{4}\otimes e_{4}\\
+&B_1e_1\otimes a_{11}+B_1e_1\otimes a_{22}+B_2e_2\otimes a_{11}-B_2e_2\otimes a_{22}\\
+&B_3e_3\otimes a_{11}-B_3e_3\otimes a_{22}+B_4e_4\otimes a_{11}+B_4e_4\otimes a_{22}\\
+&C_1a_{11}\otimes e_1+C_1 a_{22}\otimes e_1+C_2a_{11}\otimes e_2-C_2a_{22}\otimes e_2\\
+&C_3a_{11}\otimes e_3-C_3a_{22}\otimes e_3+C_4a_{11}\otimes e_4+C_4a_{22}\otimes e_4\\
+&D_{1212}a_{12}\otimes a_{12}+D_{1221}a_{12}\otimes a_{21}+D_{1221}a_{21}\otimes a_{12}+D_{1212}a_{21}\otimes a_{21}\end{align*}
and by cosidering $(\id\otimes a_{12}^*)R$ and $(\id\otimes a_{21}^*)R$ we have
\begin{align*}
&D_{1212}a_{12}+D_{1221}a_{21}, D_{1221}a_{12}+D_{1212}a_{21}\in A_R.
\end{align*}
Hence we see that $a_{12}, a_{21}\in A_R$. Since $A_R$ is an algebra, we have $a_{11}, a_{22}\in A_R$. By considering $(e_{i}^*\otimes\id)R$, $(a_{ii}^*\otimes\id)R$ and the fact that $a_{ii}\in A_R$ ($i=1,2$) we have
\begin{align*}
&e_1+e_2+e_3+e_4\in A_R, e_1-e_2-e_3+e_4\in A_R,\\
&e_1-e_2+e_3-e_4\in A_R, e_1+e_2-e_3-e_4\in A_R.
\end{align*}
Thus $e_i\in A_R$ for all $i=1,2,3,4$. Therefore we have $A=A_R$. We have shown the following.
\begin{cor}(cf.~\cite{Suz,W})
The former four $R$-matrices in Theorem \ref{quasi} are not minimal. The latter four $R$-matrices in Theorem \ref{quasi} are minimal.
\end{cor} 
\end{document}